\documentclass[a4paper,11pt]{article}

\usepackage[T1]{fontenc}
\usepackage{lmodern}

\usepackage{dsfont}

\usepackage[latin1]{inputenc}
\usepackage{amsmath}
\usepackage{amsthm}
\usepackage{amssymb}
\usepackage{mathrsfs}
\usepackage{graphicx}
\usepackage[all]{xy}
\usepackage{hyperref}

\usepackage{stmaryrd}
\usepackage{caption}

\usepackage{abstract} 

\newtheorem{thm}{Theorem}[section]
\newtheorem{cor}[thm]{Corollary}
\newtheorem{claim}[thm]{Claim}
\newtheorem{fact}[thm]{Fact}

\newtheorem{lemma}[thm]{Lemma}
\newtheorem{prop}[thm]{Proposition}

\theoremstyle{definition}
\newtheorem{definition}[thm]{Definition}
\newtheorem{ex}[thm]{Example}
\newtheorem{remark}[thm]{Remark}
\newtheorem{question}[thm]{Question}

\def\rquotient#1#2{%
	\makeatletter
	\raise.3ex\hbox{$#1$}/\lower.3ex\hbox{$#2$}%
	\makeatother
}	

\makeatletter
\newcommand{\subjclass}[2][2010]{%
	\let\@oldtitle\@title%
	\gdef\@title{\@oldtitle\footnotetext{#1 \emph{Mathematics subject classification.} #2}}%
}
\newcommand{\keywords}[1]{%
	\let\@@oldtitle\@title%
	\gdef\@title{\@@oldtitle\footnotetext{\emph{Key words and phrases.} #1.}}%
}
\makeatother

\newcommand{\Address}{{
		\bigskip
		\small
		
		\textsc{Institut Montpellierain Alexander Grothendieck, 499-554 Rue du Truel, 34090 Montpellier, France.}\par\nopagebreak
		\textit{E-mail address}: \texttt{anthony.genevois@umontpellier.fr}
\medskip

		\textsc{Institut de Math\'ematiques de Jussieu-Paris Rive Gauche, Place Aur\'elie Nemours, 75013 Paris, France.}\par\nopagebreak
		\textit{E-mail address}: \texttt{romain.tessera@imj-prg.fr}
\medskip
		
}}

\title{Asymptotic geometry of lamplighters over one-ended groups}
\date{\today}
\author{Anthony Genevois and Romain Tessera}


\begin{document}

\maketitle

\begin{abstract}
\scriptsize This article is dedicated to the asymptotic geometry of wreath products $F\wr H := \left( \bigoplus_H F \right) \rtimes H$ where $F$ is a finite group and $H$ a one-ended finitely presented group. Our main result is a complete classification of these groups up to quasi-isometry. More precisely, given two finite groups $F_1,F_2$ and two finitely presented one-ended groups $H_1,H_2$, we show that $F_1 \wr H_1$ and $F_2 \wr H_2$ are quasi-isometric if and only if either (i) $H_1,H_2$ are non-amenable quasi-isometric groups and $|F_1|,|F_2|$ have the same prime divisors, or (ii) $H_1,H_2$ are amenable, $|F_1|=k^{n_1}$ and $|F_2|=k^{n_2}$ for some $k,n_1,n_2 \geq 1$, and there exists a quasi-$(n_2/n_1)$-to-one quasi-isometry $H_1 \to H_2$. The article also contains algebraic information on groups quasi-isometric to such wreath products. This can be seen as far reaching extension of a celebrated work of Eskin-Fisher-Whyte who treated the case of  $H=\mathbb{Z}$. Our approach is however fundamentally different, as it crucially exploits the assumption that $H$ is one-ended. Our central tool is a new geometric interpretation of lamplighter groups involving natural families of quasi-median spaces. 
\end{abstract}

\tableofcontents

\section{Introduction}

\noindent
One of the central ideas in geometric group theory is that a finitely generated group can be considered itself as a geometric object. This  can be done by considering one of its Cayley graphs, or more generally any geodesic metric space on which the group acts properly and cocompactly by isometries. The specific choice of metric space may be important, for instance in the concept of CAT(0) groups. But since any two such spaces are quasi-isometric, any large scale geometric property is in fact an intrinsic property of the group.
 Historically, this point of view is motivated by early results exhibiting a tight connection between large scale geometric properties of a group and its algebraic structure, such as Stallings' theorem about multi-ended groups, Gromov's theorem about groups of polynomial growth, and Gromov's theory of hyperbolic groups; even though earlier motivations can be found, including for instance small cancellation groups and Mostow's rigidity. Nowadays, the program of classifying finitely generated groups up to quasi-isometry, as popularized in \cite{GromovAsymptotic}, is well-established and very active. 

\medskip \noindent
In this article, we contribute to this program by considering \emph{wreath products} of groups. Recall that, given two groups $F$ and $H$, the wreath product $F \wr H$ is defined as the semidirect product $\left( \bigoplus_H F \right) \rtimes H$ where $H$ acts on the direct sum by permuting the coordinates. These groups are also called \emph{lamplighter groups}, a terminology coined by Jim Cannon (see \cite{MR1062874}). The family of lamplighter groups is well-known in group theory, and has been studied from various perspectives over the years, including for instance random walks \cite{MR732356, MR1905862, MR1708557}, isoperimetric profiles \cite{MR2011120}, functional analysis \cite{MR2557962}, subgroup distortion \cite{MR2811580}, Haagerup property \cite{MR2318545, MR2393636, MR2888241, LampMedian}, and Hilbert space compression \cite{MR2271228, MR2644886, MR2783928, QM}. On the one hand, lamplighter groups have an easy and explicit definition, allowing an easy access to various properties and calculations. On the other hand, these groups are sufficiently exotic, i.e. sufficiently far away from most of the well-understood classes of groups exhibited in the literature, in order to exhibit interesting behaviours. The combination of these two observations probably explains the success of lamplighter groups, and why they are often used to produce counterexamples (see for instance \cite{MR1797748, MR1800990}). 

\medskip \noindent
The first work dedicated to the classification of wreath products up to quasi-isometry seems to be \cite{MR1800990}, whose elementary observations lead to non-trivial examples of quasi-isometric groups. More explicitly, given four finitely generated groups $F_1,F_2,H_1,H_2$, we know that, if $F_1$ (resp. $H_1$) is biLipschitz equivalent to $F_2$ (resp. $H_2$), then $F_1 \wr H_1$ and $F_2 \wr H_2$ are biLipschitz equivalent. As a consequence, the wreath products $\mathbb{Z}_{60} \wr \mathbb{Z}$ and $\mathfrak{A}_5 \wr \mathbb{Z}$, where $\mathfrak{A}_5$ and $\mathbb{Z}_{60}$ are the alternating and cyclic groups of order $60$, are quasi-isometric while $\mathbb{Z}_{60} \wr \mathbb{Z}$ is solvable but $\mathfrak{A}_5 \wr \mathbb{Z}$ not virtually solvable (contrasting with Gromov's theorem about groups of polynomial growth); also, the wreath products $\mathbb{Z} \wr \mathbb{Z}$ and $\mathbb{D}_\infty \wr \mathbb{Z}$ are quasi-isometric while $\mathbb{Z} \wr \mathbb{Z}$ is torsion-free but $\mathbb{D}_\infty \wr \mathbb{Z}$ not virtually torsion-free. However, a complete classification of wreath products up to quasi-isometry in full generality seems to be out of reach right now. In this article, we focus on the following question: 

\begin{question}
Let $F_1,F_2$ be two (non-trivial) finite groups and $H_1,H_2$ two finitely generated groups. When are $F_1 \wr H_1$ and $F_2 \wr H_2$ quasi-isometric?
\end{question}

\noindent
The same question can be found in \cite{MR1786869} for the specific case $H_1=H_2= \mathbb{Z}$, and has been solved in \cite{EFWI,EFWII}. A slight modification of the classification they get is the following (see Section \ref{section:Last} for more details):

\begin{thm}[\cite{EFWI,EFWII}]\label{thm:EFW}
Let $F_1,F_2$ be two finite groups and $H_1,H_2$ two finitely generated groups. Assume that $H_1$ is two-ended. Then $F_1 \wr H_1$ and $F_2 \wr H_2$ are quasi-isometric if and only if $H_1,H_2$ are quasi-isometric and $|F_1|,|F_2|$ are powers of a common number.
\end{thm}

\noindent
The proof of this theorem is highly non-trivial and does not seem to provide any valuable information outside the two-ended case. More precisely, the proof is based on the observation that $\mathbb{Z}_n \wr \mathbb{Z}$, where $\mathbb{Z}_n$ denotes the cyclic group of order $n$, admits a Cayley graph isomorphic to the horospherical product of two $n$-regular trees (also known as the Diestel-Leader graph $\mathrm{DL}(n)$) and it fundamentally exploits the quotient map $\mathbb{Z}_n \wr \mathbb{Z} \twoheadrightarrow \mathbb{Z}$ thought of as a \emph{height function} (as inspired by the solvable Baumslag-Solitar groups \cite{MR1608595, MR1709862} and other finitely presented abelian-by-cyclic groups \cite{MR1768110}). When replacing $\mathbb{Z}$ with a one-ended group, for instance $\mathbb{Z}^2$, such structures do not occur any more, hence the need of a new point of view. 

\medskip \noindent
As expressed in \cite[Paragraph IV.B.44]{MR1786869}, it is reasonable to think that the geometry of wreath products such as $\mathbb{Z}_n \wr \mathbb{Z}^2$ is more complicated than the geometry of $\mathbb{Z}_n \wr \mathbb{Z}$ because the solution to the travelling salesman problem (which is closely related to the metrics in wreath products) is more complicated on $\mathbb{Z}^2$ than on $\mathbb{Z}$, where it is particularly easy. This idea is also motivated by \cite{leBoudec}, where algebraically simple groups quasi-isometric to lamplighter groups over non-abelian free groups are constructed, which contracts drastically with lamplighters over $\mathbb{Z}$. By contrast, we show in this article that lamplighter groups over one-ended finitely presented groups are even more rigid than in the case of $\mathbb{Z}$. Exploiting this rigidity allows us  to completely classify them up to quasi-isometry.

\medskip \noindent
In fact, our techniques apply to a wider class of (non-necessarily vertex-transitive) graphs, defined as follows.

\begin{definition}
Let $X$ be a graph and $n \geq 2$ an integer. The \emph{lamplighter graph} $\mathcal{L}_n(X)$ is the graph 
\begin{itemize}
	\item whose vertices are the pairs $(c,x)$ with $c : V(X) \to \mathbb{Z}_n$ a finitely supported colouring taking values in the cyclic group of order $n$ and $x \in V(X)$ a vertex;
	\item and whose edges connect $(c_1,x_1)$ and $(c_2,x_2)$ either if $c_1=c_2$ and $x_1,x_2$ are adjacent, or if $x_1=x_2$ and $c_1,c_2$ differ only at this vertex.
\end{itemize}
A \emph{leaf} is a subgraph $\{(c,x) \mid x \in X\}$ where $c \in \mathbb{Z}_n^{(X)}$ is a fixed colouring.
\end{definition}

\noindent
Observe that, given a non-trivial finite group $F$ and a group $H$ with a finite generating set $S$, the Cayley graph $\mathrm{Cayl}(F \wr H, F \cup S)$ coincides with $\mathcal{L}_{|F|}(\mathrm{Cayl}(H,S))$ and its leaves correspond to $H$-cosets, justifying the terminology and the fact that this framework subsumes the geometric study of our wreath products. 

\medskip \noindent
The main result of this paper is the following classification:

\begin{thm}\label{thm:LampRigid}
Let $n,m \geq 2$ be two integers and $X,Y$ two coarsely $1$-connected uniformly one-ended graphs of bounded degree. 
 \begin{itemize}
  	\item If $X$ is amenable, then $\mathcal{L}_{n}(X)$ and $\mathcal{L}_{m}(Y)$ are quasi-isometric if and only if there exist $k,r,s \geq 1$ such that $n=k^{r}$, $m=k^{s}$ and such that there exists a quasi-$(s/r)$-to-one quasi-isometry $X \to Y$.
	\item If $X$ is non-amenable, then $\mathcal{L}_n(X)$ and $\mathcal{L}_m(Y)$ are quasi-isometric if and only if $X,Y$ are quasi-isometric and $n,m$ have the same prime divisors.
 \end{itemize}
 \end{thm}

\noindent
We refer to Section \ref{section:prel} for the definition of \emph{quasi-$\kappa$-to-one} maps, and more specifically to Proposition \ref{prop:kappa} when $\kappa$ is rational. The algebraic counterpart of Theorem \ref{thm:LampRigid} is the following (see Section \ref{section:Proofs} for a more general version that deals with \emph{permutational wreath products}):

\begin{cor}\label{cor:LampRigid}
Let $F_1,F_2$ be two non-trivial finite groups and $H_1,H_2$ two finitely presented groups. Assume that $H_1$ is one-ended. 
 \begin{itemize}
  	\item If $H_1$ is amenable, then $F_1 \wr H_1$ and $F_2 \wr H_2$ are quasi-isometric if and only if there exist $k,n_1,n_2 \geq 1$ such that $n=k^{n_1}$, $m=k^{n_2}$ and such that there exists a quasi-$(n_2/n_1)$-to-one quasi-isometry $H_1 \to H_2$.
	\item If $H_1$ is non-amenable, then $F_1 \wr H_1$ and $F_2 \wr H_2$ are quasi-isometric if and only if $H_1,H_2$ are quasi-isometric and $n_1,n_2$ have the same prime divisors.
 \end{itemize}
 \end{cor}

\noindent
The dichotomy provided by Theorem \ref{thm:LampRigid} between lamplighters over amenable and non-amenable graphs is twofold. The first difference deals with the indices $n,m$. In the amenable case, we recover from Theorem~\ref{thm:LampRigid} a criterion similar to Theorem~\ref{thm:EFW}. But in the non-amenable case, more lamplighter graphs turn out to be quasi-isometric. We emphasize that this phenomenon is not specific to coarsely $1$-connected and one-ended graphs: 

\begin{prop}
Let $X,Y$ be two arbitrary non-amenable graphs with bounded degree. If $n,m \geq 2$ are two integers having the same prime divisors, then $\mathcal{L}_n(X)$ and $\mathcal{L}_m(Y)$ are quasi-isometric. 
\end{prop}

\noindent
As a particular case, if $\mathbb{F}$ is a free group of finite rank $ \geq 2$, then $\mathbb{Z}_n \wr \mathbb{F}$ and $\mathbb{Z}_m \wr \mathbb{F}$ are quasi-isometric as soon as $n,m$ have the same prime divisors. See Proposition \ref{prop:QInonamenable} below for more details. Even though the proof is elementary, up to our knowledge, this fact has not been noticed before. The second difference comes from the fact that every quasi-isometry between two non-amenable graphs of bounded degree is quasi-one-to-one, i.e. it lies at finite distance from a bijective quasi-isometry \cite{MR1693847, MR1700742}, which may not be the case for amenable graphs of bounded degree \cite{MR1616135, MR1616159} or even for finitely generated amenable groups \cite{MR2730576}. (More details on this subject are given below.) This leads to a major difference between the amenable and non-amenable cases. As an illustration, because \cite{MR3318424} shows that higher rank lamplighter groups may be quasi-isometric without being biLipschitz equivalent, we deduce from Corollary~\ref{cor:LampRigid} the following somehow surprising consequence:

\begin{cor}
Let $F$ be a non-trivial finite group. There exist two finitely presented one-ended amenable groups $H_1,H_2$ that are quasi-isometric but such that $F \wr H_1, F\wr H_2$ are not quasi-isometric.
\end{cor}

\noindent
This second aspect of the difference between lamplighters over amenable and non-amenable groups is well-illustrated by the particular case $H_1=H_2$. For every finitely generated group $H$, set
$$\kappa(H):= \{ \kappa>0 \mid \text{there exists a quasi-$\kappa$-to-one quasi-isometry $H \to H$} \}.$$
(Observe that, since two word metrics (with respect to finite generating sets) are biLipschitz equivalent, the set $\kappa(H)$ does not depend on a particular choice of a finite generating set.) Then we have:

\begin{cor}\label{cor:QIkappa}
Let $F_1,F_2$ be two finite groups and $H$ a finitely generated amenable group. Then $F_1\wr H$ and $F_2 \wr H$ are quasi-isometric if and only if there exist $k,n_1,n_2 \geq 1$ such that $|F_1|=k^{n_1}$, $|F_2|=k^{n_2}$, and $n_1/n_2 \in \kappa(H)$.
\end{cor}

\noindent
This leads to the natural problem of investigating the structure of $\kappa(H)$ for a finitely generated amenable $H$. A first observation is that, as a consequence of Proposition \ref{prop:EasyKappa}, $\kappa(H)$ turns out to be a subgroup of the multiplicative group $\mathbb{R}_{>0}$. However, it is not clear what possible values can be taken by $\kappa(H)$. One the one hand, using homotheties in Euclidean spaces easily leads to the equality $\kappa(\mathbb{Z}^n)= \mathbb{R}_{>0}$ for every $n \geq 1$, hence:

\begin{cor}\label{cor:Abelian}
Let $F_1,F_2$ be two finite groups and $n \geq 2$ an integer. The wreath products $F_1 \wr \mathbb{Z}^n$ and $F_2 \wr \mathbb{Z}^n$ are quasi-isometric if and only if $|F_1|,|F_2|$ are powers of a common number. 
\end{cor}

\noindent
Based on the same idea, it can be shown that $\kappa(H)= \mathbb{R}_{>0}$ for every uniform lattice $H$ in a Carnot group, including the Heisenberg group. On the other hand, as a consequence of Corollary \ref{cor:kappatrivial} mentioned below, $\kappa(F \wr H)= \{1\}$ for every non-trivial finite group $F$ and every finitely presented one-ended amenable group $H$ (e.g. $\mathbb{Z}^2$). Intermediate values are also possible: indeed, it follows from \cite{MR2730576} that, for every $n \geq 2$, $\kappa(\mathbb{Z}_n \wr \mathbb{Z}) = \langle \text{prime factors of $n$} \rangle \subset \mathbb{Q}$. A more detailed discussion about the sets $\kappa(\cdot)$ is available in our work \cite{Kappa}.

\medskip \noindent
Let us also mention the following funny characterisation of amenability provided by Theorem \ref{thm:LampRigid} (even though it is not clear that it can be useful in practice):

\begin{cor}
Let $X$ be a coarsely $1$-connected uniformly one-ended graph of bounded degree. Then $X$ is amenable if and only if $\mathcal{L}_6(X)$ and $\mathcal{L}_{12}(X)$ are quasi-isometric. 
\end{cor}

\noindent
Of course, there is nothing specific about $6$ and $12$: they can be replaced with any two numbers that have the same prime divisors but that are not powers of a common number. 

\medskip \noindent
Finally, removing the one-ended assumption on the graphs, we are still able to show the following result.
\begin{thm}\label{thm:QIfactors}
Let $F_1,F_2$ be two finite groups and $H_1,H_2$ two finitely presented groups. If $F_1 \wr H_1$ and $F_2 \wr H_2$ are quasi-isometric, then so are $H_1$ and $H_2$. 
\end{thm}

\paragraph{Quasi-isometry vs.\ biLipschitz equivalence.} An early question in geometric group theory, which can be found in \cite[1.A']{GromovAsymptotic}, asks to which extend being quasi-isometric and being biLipschitz equivalent are different. For instance, are two quasi-isometric graphs with bounded degree necessarily biLipschitz equivalent? Partial positive answers was obtained for homogeneous trees \cite{MR1326733}, and next for hyperbolic groups \cite{Bogop}, before it was realised that a positive answer holds for every non-amenable graph with bounded degree in a strong way: 

\begin{thm}[\cite{MR1693847, MR1700742}]\label{thm:QIvsBil}
Every quasi-isometry between two non-amenable graphs with bounded degree lies at finite distance from a bijection.
\end{thm}

\noindent
In the opposite direction, counterexamples for amenable graphs with bounded degree was constructed, for instance in \cite{MR1616135, MR1616159}. However, such graphs were not Cayley graphs, so it was natural to specify the question to finitely generated groups instead of arbitrary (amenable) graphs with bounded degree: are two quasi-isometric finitely generated groups necessarily biLipschitz equivalent? This question can be found \cite[IV.B.46(vi)]{MR1786869} for instance. It was first proved that no analogue of Theorem \ref{thm:QIvsBil} holds for amenable groups: If $G$ is a finitely generated amenable group and $H \leq G$ a proper finite-index subgroup, then the inclusion $H \hookrightarrow G$ does not lie at finite distance from a bijection \cite{MR2139686}. (This observation is also a consequence of Lemma~\ref{lem:KappaWellDefined} and Proposition~\ref{prop:EasyKappa} below.) But this does not show that there does not exist a biLipschitz equivalence $H \to G$ (that would not be at finite distance from the inclusion). The first counterexamples were constructed in \cite{MR2730576} by considering lamplighter groups over~$\mathbb{Z}$. 

\begin{thm}[\cite{MR2730576}]
If $F_1,F_2$ are two finite groups such that $|F_2|=|F_1|^k$ for some $k \geq 2$ that is not a product of prime factors of $n$, then $F_2 \wr \mathbb{Z}$ is a finite-index subgroup of $F_1 \wr \mathbb{Z}$ but the two groups are not biLipschitz equivalent.
\end{thm}

\noindent
The proof of the theorem is fundamentally based on the structure of the quasi-isometry groups of lamplighter groups over $\mathbb{Z}$ as described in \cite{EFWI, EFWII}. Later, finitely presented counterexamples were obtained from higher rank lamplighter groups \cite{MR3318424}, based on the geometric picture of these groups obtained in \cite{MR2860983, MR2860984}, also based on the heavy machinery of coarse differentiation introduced in \cite{EFWI, EFWII}.

\medskip \noindent
As a consequence of our work, we obtain more elementary examples of quasi-isometric groups that are not biLipschitz equivalent:

\begin{cor}\label{cor:NotBiL}
Let $F_1,F_2$ be two non-trivial finite groups and $H$ one finitely presented amenable one-ended group. Then $F_1 \wr H$ and $F_2\wr H$ are biLipschitz equivalent if and only if $|F_1|=|F_2|$. 
\end{cor}

\noindent
As a consequence, if we assume that there exist $k,n_1,n_2$ such that $|F_1|=k^{n_1}$, $|F_2|=k^{n_2}$, and $n_1/n_2 \in \kappa(H)$, then $n_1 \neq n_2$ implies that $F_1 \wr H$ and $F_2 \wr H$ are quasi-isometric but they are not biLipschitz equivalent. As an illustration, $\mathbb{Z}_4 \wr \mathbb{Z}^2$ is a finite-index subgroup of $\mathbb{Z}_2 \wr \mathbb{Z}^2$ but these two groups are not biLipschitz equivalent. We also show that Theorem~\ref{thm:QIvsBil} does not characterise non-amenable groups by constructing the first examples of amenable groups such that every auto-quasi-isometry lies at finite distance from a bijection:

\begin{cor}\label{cor:kappatrivial}
Let $F$ be a non-trivial finite group and $H$ a finitely presented amenable one-ended group. Then every quasi-isometry $F \wr H \to F\wr H$ lies at finite distance from a bijection. 
\end{cor}

\noindent
We emphasize that such a property has strong consequences. Loosely speaking, $\kappa$ can be used as an Euler characteristic in this context. For instance:

\begin{cor}\label{cor:Commensurability}
Let $F$ be a non-trivial finite group and $H$ a finitely presented one-ended amenable group. Fix two finite-index subgroups $K_1,K_2 \leq F \wr H$. If $K_1$ and $K_2$ are biLipschitz equivalent (e.g. isomorphic) then they have the same index in $F \wr H$. 
\end{cor}

\noindent
Notice that the algebraic part of this statement cannot be proved by using an actual Euler characteristic as a consequence of \cite{MR837621} (see also \cite{MR1171301}). We refer to Proposition~\ref{prop:commensurable} for a more general statement.

\medskip \noindent
In the rest of the introduction, we give some details about the strategy used in order to prove Theorem \ref{thm:LampRigid}.

\paragraph{Aptolic quasi-isometries.} A central idea of the article is that a quasi-isometry $\mathcal{L}_n(X) \to \mathcal{L}_m(Y)$, where $n,m \geq 2$ where and $X,Y$ are coarsely $1$-connected uniformly one-ended, is compatible with the lamplighter structure in a strong way. We formalise this idea through the following concept:

\begin{definition}
Let $n,m \geq 1$ be two integers and $X,Y$ two graphs. A map $q: \mathcal{L}_n(X) \to \mathcal{L}_m(Y)$ is of {\it aptolic form}\footnote{This adjective comes from the contraction of the two Greek words $\alpha\pi\tau \omega$ (to light) and $\lambda \upsilon \chi \nu o \varsigma$ (lamp). It refers to a map that preserves the lamplighter structure.} if there exist $\alpha : \mathbb{Z}_n^{(X)} \to \mathbb{Z}_m^{(Y)}$ and $\beta :X \to Y$ such that $q(c,x)=(\alpha(c),\beta(x))$ for all $(c,x) \in \mathcal{L}_n(X)$. A quasi-isometry $\mathcal{L}_n(X) \to \mathcal{L}_m(Y)$ is \emph{aptolic} if it is of aptolic form and if it admits a quasi-inverse of aptolic form.
\end{definition}

\noindent
The first step towards the proof of Theorem \ref{thm:LampRigid} is to show that, if there exists an aptolic quasi-isometry $\mathcal{L}_n(X) \to \mathcal{L}_m(Y)$ between two lamplighter graphs, where $X,Y$ are not necessarily coarsely $1$-connected nor one-ended, then $X,Y$ must be quasi-isometric and $n,m$ must have the same prime divisors. Moreover, if in addition $X$ is amenable, then we obtain a stronger conclusion: $n,m$ must be powers of a common number, say $n=k^r$, $m=k^s$ for some $r,s \geq 1$; and there must exist a quasi-isometry $X \to Y$ that is quasi-$(s/r)$-to-one. This is done by elementary combinatorial arguments in Section \ref{bigsection:Aptolic}. The hard part is to prove that, if $X,Y$ are coarsely $1$-connected and uniformly one-ended graphs, an arbitrary quasi-isometry between our two graphs always lies at finite distance from an aptolic quasi-isometry, that is:

\begin{thm}\label{thm:StructureQI}
Let $n,m \geq 2$ be two integers and $X,Y$ two coarsely $1$-connected uniformly one-ended graphs of bounded degree. Then every quasi-isometry $\mathcal{L}_n(X) \to \mathcal{L}_m(Y)$ is at bounded distance from an aptolic quasi-isometry. 
 \end{thm}

\noindent
Observe that, as a consequence of Proposition \ref{prop:NonAptolic} below, the conclusion always fails outside the one-ended case. Combining this theorem with the previous observations related to aptolic quasi-isometries leads to a proof of Theorem \ref{thm:LampRigid}.

\paragraph{Embedding theorem.} Theorem \ref{thm:StructureQI} is proved in two steps. First, we prove in Section~\ref{section:CosetAptolic} that a quasi-isometry $\mathcal{L}_n(X) \to \mathcal{L}_m(Y)$, where $X,Y$ are not necessarily coarsely $1$-connected nor one-ended, that is \emph{leaf-preserving}, i.e. that sends every \emph{leaf} of $\mathcal{L}_n(X)$ at finite Hausdorff distance from a \emph{leaf} of $\mathcal{L}_n(X)$, must be at finite distance from an aptolic quasi-isometry. Next, we prove that, if $X,Y$ are coarsely $1$-connected and uniformly one-ended, then any quasi-isometry $\mathcal{L}_n(X) \to \mathcal{L}_m(Y)$ must be leaf-preserving. This step is the core of the article, and it lies on the following general embedding theorem:

\begin{thm}\label{thm:MainThm}
Let $X,Z$ be two graphs, $n \geq 2$ an integer, and $\rho : Z \to \mathcal{L}_n(X)$ a coarse embedding. If $Z$ is coarsely $1$-connected and uniformly one-ended, then the image of $\rho$ lies in the neighbourhood of a leaf in $\mathcal{L}_n(X)$.
\end{thm}

\noindent
Let us motivate and illustrate the strategy we follow in order to prove Theorem \ref{thm:MainThm} by considering lamplighter groups instead of more general lamplighter graphs. First, assume that there exists a $1$-Lipschitz coarse embedding $\rho$ from $\mathbb{Z}^2$ (identified with its Cayley graph associated to $S:= \{(0,1),(1,0)\}$) to $\mathbb{Z}_2 \wr \mathbb{Z}^2$ (identified with its Cayley graph associated to $\mathbb{Z}_2 \cup S$). It is an elementary observation that $\mathbb{Z}_2 \wr \mathbb{Z}^2$ has no $3$-cycle and that every $4$-cycle in $\mathbb{Z}_2 \wr \mathbb{Z}^2$ lies in a $\mathbb{Z}^2$-coset. As a consequence, the image of $\rho$ in the $2$-complex $X$ obtained from $\mathbb{Z}_2 \wr \mathbb{Z}^2$ by filling in all the $4$-cycles with discs is necessarily homotopically trivial. So $\rho$ lifts to the universal cover $\widetilde{X}$, giving a coarse embedding $\widetilde{\rho} : \mathbb{R}^2 \to \widetilde{X}$. But the geometry of $\widetilde{X}$ is quite specific. Intuitively, we think of $\mathbb{Z}_2 \wr \mathbb{Z}^2$ as endowed with a \emph{leaf structure} induced by the $\mathbb{Z}^2$-cosets. In $\mathbb{Z}_2 \wr \mathbb{Z}^2$, two leaves do not fellow-travel, i.e. the intersection between the neighbourhoods of two distinct leaves is always bounded. This motivates the idea that $\mathbb{Z}_2 \wr \mathbb{Z}^2$ has the local geometry of a tree of flats and that $\widetilde{X}$ should be a tree of flats. An algebraic justification of this picture is that $\widetilde{X}$ coincides with the Cayley $2$-complex of the truncated presentation
$$\langle a,r,s \mid a^2=1, [r,s]=1 \rangle \simeq \mathbb{Z}_2 \ast \mathbb{Z}^2$$
obtained from the presentation
$$\langle a,r,s \mid a^2=1, [r,s]=1, \left[ a, r^ms^nas^{-n}r^{-m} \right]=1 \ (n,m \in \mathbb{Z}) \rangle$$
of $\mathbb{Z}_2 \wr \mathbb{Z}^2$. Therefore, $\widetilde{X}$ is indeed a tree of flats, which implies that the image of $\widetilde{\rho} : \mathbb{Z}^2 \to \widetilde{X}$ must lie in the neighbourhood of a flat; and because the covering map $\widetilde{X} \to X$ sends every flat to a $\mathbb{Z}^2$-coset (up to finite distance), we conclude that the image of $\rho : \mathbb{Z}^2 \to \mathbb{Z}_2 \wr \mathbb{Z}^2$ lies in the neighbourhood of $\mathbb{Z}^2$-coset. 

\medskip \noindent
In the general case of an arbitrary coarse embedding $\rho : Z \to \mathbb{Z}_2 \wr H$ from a coarsely $1$-connected uniformly one-ended graph $Z$ (e.g. the Cayley graph of a finitely presented one-ended group like $\mathbb{Z}^2$), we follow the same idea. We fix a large $R \geq 0$ and we construct a $2$-complex $X$ from $\mathbb{Z}_2 \wr H$ by filling in with discs all the cycles lying in $H$-cosets and all the cycles of length $\leq R$. If $R$ is well-chosen, the image of any loop of $Z$ by $\rho$ is homotopically trivial in $X$, so $\rho$ lifts to the universal cover $\widetilde{X}$, giving a coarse embedding $\widetilde{\rho} : Z \to \widetilde{X}$. In order to understand the geometry of $\widetilde{X}$, observe that it coincides with the Cayley $2$-complex of the truncated presentation
$$\langle a, H \mid a^2=1, [a,hah^{-1}]=1 \ (h \in H) \rangle \text{ for some $S \subset H$ finite}$$
obtained from the presentation
$$\langle a, H \mid a^2=1, [a,hah^{-1}]=1 \ (h \in H) \rangle$$
of $\mathbb{Z}_2 \wr H$. However, the group $\mathbb{Z}_2 \square_S H$ defined by the former presentation may no longer be a tree of copies of $H$, so it is not immediately obvious that the image of $\widetilde{\rho} : Z \to \mathbb{Z}_2 \square_S H$ has to lie in the neighbourhood of an $H$-coset. Nevertheless, $\mathbb{Z}_2 \square_S H$ turns out to have a remarkable algebraic structure: it splits as a semidirect product $C(\Gamma) \rtimes H$ where $C(\Gamma)$ denotes the right-angled Coxeter group defined by $\Gamma:= \mathrm{Cayl}(H,S)$. The key observation is that the well-known structure of $C(\Gamma)$ as a \emph{median graph} (i.e. the one-skeleton of a CAT(0) cube complex) induces a wallspace structure on $\mathbb{Z}_2 \square_S H$ with the property that every wall has a bounded image in $\mathbb{Z}_2 \square_S H$ under the covering map $\mathbb{Z}_2 \square_S H \twoheadrightarrow \mathbb{Z}_2 \wr H$ (which coincides with the quotient map from an algebraic point of view). This implies that the image of $\widetilde{\rho}$ in $\mathbb{Z}_2 \square_S H= C(\Gamma) \rtimes H$ has to avoid the factor $C(\Gamma)$. Indeed, otherwise it would be possible to separate this image with a wall, and consequently to separate the image of $\rho$ with a bounded set, contradicting the assumption that $Z$ is one-ended. In other words, the image of $\widetilde{\rho}$ in $\mathbb{Z}_2 \square_S H$ must lie in the neighbourhood of an $H$-coset, and we conclude that the image of $\rho$ in $\mathbb{Z}_2 \wr H$ must lie in the neighbourhood of an $H$-coset, as desired. 

\medskip \noindent
The wallspace structure we define on $\mathbb{Z}_2 \square_S H$ follows from a description of the Cayley graph of $\mathbb{Z}_2 \square_S H$ in terms of pointed edges in the median graph associated to $C(\Gamma)$. Indeed, observe that an element of $C(\Gamma) \rtimes H$ is given by a pair $(g,h)$ where $g \in C(\Gamma)$ can be thought of as a vertex in the median graph $\mathrm{Cayl}(C(\Gamma),H)$ of $C(\Gamma)$ and where $h \in H$ can be thought of as a direction starting from $g$; in other words, $(g,h)$ naturally corresponds to the edge $(g,gh)$ of $\mathrm{Cayl}(C(\Gamma),H)$ pointed at $g$. 

\medskip \noindent
However, this description is specific to the lamplighters $\mathbb{Z}_2 \wr H$. When $\mathbb{Z}_2$ is replaced with a larger finite group, the right-angled Coxeter group $C(\Gamma)$ has to be replaced with a \emph{graph product} of finite groups, and median geometry has to be replaced with \emph{quasi-median geometry}. But the arguments can be adapted with no major modifications. When dealing with \emph{lamplighter graphs} instead of lamplighter groups, there is no presentation to truncate, but thinking in terms of pointed cliques in quasi-median graphs (generalising our previous pointed edges in median graphs) remains possible. We develop this point of view in Section~\ref{section:approximation}, and adapt the strategy described above in Sections~\ref{section:EmbeddingProof} and~\ref{section:ProofsLemmas}.

\paragraph{Other applications.} In view of the classification provided by Corollary \ref{cor:LampRigid}, the natural question to ask next is: when is a finitely generated group quasi-isometric to a lamplighter group $F \wr H$ where $F$ is a finite group and $H$ a finitely presented one-ended group? Often, answering such a question requires a description of the quasi-isometry group of the group under consideration. Our work provides a promising partial description of $\mathrm{QI}(F\wr H)$, but obtaining a precise global picture will require further work. We plan to write on the subject in a near future. 

\medskip \noindent
Meanwhile, avoiding a full description of the quasi-isometry groups thanks to the recent \cite[Theorem 1.1]{QIsubgroup}, we are able to prove the following partial solution to our problem:

\begin{thm}\label{intro:QItoW}
Let $F$ be a non-trivial finite group, $H$ a finitely presented one-ended group, and $G$ a finitely generated group. If $G$ is quasi-isometric to $F \wr H$, then there exist finitely many subgroups $H_1, \ldots, H_n \leq G$ such that:
\begin{itemize}
	\item $H_1, \ldots, H_n$ are all quasi-isometric to $H$;
	\item the collection $\{H_1, \ldots, H_n\}$ is almost malnormal;
	\item for every finitely presented one-ended subgroup $K\leq G$, there exist $g \in G$ and $1 \leq i \leq n$ such that $K \leq gH_ig^{-1}$.
\end{itemize}
\end{thm}

\noindent
Theorem \ref{intro:QItoW} imposes severe algebraic restrictions on the finitely generated groups that are quasi-isometric to our wreath products. As an application, the combination of Corollaries \ref{cor:PermutationalW} and \ref{cor:PermutationW} below characterises when a permutational wreath product between finite and finitely presented one-ended groups is quasi-isometric to a lamplighter group.

\paragraph{Acknowledgements.} We thank A. Le Boudec, J. Brieussel, Y. Cornulier, D. Fisher for their comments on the first version of our manuscript.

\section{Preliminaries}\label{section:prel}

\noindent
In this section, we collect some basic definitions and notations that will be used in the rest of the article.

\paragraph{2.1. Lamplighter graphs.} Recall from the introduction that, given an integer $n \geq 2$ and a graph $X$, the \emph{lamplighter graph} $\mathcal{L}_n(X)$ is the graph
\begin{itemize}
	\item whose vertices are the pairs $(c,x)$ with $c : V(X) \to \mathbb{Z}_n$ a finitely supported colouring (denoted by $c \in \mathbb{Z}_n^{(X)}$, where $\mathbb{Z}_n$ is the cyclic group of order $n$) and $x \in V(X)$ a vertex (thought of as an arrow pointing at $x$);
	\item and whose edges connect $(c_1,x_1)$ and $(c_2,x_2)$ either if $c_1=c_2$ and $x_1,x_2$ are adjacent, or if $x_1=x_2$ and $c_1,c_2$ differ only at this vertex.
\end{itemize}
Loosely speaking, moving a vertex in $\mathcal{L}_n(X)$ amounts to moving an arrow in $X$ that is able to modify a colouring of $X$ where it is. Notice that, given a non-trivial finite group $F$, a group $H$, and a generating set $S$, the Cayley graph $\mathrm{Cayl}(F \wr H,F \cup S)$ coincides with $\mathcal{L}_{|F|} ( \mathrm{Cayl}(H,S))$. 

\medskip \noindent
In the sequel, we shall use the following useful notation. Given a subset $A\subset X$, we denote by $\mathcal{L}(A)$ the subgroup $\bigoplus_{a\in A} \mathbb{Z}_n$ of $\bigoplus_A \mathbb{Z}_n$. In other words, $\mathcal{L}(A)$ is the collection of all colourings supported in $A$. 

\medskip \noindent
As a graph, $\mathcal{L}_n(X)$ has a canonical metric, i.e. the distance between any two vertices corresponds to the minimal length of a path between them (each edge having length one). However, it may be convenient to endow $\mathcal{L}_n(X)$ with another metric. These two metrics, referred to as the \emph{diligent} and \emph{lazy} metrics, are biLipschitz equivalent so choosing one instead of the other has no consequence on our study of the asymptotic geometry of lamplighter graphs. The convention we follow is that $\mathcal{L}_n(X)$ is by default endowed with its graph metric, and the use of the diligent metric will be always explicitly mentioned. Loosely speaking, in order to go from $(c_1,p_1)$ to $(c_2,p_2)$ in $\mathcal{L}_n(X)$ with respect to the lazy metric (i.e. the graph metric), the arrow moves from $p_1$ to $p_2$ in $X$ and stops at each point where $c_1, c_2$ differ in order to modify the colouring from the value of $c_1$ to the value of $c_2$; with respect to the diligent metric, the arrow passes through each point where $c_1,c_2$ differ but it does not need to stop in order to modify the colouring.

\paragraph{The diligent metric.} The graph metric obtained from $\mathcal{L}_n(X)$ by adding an edge between any two vertices $(c_1,p_1),(c_2,p_2)$ such that $p_1,p_2$ are adjacent in $X$ and such that $c_1,c_2$ may only differ at $p_1$ is referred to as the \emph{diligent metric}. With respect to this metric, the distance between any two points $(c_1,p_1),(c_2,p_2) \in \mathcal{L}_n(X)$ coincides with the shortest length of a path in $X$ starting from $p_1$, visiting all the points where $c_1,c_2$ differ (i.e. all the points in $\mathrm{supp}(c_1-c_2)$), and ending at $p_2$. 

\medskip \noindent
(For simplicity, here we use the convention that the length of a path reduced to a single point is one. Indeed, observe that, if $\mathrm{supp}(c_1-c_2)=\{p_1\}= \{p_2\}$, then the distance between $(c_1,p_1)$ and $(c_2,p_2)$ is $1$ while the shortest path starting from $p_1$, visiting all the points in $\mathrm{supp}(c_1-c_2)$, and ending at $p_2$ is reduced to a single point, namely $p_1=p_2$.) 

\medskip \noindent
Notice that, given a non-trivial finite group $F$ and a group $H$ generating by some $S \subset H$, the diligent metric defined on $\mathrm{Cayl}(F\wr H, F \cup S)$ (through its identification with $\mathcal{L}_{|F|}(\mathrm{Cayl}(H,S))$) coincides with the word metric associated to the generating set $F \cdot (S \cup \{1\})$.

\paragraph{The lazy metric.} We refer to the graph metric of $\mathcal{L}_n(X)$ as the \emph{lazy metric}. Observe that the lazy distance between any two points $(c_1,p_1),(c_2,p_2) \in \mathcal{L}_n(X)$ coincides with
$$d_{\mathrm{dil}}((c_1,p_1),(c_2,p_2)) + | \mathrm{supp}(c_1-c_2)|$$
where $d_\mathrm{dil}$ denotes the diligent metric. As a consequence, we have $d_\mathrm{dil} \leq d_{\mathrm{laz}} \leq 2 d_\mathrm{dil}$, so our two metrics are biLipschitz equivalent.

\paragraph{Leaves.} The lamplighter graph $\mathcal{L}_n(X)$ contains natural copies of $X$, namely the subgraphs 
$$X(c):=\{(c,x) \mid x \in X\} \text{ where $c \in \mathbb{Z}_n^{(X)}$ is a fixed colouring.}$$
For convenience, we identify $X(0)$ with $X$. We refer to these subgraphs as the \emph{leaves} of $\mathcal{L}_n(X)$. Observe that, in $\mathcal{L}_n(X)$, the leaves do not fellow-travel:

\begin{fact}
For every $K \geq 0$ and for any two distinct leaves $L_1,L_2 \subset \mathcal{L}_n(X)$, the intersection $L_1^{+K} \cap L_2^{+K}$ of the $K$-neighbourhoods of $L_1,L_2$ is bounded.
\end{fact}

\begin{proof}
Fix two distinct colourings $c_1,c_2 \in \mathbb{Z}_n^{(X)}$ such that $L_1=X(c_1)$ and $L_2=X(c_2)$. If a vertex $(c,p)$ belongs to $L_1^{+K}$, then either $c=c_1$ or $c,c_1$ differ in $B(p,K)$; similarly, if $(c,p)$ belongs to $L_2^{+K}$ then either $c=c_2$ or $c,c_2$ differ in $B(p,K)$. Because $c_1,c_2$ are distinct, we have
$$L_1^{+K} \cap L_2^{+K} \subset \{ (c,p) \in \mathcal{L}_n(X) \mid \{p\} \cup \mathrm{supp}(c-c_1) \subset \mathrm{supp}(c_1-c_2)^{+K}\}.$$
The subset in the right-hand side being bounded, the desired conclusion follows.
\end{proof}

\noindent
Finally, observe that $\mathcal{L}_n(X)$ naturally projects onto $X$ through $\pi_X : (c,p) \mapsto p$. Clearly, $\pi_X$ is $1$-Lipschitz (with respect to the diligent and lazy metrics). Algebraically speaking, given a non-trivial finite group $F$ and a group $H$ generating by some $S \subset H$, the projection of $\mathrm{Cayl}(F \wr H,F \cup S)$ (when thought of as $\mathcal{L}_{|F|}(\mathrm{Cayl}(H,S))$) onto $\mathrm{Cayl}(H,S)$ as defined above coincides with the quotient map $F \wr H \twoheadrightarrow H$.

\paragraph{2.2. Coarse embeddings.} A map $f : X \to Y$ between two metric spaces $X$ and $Y$ is a \emph{coarse embedding} if there exist two functions $\rho_1,\rho_2 : [0,+ \infty) \to [0,+ \infty)$ tending to infinity such that
$$\rho_1(d(x,y)) \leq d(f(x),f(y)) \leq \rho_2(d(x,y))) \text{ for all $x,y \in X$.}$$
The functions $\rho_1,\rho_2$ are referred to as the \emph{parameters} of $f$. A coarse embedding with affine parameters is a \emph{quasi-isometric embedding}. More precisely, given $A>0$ and $B \geq 0$, a map $f : X \to Y$ is an \emph{$(A,B)$-quasi-isometric embedding} if
$$\frac{1}{A} \cdot d(x,y) - B \leq d(f(x),f(y)) \leq A \cdot d(x,y) + B \text{ for all $x,y \in X$}.$$
It is an \emph{$(A,B)$-quasi-isometry} if in addition every point in $Y$ is within $B$ from $f(X)$. A \emph{biLipschitz equivalence} is a bijective coarse embedding with linear parameters. Among discrete metric spaces (like graphs), a bijective quasi-isometry is automatically a biLipschitz equivalence.

\medskip \noindent
We record the following statement for future use:

\begin{lemma}\label{lem:CoarseLift}
Let $Z$ be a connected graph, $\pi : A \to B$ a covering map between two cellular complexes, and $\eta : Z \to B^{(1)}$ a continuous map. Assume that $\eta(Z)$ is simply connected in $B$. Then $\eta$ lifts as $\widetilde{\eta} : Z \to A^{(1)}$ and
$$d_A(\widetilde{\eta}(x),\widetilde{\eta}(y)) = d_B(\eta(x),\eta(y)) \text{ for all vertices $x,y \in Z$}$$
where $d_A,d_B$ refer to the graph metrics in $A^{(1)},B^{(1)}$. As a consequence, if $\eta$ is a coarse embedding then $\widetilde{\eta}$ is a coarse embedding with the same parameters.
\end{lemma}

\begin{proof}
Fix two vertices $x,y \in Z$. Because $\pi$ is $1$-Lipschitz, we have 
$$d_A(\widetilde{\eta}(x),\widetilde{\eta}(y)) \geq d_B(\pi \circ \widetilde{\eta} (x), \pi \circ \widetilde{\eta}(y)) = d_B(\eta(x),\eta(y)).$$
Next, let $\zeta$ be a geodesic in $\eta(Z) \subset B^{(1)}$ from $\eta(x)$ to $\eta(y)$ and let $\widetilde{\zeta} \subset A$ denote the lift of $\zeta$ that starts from $\widetilde{\eta}(x)$. Observe that $\widetilde{\zeta}$ ends at $\widetilde{\eta}(y)$. Otherwise, by letting $\widetilde{\xi}$ be a path in $\widetilde{\eta}(Z)$ from $\widetilde{\eta}(x)$ to $\widetilde{\eta}(y)$, the concatenation of $\gamma$ and $\pi(\widetilde{\xi})$ would create a loop in $\eta(Z)$ that is not homotopically trivial in $B$. So $\widetilde{\zeta}$ indeed ends at $\widetilde{\eta}$, hence
$$d_A(\widetilde{\eta}(x),\widetilde{\eta}(y)) \leq \mathrm{length}(\widetilde{\zeta})= \mathrm{length}(\zeta) = d_B(\eta(x),\eta(y)).$$
This completes the proof of our lemma.
\end{proof}

\paragraph{2.3. Amenable graphs.} Given a locally finite graph $X$, a \emph{F\o lner sequence} is a sequence of finite subsets $(A_n)$ such that $|\partial A_n|/ |A_n| \to 0$ as $n \to + \infty$, where $|\cdot|$ denotes the cardinality of the subset under consideration and where we denote by $\partial S$ the \emph{boundary} of a finite subset $S \subset X$ (i.e. the set of vertices not in $S$ that are adjacent to vertices in $S$). A graph is \emph{amenable} if it admits a F\o lner sequence. It is well-known that a finitely generated group is amenable if and only if its Cayley graphs (with respect to finite generating sets) are amenable in the above sense.

\paragraph{2.4. Scaling quasi-isometries.} Let us record from \cite{Kappa} how to define maps that are ``coarsely $\kappa$-to-one'' (for some real number $\kappa>0$) and a few elementary properties satisfied by such maps. The concept of quasi-$\kappa$-to-one quasi-isometries will be central in Section~\ref{section:AmenableCase}, dedicated to aptolic quasi-isometries between lamplighters over amenable graphs.

\begin{definition}\label{def:quasione}
Let $f:X\to Y$ be a proper map between two graphs $X,Y$ and let $\kappa>0$. Then $f$ is \emph{quasi-$\kappa$-to-one} if there exists a constant $C>0$ such that
\[\left| \kappa |A|- |f^{-1}(A)| \right|\leq C|\partial A|\]
for all finite subset $A\subset Y$.
\end{definition}

\noindent
Notice that, for every integer $n \geq 1$, an $n$-to-one map is quasi-$n$-to-one. The terminology used in Definition \ref{def:quasione} is also justified by the fact that a quasi-isometry that is quasi-one-to-one lies at finite distance from a bijection. Proposition \ref{prop:kappa} below also gives alternative definitions of being quasi-$\kappa$-to-one when $\kappa$ is rational. The former observation is a straightforward consequence of a result of Whyte \cite[Theorems~A and~C]{MR1700742} (see also \cite[Theorems~5.3 and~5.4]{MR2730576}). We refer to \cite[Proposition~4.1]{Kappa} for more details.

\begin{prop}\label{prop:QIdistBij}
Let $f:X_1\to X_2$ be a quasi-isometry between two graphs with bounded degree. Then $f$ is at bounded distance from a bijection if and only if it is quasi-one-to-one.
\end{prop}

\noindent
It is not difficult to show that, given a finitely generated group $G$ and a finite-index subgroup $H \leq G$, the inclusion $H \hookrightarrow G$ is quasi-$(1/[G:H])$-to-one. On the other hand, if $G$ is non-amenable, it follows from Theorem \ref{thm:QIvsBil} that $H \hookrightarrow G$ is also quasi-one-to-one. Therefore, the property of being quasi-$\kappa$-to-one is not quite informative in the non-amenable case. The next lemma, proved in \cite[Lemma~3.5]{Kappa}, shows that this phenomenon does not happen in the amenable case.

\begin{lemma}\label{lem:KappaWellDefined}
Let $f : X \to Y$ be a proper map between two graphs. Assume that $X$ is amenable. If $f$ is both quasi-$\kappa_1$-to-one and quasi-$\kappa_2$-to-one for some $\kappa_1,\kappa_2>0$, then $\kappa_1=\kappa_2$.
\end{lemma}

\noindent
Our next statement, proved in \cite[Proposition~3.6]{Kappa}, shows how being quasi-$\kappa$-to-one is compatible with composition.  

\begin{prop}\label{prop:EasyKappa}
Let $X,Y,Z$ be three connected graphs with  bounded degree, $\kappa_1,\kappa_2>0$ two real numbers, and $f,h:X\to Y$ and $g:Y\to Z$  three quasi-isometries. 
\begin{itemize}
\item[(i)] If $f, h$ are at bounded distance and if $f$ is quasi-$\kappa_1$-to-one, then $h$ is also quasi-$\kappa_1$-to-one.
\item[(ii)] If $f$ and $g$ are respectively quasi-$\kappa_1$-to-one and quasi-$\kappa_2$-to-one, then $g\circ f$ is quasi-$\kappa_1\kappa_2$-to-one. 
\item[(iii)] If $\bar{f}$ is a quasi-inverse of $f$ and if $f$ is quasi-$\kappa_1$-to-one, then $\bar{f}$ is quasi-$(1/\kappa_1)$-to-one.  
\end{itemize}
\end{prop}

\noindent
Finally, in case $\kappa$ is rational, we proved in \cite[Proposition~4.2]{Kappa} the following equivalent formulations of Definition~\ref{def:quasione} .

\begin{prop}\label{prop:kappa}
Let $m,n \geq 1$ be natural integers and $f:X\to Y$ a quasi-isometry between two graphs with bounded degree. The following statements are equivalent:
\begin{itemize}
\item[(i)] $f$ is quasi-$(m/n)$-to-one;
\item[(ii)] the map $\iota \circ f\circ \pi$ is at bounded distance from a bijection, where $\pi : X\times \mathbb{Z}_n \twoheadrightarrow X$ is the canonical embedding and $\iota : Y \hookrightarrow Y\times \mathbb{Z}_m$ the canonical projection.
\item[(iii)] there exist a partition $\mathcal{P}_X$ (resp. $\mathcal{P}_Y$) of $X$ (resp. of $Y$) with uniformly bounded pieces of size $m$ (resp. $n$) and a bijection $\psi:\mathcal{P}_X\to \mathcal{P}_Y$ such that $f$ is at bounded distance from a map $g : X \to Y$ satisfying $g(P) \subset \psi(P)$ for every $P \in \mathcal{P}_X$.
\end{itemize}
\end{prop}

\paragraph{2.5. A few facts.} We conclude this preliminary section with a few elementary observations about graphs with bounded degree.

\begin{fact}\label{fact:Easy}
Let $X$ be a graph with bounded degree. Then $|A^{+K}| \leq N^K \cdot |A|$ for every finite subset $A \subset X$ and every constant $K \geq 0$, where $N \geq 3$ is a fixed integer larger than the maximal degree of a vertex in $X$.
\end{fact}

\begin{proof}
Because every vertex in $X$ has at most $N$ neighbours, it follows that
$$|A^{+K}| \leq  \sum\limits_{i=0}^{K-1} N^i \cdot | A| \leq N^K \cdot |\partial A|$$
as desired.
\end{proof}

\begin{fact}\label{fact:EasyOne}
Let $X$ be a graph with bounded degree. Then $|A^{+K} \backslash A | \leq  N^{K} \cdot |\partial A|$ for every finite subset $A \subset X$ and every constant $K \geq 1$, where $N \geq 3$ is a fixed integer larger than the maximal degree of a vertex in $X$.
\end{fact}

\begin{proof}
By noticing that $A^{+K} \backslash A \subset (\partial A)^{+(K-1)}$, the desired conclusion follows from Fact~\ref{fact:Easy}.
\end{proof}

\begin{fact}\label{fact:EasyTwo}
Let $X,Y$ be two graphs with bounded degree, $\kappa>0$ a real number, and $f : X \to Y$ a quasi-isometry. There exists a constant $M \geq 0$ such that $|\partial f^{-1}(A)| \leq M \cdot |\partial A|$ for every finite subset $A \subset Y$. 
\end{fact}

\begin{proof}
Let $C \geq 0$ be such that $\left| \kappa |A|- |f^{-1}(A)| \right| \leq C \cdot |\partial A|$ for every finite subset $A \subset Y$. Also, let $Q \geq 0$ be such that $f$ sends two adjacent vertices to two vertices at distance $\leq Q$; and let $P \geq 1$ denote the maximal cardinality of the preimage under $f$ of a point.

\medskip \noindent
Now, fix a finite subset $A \subset Y$. By definition, if $x$ belongs to $\partial f^{-1}(A)$ then it does not belong to $f^{-1}(A)$ and it has a neighbour $y \in f^{-1}(A)$, so $f(x)$ does not belong to $A$ and it is within $Q$ from $f(y) \in A$, i.e. $f(x) \in A^{+Q} \backslash A$. In other words, we have proved that $\partial f^{-1}(A) \subset f^{-1} ( A^{+Q} \backslash A)$. Then
$$|\partial f^{-1}(A)| \leq \left| f^{-1}(A^{+Q} \backslash A) \right| \leq P \left| A^{+Q} \backslash A \right| \leq PN^Q \cdot |\partial A|$$
where the last inequality is justified by Fact \ref{fact:EasyTwo}. 
\end{proof}

\section{Aptolic quasi-isometries}\label{bigsection:Aptolic}

\subsection{Generalities}\label{section:Aptolic}

\noindent
Recall from the introduction that, given two integers $n,m \geq 2$ and two graphs $X,Y$, a quasi-isometry $q : \mathcal{L}_n(X) \to \mathcal{L}_m(Y)$ is aptolic if it is of aptolic form and if it admits a quasi-inverse of aptolic form, i.e. there exist four maps $\alpha : \mathbb{Z}_n^{(X)} \to \mathbb{Z}_n^{(Y)}$, $\alpha' : \mathbb{Z}_n^{(Y)} \to \mathbb{Z}_n^{(X)}$, $\beta : X \to Y$ and $\beta' : Y \to X$ such that
$$q(c,p)= \left( \alpha(c),\beta(p) \right) \text{ for all $(c,p) \in \mathcal{L}_n(X)$}$$
and such that 
$$(c,p) \mapsto (\alpha'(c), \beta'(p)), \ (c,p) \in \mathcal{L}_m(Y)$$
is a quasi-inverse of $q$.
In this section, we record a few elementary observations about aptolic quasi-isometries. Regarding the classification of lamplighter groups up to quasi-isometry, our main result states that, if there exists an aptolic quasi-isometry $\mathcal{L}_n(X) \to \mathcal{L}_m(Y)$, then $n$ and $m$ must have the same prime divisors. See Proposition \ref{prop:general}. In the sequel, we endow every lamplighter graph with the diligent metric.

\medskip \noindent
We begin by characterising aptolic quasi-isometries among maps of aptolic form.

\begin{prop}\label{prop:AptoQI}
Let $n,m \geq 2$ be two integers, $X,Y$ two unbounded graphs, $\alpha : \mathbb{Z}_n^{(X)} \to \mathbb{Z}_n^{(Y)}$ and $\beta : X \to Y$ two maps. Then
$$q : \left\{ \begin{array}{ccc} \mathcal{L}_n(X) & \to & \mathcal{L}_m(Y) \\ (c,p) & \mapsto & (\alpha(c),\beta(p)) \end{array} \right.$$
is an aptolic quasi-isometry if and only if the following conditions hold:
\begin{itemize}
	\item[(i)] $\alpha$ is a bijection;
	\item[(ii)] $\beta$ is a quasi-isometry;
	\item[(iii)] there exists $Q \geq 0$ such that, for all colourings $c_1,c_2 \in \mathbb{Z}_n^{(X)}$, the Hausdorff distance between $\beta(\mathrm{supp}(c_1-c_2))$ and $\mathrm{supp}(\alpha(c_1)-\alpha(c_2))$ is at most $Q$. 
\end{itemize}
If so, every aptolic quasi-inverse of $q$ is of the form
$$(c,p) \mapsto \left( \alpha^{-1}(c), \bar{\beta}(p) \right), \ (c,p) \in \mathcal{L}_m(Y)$$
where $\bar{\beta} : Y \to X$ is a quasi-inverse of $\beta$. 
\end{prop}

\begin{proof}
First, assume that $q$ is an aptolic quasi-isometry. Let $C,K \geq 0$ be two constants and $\bar{\alpha} : \mathbb{Z}_n^{(Y)}$, $\bar{\beta} : Y \to X$ two maps such that $q$ is a $(C,K)$-quasi-isometry, such that
$$\bar{q} : (c,p) \mapsto (\bar{\alpha}(c), \bar{\beta}(p)), \ (c,p) \in \mathcal{L}_m(Y)$$
is a quasi-isometry, and such that $q \circ \bar{q}$, $\bar{q} \circ q$ are within $K$ from identities. 

\medskip \noindent
We begin by proving $(ii)$. Notice that, for all $a,b \in X$, we have
$$\begin{array}{lcl} d(\beta(a),\beta(b)) & = & d((\alpha(0),\beta(a)),(\alpha(0),\beta(b))) = d(q(0,a),q(0,b)) \\ \\ & \leq & C \cdot d((0,a),(0,b))+K = C \cdot d(a,b)+K, \end{array}$$
and we show similarly that
$$d(\beta(a),\beta(b)) \geq \frac{1}{C} d(a,b) - K.$$
Thus, $\beta$ defines a $(C,K)$-quasi-isometric embedding. Next, notice that, for every $h \in Y$, there exists some $(c,p) \in \mathcal{L}_n(X)$ such that $d(q(c,p),(0,h)) \leq K$. Since 
$$d\left(\beta(p),h \right) \leq d\left((\alpha(c),\beta(p)),(0,h) \right) = d \left( q(c,p),(0,h) \right) \leq K,$$
we conclude that $\beta$ is a quasi-isometry. Notice our arguments only use the fact that $q$ is a quasi-isometry. We record this assertion for future use.

\begin{fact}\label{fact:Beta}
Given two maps $\alpha : \mathbb{Z}_n^{(X)} \to \mathbb{Z}_n^{(Y)}$ and $\beta : X \to Y$, if
$$(c,p) \mapsto (\alpha(c), \beta(p)), \ (c,p) \in \mathcal{L}_n(X)$$
defines a quasi-isometry $\mathcal{L}_n(X) \to \mathcal{L}_m(Y)$, then $\beta$ defines a quasi-isometry $X \to Y$. 
\end{fact}

\noindent
Thus, we have proved $(ii)$. Observe that, by symmetry, we also know that $\bar{\beta}$ is a quasi-isometry. 

\medskip \noindent
Now, we want to prove $(i)$. Given $(c,p) \in \mathcal{L}_m(Y)$, we have
$$d((\alpha \circ \bar{\alpha} (c), \beta \circ \bar{\beta}(p)), (c,p)) = d( q \circ \bar{q}(c,p), (c,p)) \leq K.$$
Therefore, we must have $d(\beta \circ \bar{\beta}(p),p) \leq K$ and $c, \alpha \circ \bar{\alpha}(c)$ may only differ in the ball $B(p,K)$. The former observation implies that $\beta \circ \bar{\beta}$ is within $K$ from the identity, and we deduce from the latter observation by letting $p$ go to infinity in $Y$ that $\alpha \circ \bar{\alpha}$ is the identity. By symmetry, we obtain similarly that $\bar{\beta} \circ \beta$ is within $K$ from the identity and that $\bar{\alpha} \circ \alpha$ is the identity. Consequently, $\alpha$ is a bijection and $\bar{\alpha}=\alpha^{-1}$, proving $(i)$; also, $\bar{\beta}$ must be a quasi-inverse of $\beta$, proving the last assertion of our proposition. 

\medskip \noindent
Finally, we want to prove $(iii)$. So let $c_1,c_2 \in \mathbb{Z}_n^{(X)}$ be two colourings. Fix a sequence 
$$a_1=c_1, \ a_2, \ldots, \ a_{n-1}, \ a_n=c_2 \in \mathbb{Z}_n^{(X)}$$
such that, for every $1 \leq i \leq n-1$, $a_i$ and $a_{i+1}$ differ at exactly one point $p_i$. Observe that, for every $1 \leq i \leq n-1$, we have
$$\begin{array}{lcl} d((\alpha(a_i),\beta(p_i)), (\alpha(a_{i+1}),\beta(p_i))) & = & d( q(a_i,p_i),q(a_{i+1},p_i)) \\ \\ & \leq & C d( (a_i,p_i), (a_{i+1},p_i)) +K = C+K, \end{array}$$
so $\alpha(a_i)$ and $\alpha(a_{i+1})$ may only differ in the ball $B(\beta(p_i),C+K)$. It follows that $\alpha(a_1)=\alpha(c_1)$ and $\alpha(a_n)=\alpha(c_2)$ may only differ in 
$$\bigcup\limits_{i=1}^{n-1} B(\beta(p_i),C+K) = \beta( \{p_1,\ldots, p_{n-1} \})^{+C+K} = \beta(\mathrm{supp}(c_1-c_2)^{+C+K}.$$
In other words, we have proved that $\mathrm{supp}(\alpha(c_1)-\alpha(c_2))$ lies in the $(C+K)$-neighbourhood of $\beta(\mathrm{supp}(c_1-c_2))$. The same argument applied to $\bar{q}$ and $\alpha(c_1),\alpha(c_2)$ shows that $\mathrm{supp}(c_1-c_2)$ lies in the $(C+K)$-neighbourhood of $\bar{\beta}(\mathrm{supp}(\alpha(c_1)-\alpha(c_2)))$. So $\beta(\mathrm{supp}(c_1-c_2))$ must lie in the $(C(C+K)+K)$-neighbourhood of $\beta \circ \bar{\beta}(\mathrm{supp}(\alpha(c_1)-\alpha(c_2)))$, the latter being contained in the $K$-neighbourhood of $\mathrm{supp}(\alpha(c_1)-\alpha(c_2))$. Thus, we know that, conversely, $\beta(\mathrm{supp}(c_1-c_2))$ lies in the $(C(C+K)+2K)$-neighbourhood of $\mathrm{supp}(\alpha(c_1)-\alpha(c_2))$. This concludes the proof of $(iii)$.

\medskip \noindent
Conversely, assume that $(i)-(iii)$ hold and let us prove that $q$ is an aptolic quasi-isometry. Let $C,K \geq 0$ be such that $\beta$ is a $(C,K)$-quasi-isometry and such that there exists a $(C,K)$-quasi-isometry $\bar{\beta}$ with $\beta \circ \bar{\beta}$, $\bar{\beta} \circ \beta$ within $K$ from identities. Set
$$\bar{q} : \left\{ \begin{array}{ccc} \mathcal{L}_m(Y) & \to & \mathcal{L}_n(X) \\ (c,p) & \mapsto & \left( \alpha^{-1}(c) , \bar{\beta}(p) \right) \end{array} \right.,$$
and observe that 
$$q \circ \bar{q} : (c,p) \mapsto (c, \beta \circ \bar{\beta}(p)) \text{ and } \bar{q} \circ q : (c,p) \mapsto (c, \bar{\beta}\circ \beta(p))$$
are at distance $\leq K$ from identities. 

\medskip \noindent
Let $(c_1,p_1),(c_2,p_2) \in \mathcal{L}_n(X)$ be two points. Fix a path $\zeta$ of minimal length that starts from $p_1$, visits all the points in $\mathrm{supp}(c_1-c_2)$, and ends at $p_2$. Let $\xi \subset \mathcal{L}_m(Y)$ denote a concatenation of geodesics connecting any two consecutive points along $q(\zeta)$. Notice that $\xi$ has length $\leq (C+K) \mathrm{length}(\zeta)$ according to $(ii)$. By construction, $\xi$ starts from $\beta(p_1)$, visits all the points in $\beta(\mathrm{supp}(c_1-c_2))$, and ends at $\beta(p_2)$. Denoting by $L \geq 1$ the length of one path that visits all the points in a ball of radius $Q$ in $\mathcal{L}_m(Y)$ and that both starts and ends at the centre, we know that there exists a path $\eta$ of length $\leq L \cdot \mathrm{length}(\xi)$ that starts from $\beta(p_1)$, visits all the points in the $Q$-neighbourhood of $\beta(\mathrm{supp}(c_1-c_2))$, and ends at $\beta(p_2)$. Because $\mathrm{supp}(\alpha(c_1)-\alpha(c_2))$ lies in the $Q$-neighbourhood of $\beta(\mathrm{supp}(c_1-c_2))$ according to $(iii)$, it follows that
$$\begin{array}{lcl} d(q(c_1,p_1),q(c_2,p_2)) & = & d((\alpha(c_1),\beta(p_1)), (\alpha(c_2),\beta(p_2))) \leq \mathrm{length}(\eta) \leq L \cdot \mathrm{length}(\xi) \\ \\ & \leq & L(C+K) \cdot \mathrm{length}(\zeta) = L(C+K) \cdot d((c_1,p_1),(c_2,p_2)). \end{array}$$
Observe that $\bar{q}$ also satisfies $(i)-(iii)$. For $(i)$ and $(ii)$, it is clear. For $(iii)$, we know that, for all colourings $c_1,c_2 \in \mathbb{Z}_n^{(Y)}$, the Hausdorff distance between $\beta(\mathrm{supp}(\alpha^{-1}(c_1)-\alpha^{-1}(c_2)))$ and $\mathrm{supp}(c_1-c_2)$ is at most $Q$. So the Hausdorff distance between $\bar{\beta} \circ \beta(\mathrm{supp}(\alpha^{-1}(c_1)-\alpha^{-1}(c_2)))$ and $\bar{\beta}(\mathrm{supp}(c_1-c_2))$ is at most $(C+K)Q$. But the Hausdorff distance between $\bar{\beta} \circ \beta(\mathrm{supp}(\alpha^{-1}(c_1)-\alpha^{-1}(c_2)))$ and $\mathrm{supp}(\alpha^{-1}(c_1)-\alpha^{-1}(c_2))$ is at most $K$, so we conclude that the Hausdorff distance between  $\bar{\beta}(\mathrm{supp}(c_1-c_2))$ and $\mathrm{supp}(\alpha^{-1}(c_1)-\alpha^{-1}(c_2))$ is at most $(C+K)Q+K$, as desired. Therefore, by reproducing the previous argument, we show that 
$$d(\bar{q}(c_1,p_1),\bar{q}(c_2,p_2)) \leq M (C+K) \cdot d((c_1,p_1),(c_2,p_2))$$
for all $(c_1,p_1),(c_2,p_2) \in \mathcal{L}_m(Y)$, where $M$ denotes the length of one path that visits all the points in a ball of radius $(C+K)Q+K$ in $\mathcal{L}_n(X)$ and that both starts and ends at the centre. We deduce from the previous two centred inequalities that
$$\begin{array}{lcl} d(q(c_1,p_1),q(c_2,p_2)) & \geq & \displaystyle \frac{1}{M(C+K)} d( \bar{q} \circ q (c_1,p_1), \bar{q}\circ q(c_2,p_2)) \\ \\ & \geq & \displaystyle \frac{1}{M(C+K)} d((c_1,p_1),(c_2,p_2)) - \frac{2K}{M(C+K)} \end{array}$$
for all $(c_1,p_1),(c_2,p_2) \in \mathcal{L}_n(X)$; and that
$$\begin{array}{lcl} d(\bar{q}(c_1,p_1), \bar{q}(c_2,p_2)) & \geq & \displaystyle \frac{1}{L(C+K)} d(q \circ \bar{q} (c_1,p_1), q\circ \bar{q}(c_2,p_2)) \\ \\ & \geq & \displaystyle \frac{1}{L(C+K)} d((c_1,p_1),(c_2,p_2)) - \frac{2K}{L(C+K)}. \end{array}$$
Thus, $q$ is a quasi-isometry with $\bar{q}$ as a quasi-inverse, proving that $q$ is an aptolic quasi-isometry.
\end{proof}

\noindent
Now, we are ready to state the main result of this section. The key point is that it imposes restrictions on the integers $n,m$ if there exists an aptolic quasi-isometry $\mathcal{L}_n(X) \to \mathcal{L}_m(Y)$.

\begin{prop}\label{prop:general}
Let $n,m \geq 2$ be two integers, $X,Y$ two unbounded graphs, and $q: \mathcal{L}_n(X)\to \mathcal{L}_m(Y)$ an aptolic quasi-isometry, i.e. there exist a bijection $\alpha : \mathbb{Z}_n^{(X)} \to \mathbb{Z}_n^{(Y)}$ and a quasi-isometry $\beta : X \to Y$ such that $q(c,p)= (\alpha(c),\beta(p))$ for all $(c,p) \in \mathcal{L}_n(X)$. For every quasi-inverse $\bar{\beta}$ of $\beta$, there exists a constant $Q \geq 0$ such that:
\begin{itemize}
	\item[] For all subset $A_1\subset X$ and number $Q' \geq Q$, $\alpha^{-1}\left( \mathcal{L}\left( \beta(A_1)^{+Q'} \right) \right)$ is a union of cosets of $\mathcal{L}(A_1)$; conversely, for all subset $A_2\subset Y$ and number $Q' \geq Q$, 
$\alpha \left( \mathcal{L} \left( \bar{\beta}(A_2)^{+Q'} \right) \right)$ is a union of cosets of $\mathcal{L}(A_2)$. 
\end{itemize}
As a consequence, $n$ and $m$ have the same prime divisors.
 \end{prop}

\noindent
Our proof relies on the following two preliminary lemmas.

\begin{lemma}\label{lem:F(A)}
Let $n,m \geq 2$, be two integers and $X,Y$ two graphs. Assume that we are given a constant $Q \geq 0$ and two maps $\alpha : \mathbb{Z}_n^{(X)} \to \mathbb{Z}_n^{(Y)}$ and $\beta :X \to Y$ such that, for all colourings $a,b \in \mathbb{Z}_n^{(X)}$ that satisfy $\mathrm{supp}(a-b)\subset \{p\}$ for some $p \in X$, we have $\mathrm{supp}(a-b)\subset B(\beta (p),Q)$.
 Then, for every subset $A\subset X$ and every colouring $c\in \mathbb{Z}_n^{(X)}$, we have 
 \[\alpha(c+\mathcal{L}(A))\subset \alpha(c)+ \mathcal{L} \left( \beta(A)^{+Q} \right).\]
\end{lemma}

\begin{proof}
Clearly, it is enough to prove the inclusion above for all finite subsets in $A$, so we can assume without loss of generality that $A$ is finite. We argue by induction over the cardinality of $A$. The case $|A|=1$ is covered by the assumption of the lemma. Now, assume that our assertion holds for a given cardinality and for every colouring. Fix an $a\in A$ and obverse that, given two colourings $c\in \mathcal{L}(H_1)$ and $c'\in c+\mathcal{L}(A)$, we have $c'\in c''+\mathcal{L}(\{a\})$ for some $c''\in c+\mathcal{L}(A \backslash \{a\})$ since $c+F_1(A)=c+\mathcal{L}(A\backslash\{a\})+\mathcal{L}(\{a\})$. Applying the inductive assumption to $c''$, we deduce that $\alpha(c'')\in \alpha(c)+\mathcal{L}(\beta(A\backslash \{a\})^{+Q})$. And applying our assumption to $c''$ yields $\alpha(c')\in \alpha(c'')+\mathcal{L}(\{a\}^{+Q})$. Hence
\[\alpha(c')\in \alpha(c)+\mathcal{L} \left( \beta(A \backslash \{a\})^{+Q} \right) +\mathcal{L} \left( \{a\}^{+Q} \right) =\alpha (c)+\mathcal{L} \left( \beta(A\backslash \{a\})^{+Q} \cup \{\beta(a)\}^{+Q} \right).\]
And we conclude thanks to the general formula $(X \cup Y)^{+Q}=X^{+Q}\cup Y^{+Q}$.
\end{proof}

\begin{lemma}\label{lem:cosetsF(A)}
Under the assumptions of Lemma \ref{lem:F(A)}, we have that, for every colouring $c\in \mathbb{Z}_n^{(X)}$, every subset $A \subset X$, and every $Q'\geq Q$, $\alpha^{-1}(\alpha(c)+\mathcal{L}(\beta(A)^{+Q'}))$ is a union of cosets of $\mathcal{L}(A)$. In particular, its cardinality is a multiple of $m^{|A|}$.
\end{lemma}

\begin{proof}
We let $c'\in \alpha^{-1}(\alpha(c)+\mathcal{L}(\beta(A)^{+Q'})$. We deduce from Lemma \ref{lem:F(A)} that 
\[\alpha(c'+\mathcal{L}(A))\in \alpha(c')+ \mathcal{L} \left( \beta(A)^{+Q} \right) \subset \alpha(c)+\mathcal{L} \left( \beta(A)^{+Q'} \right).\]
In other words, $\alpha^{-1}(\alpha(c)+\mathcal{L}(\beta(A)^{+Q}))$ is stable by multiplication by $\mathcal{L}(A)$, so the desired conclusion follows.
\end{proof}

\begin{proof}[Proof of Proposition \ref{prop:general}.]
As a consequence of Proposition \ref{prop:AptoQI}(iii), Lemmas~\ref{lem:F(A)} and~\ref{lem:cosetsF(A)} apply. The first assertion of our proposition follows from Lemma \ref{lem:cosetsF(A)} applied to $c:= \alpha^{-1}(0)$. Our second assertion is symmetric to the first one. Finally, the conclusion on the cardinalities follows from the fact that $\alpha$ is a bijection.
\end{proof}

\noindent
As mentioned in the introduction, one of the main results of this article is that, given two integers $n,m \geq 2$ and two coarsely $1$-connected uniformly one-ended graphs $X,Y$, every quasi-isometry $\mathcal{L}_n(X) \to \mathcal{L}_m(Y)$ is at finite distance from an aptolic quasi-isometry. In the next example, we show how to construct (many) quasi-isometries that are not at finite distance from aptolic quasi-isometries as soon as we remove the assumption of being one-ended.

\begin{ex}\label{ex:NonAptolic}
Let $n \geq 2$ be an integer and $X$ a multi-ended graph. Fix a vertex $x_0 \in X$ and a ball $B(x_0,R)$ whose complement in $X$ contains at least two unbounded connected components, say $A,B \subset X$. Let $\delta$ denote the colouring taking the value $1$ at $x_0$ and $0$ elsewhere. 

\begin{prop}\label{prop:NonAptolic}
The map $\Phi : \mathcal{L}_n(X) \to \mathcal{L}_n(X)$ defined by
$$(c,x)  \mapsto  \left\{ \begin{array}{cl} (c,x) & \text{if $x \in A$ or $c_{|A} \neq 0$} \\ (c+\delta,x) & \text{if $x \notin A$ and $c_{|A}=0$} \end{array} \right.$$
is a surjective $(1,2R)$-quasi-isometry, i.e.
$$d((c_1,x_1),(c_2,x_2))-2R \leq d(\Phi(c_1,x_1), \Phi(c_2,x_2)) \leq d((c_1,x_1),(c_2,x_2))+2R$$
for all $(c_1,x_1),(c_2,x_2) \in \mathcal{L}_n(X)$, and it is not at finite distance from an aptolic quasi-isometry.
\end{prop}

\noindent
Indeed, fix two points $(c_1,x_1),(c_2,x_2) \in \mathcal{L}_n(X)$. Four cases can happen:
\begin{itemize}
	\item If $x_1 \in A$ or $c_{1|A} \neq 0$ and if $x_2 \in A$ or $c_{2|A} \neq 0$, then $\Phi(c_1,x_1)=(c_1,x_1)$ and $\Phi(c_2,x_2)=(c_2,x_2)$, so there is nothing to prove.
	\item If $x_1,x_2 \notin A$ and $c_{1|A}, c_{2|A} \neq 0$, then $$d(\Phi(c_1,x_1),\Phi(c_2,x_2))=d((c_1+\delta,x_1),(c_2+\delta,x_2))=d((c_1,x_1),(c_2,x_2)),$$ which trivially implies our inequality.
	\item Assume that $x_1 \in A$ or $c_{1|A} \neq 0$ and that $x_2 \notin A$ and $c_{2|A} = 0$. Notice that $\Phi(c_1,x_1)=(c_1,x_1)$ and $\Phi(c_2,x_2)= (c_2+\delta,x_2)$. Let $\gamma$ be a path in $H$ starting at $x_1$, visiting all the points in $\mathrm{supp}(c_1-c_2)$, ending at $x_2$, and whose length coincides with the distance between $(c_1,x_1)$ and $(c_2,x_2)$ in $H$. Observe that $\gamma$ ends at a point not in $A$, namely $x_2$, and that it intersects $A$, because either $x_1 \in A$ or $c_1$, $c_2$ differ at a point in $A$. Therefore, $\gamma$ crosses the ball $B(x_0,R)$, and we can add to $\gamma$ a loop of length $\leq 2R$ passing through $x_0$. If we denote by $\gamma'$ the path thus obtained, we deduce that $$\begin{array}{lcl} d(\Phi(c_1,x_1),\Phi(c_2,x_2))& = & d((c_1+\delta,x_1),(c_2,x_2)) \leq \mathrm{length}(\gamma') \\ \\ & \leq & \mathrm{length}(\gamma)+2R = d((c_1,x_1),(c_2,x_2))+2R. \end{array}$$
Similarly, one shows that $$d((c_1,x_1),(c_2,x_2))=d((c_1+\delta - \delta,x_1),(c_2,x_2)) \leq d((c_1+\delta,x_1),(c_2,x_2))+2R$$ as desired.
	\item If $x_2 \in A$ or $c_{2|A} \neq 0$ and if $x_1 \notin A$ and $c_{1|A} = 0$, then the configuration is symmetric the previous one.
\end{itemize}
Thus, we have proved that $\Phi$ is a quasi-isometry. Now, we verify that $\Phi$ is not at finite distance from an aptolic quasi-isometry. Indeed, let $\Psi : (c,x) \mapsto (\alpha(c),\beta(x))$ be an aptolic quasi-isometry potentially at finite distance from $\Phi$, where $\alpha : \mathbb{Z}_n^{(X)} \to \mathbb{Z}_n^{(X)}$ and $\beta : X \to X$ are two maps. Clearly, if $\beta$ is not at finite distance from the identity, then the distance between $\Phi$ and $\Psi$ is not finite, so from now on we assume that $\beta = \mathrm{Id}_X$. Fix two sequences $(a_k)$ and $(b_k)$, respectively in $A$ and $B$, that goes to infinity. If $\Phi$ and $\Psi$ are at finite distance, say $C$, then
$$\left\{ \begin{array}{l} C \geq d(\Phi(0,b_k),\Psi(0,b_k)) = d((\delta,b_k),(\alpha(0),b_k)) \\ C \geq d(\Phi(0,a_k),\Psi(0,a_k)) = d((0,a_k),(\alpha(0),a_k)) \end{array} \right.$$
for every $k \geq 0$. But the first inequality implies that $\alpha(0)=\delta$ while the second inequality implies that $\alpha(0)=0$, a contradiction. Thus, our quasi-isometry $\Phi$ cannot be at finite distance from an aptolic quasi-isometry.
\end{ex}

\subsection{Lamplighters over amenable groups}\label{section:AmenableCase}

\noindent
We saw in Section \ref{section:Aptolic} that, if there exists an aptolic quasi-isometry $\mathcal{L}_n(X) \to \mathcal{L}_m(Y)$, then $n$ and $m$ must have the same prime divisors. In this section, our goal is to show that, under the additional assumption that $X$ is amenable, this observation can be strengthened. Namely:

\begin{thm}\label{thm:Amenable}
Let $n,m \geq 2$ be two integers, $X,Y$ two unbounded graphs, and $q: \mathcal{L}_n(X)\to \mathcal{L}_m(Y)$ an aptolic quasi-isometry, i.e. there exist a bijection $\alpha : \mathbb{Z}_n^{(X)} \to \mathbb{Z}_n^{(Y)}$ and a quasi-isometry $\beta : X \to Y$ such that $q(c,p)= (\alpha(c),\beta(p))$ for all $(c,p) \in \mathcal{L}_n(X)$. If $X$ is amenable, then there exist integers $k,r,s \geq 1$ satisfying $n=k^{r}$ and $m=k^{s}$. Moreover, $\beta$ (and a fortiori $q$) is quasi-$(s/r)$-to-one.
\end{thm}

\noindent
The last assertion of our theorem relies on the following observation:

\begin{lemma}\label{lem:TransQuasi}
Let $n,m \geq 2$ be two integers, $X,Y$ two graphs, and $q : \mathcal{L}_n(X) \to \mathcal{L}_m(Y)$ an aptolic quasi-isometry, i.e. there exist a bijection $\alpha : \mathbb{Z}_n^{(X)} \to \mathbb{Z}_n^{(Y)}$ and a quasi-isometry $\beta : X \to Y$ such that $q(c,p)= (\alpha(c),\beta(p))$ for every $(c,p) \in \mathcal{L}_n(X)$. If $\beta$ is quasi-$\kappa$-to-one for some $\kappa>0$, then so is $q$. 
\end{lemma}

\begin{proof}
Let $A \subset \mathcal{L}_m(Y)$ be a finite subset. Let $\mathscr{C}$ denote the set of colourings appearing as first coordinates of elements in $A$, and, for every $c \in \mathscr{C}$, let $A_c \subset A$ denote the subset of the elements having $c$ as first coordinate. Notice that
$$|A|= \sum\limits_{c \in \mathscr{C}} |A_c| \text{ and } |q^{-1}(A)|  = \sum\limits_{c \in \mathscr{C}} |q^{-1}(A_c)| = \sum\limits_{c \in \mathscr{C}} | \beta^{-1}(B_c)|$$
where $B_c$ denotes the projection of $A_c$ onto $Y$. Because $\beta$ is quasi-$\kappa$-to-one, we have
$$\left| \kappa |A| - |q^{-1}(A)| \right| \leq \sum\limits_{c \in \mathscr{C}} \left| \kappa |A_c| - |q^{-1}(A_c)| \right| = \sum\limits_{c \in \mathscr{C}} \left| \kappa |B_c| - |q^{-1}(B_c)| \right| \leq \sum\limits_{c \in \mathscr{C}} C \cdot |\partial B_c|$$
for some constant $C\geq 0$ that does not depend on $A$. Because $\partial A$ contains the disjoint union $\bigsqcup_{c \in \mathscr{C}} \{c\} \times \partial B_c$, we have
$$\left| \kappa |A| - |q^{-1}(A)| \right| \leq C |\partial A|,$$
proving that $q$ is quasi-$\kappa$-to-one.
\end{proof}

\begin{proof}[Proof of Theorem \ref{thm:Amenable}.]
Let $C,K \geq 0$ be such that $\beta$ is a $(C,K)$-quasi-isometry admitting a quasi-inverse $\bar{\beta}$ that is also a $(C,K)$-quasi-isometry with $\bar{\beta} \circ \beta, \beta \circ \bar{\beta}$ within $K$ from identities. Up to increasing $K$, we assume that $K$ is larger than the constant $Q$ given by Proposition \ref{prop:general}.

\begin{claim}\label{claim:forAmenable}
Fix a prime $p$ and let $p_1$ (resp. $p_2$) denote the $p$-valuation of $n$ (resp. $m$). There exists a constant $M \geq 1$ such that
$$\left| |A| - \frac{p_2}{p_1} | \beta(A)^{+K}| \right| \leq M \cdot | \partial A |$$
for all finite $A \subset X$.
\end{claim}

\noindent
In order to shorten the notation, we set $B:= \bar{\beta}\left( \beta(A)^{+K} \right)^{+K}$. Observe that
$$A \subset B \subset A^{+(C+3)K}.$$
The first inclusion is justified by the fact that, for every $a \in A$, we have $d(a, \bar{\beta}(\beta(a))) \leq K$ with $\bar{\beta}(\beta(a)) \in B$. The second inclusion is justified by the fact that, for every $x \in B$, there exist $y \in H_1$ and $a \in A$ such that $d(x,\bar{\beta}(y)) \leq K$ and $d(y,\beta(a)) \leq K$, hence
$$\begin{array}{lcl} d(x,a) & \leq & \displaystyle d \left( x, \bar{\beta}(y) \right) + d \left( \bar{\beta}(y), \bar{\beta}(\beta(a)) \right) + d\left( \bar{\beta}(\beta(a)), a \right) \\ \\ & \leq & K + C d(y,\beta(a)) +K +K \leq (C+3)K, \end{array}$$
i.e. $x \in A^{+(C+3)K}$ as desired. By combining these inclusions with Fact \ref{fact:EasyTwo}, we have
\begin{equation}\label{equationFirst}
|A| \leq |B| \leq |A| + N^{(C+3)K} \cdot |\partial A|
\end{equation}
where $N \geq 3$ is a fixed integer larger than the maximal degree of a vertex in $X$. Next, notice that Proposition \ref{prop:general} implies that $\alpha^{-1}(\mathcal{L}(\beta(A)^{+K}))$ is a union of cosets of $\mathcal{L}(A)$, so the cardinality of $\alpha^{-1}(\mathcal{L}(\beta(A)^{+K}))$ must be a multiple of the cardinality of $\mathcal{L}(A)$, hence
\begin{equation}\label{equationE}
p^{p_2 |\beta(A)^{+K}|} = E \cdot p^{p_1 |A|} \text{ for some $E\geq 1$}.
\end{equation}
Similarly, Proposition \ref{prop:general} implies that $\alpha(\mathcal{L}(B))$ is a union of cosets of $\mathcal{L}(\beta(A)^{+K})$, hence
\begin{equation}\label{equationF}
p^{p_1 |B|} = F \cdot p^{p_2 |\beta(A)^{+K}|} \text{ for some $F \geq 1$}.
\end{equation}
It follows from (\ref{equationF}) that
$$|B| = \frac{1}{p_1} \log (F) + \frac{p_2}{p_1} \cdot \left| \beta(A)^{+K} \right|,$$
and the combination of (\ref{equationFirst}), (\ref{equationE}), and (\ref{equationF}) implies that
$$\log(F) \leq \log(EF) = p_1 \left( |B|-|A| \right)  \leq p_1 N^{(C+3)K)} \cdot |\partial A|,$$
so we have
\begin{equation}\label{equationSecond}
\frac{p_2}{p_1} \left| \beta(A)^{+K} \right| \leq |B| \leq \frac{p_2}{p_1} \left| \beta(A)^{+K} \right| + N^{(C+3)K} |\partial A|.
\end{equation}
The combination of (\ref{equationFirst}) and (\ref{equationSecond}) leads to the desired inequalities, concluding the proof of our claim.

\medskip \noindent
Now, let us prove that $n$ and $m$ are powers of a common number. Let $(A_k)$ be a F\o lner sequence in $X$ and let $p_1,q_1$ (resp. $p_2,q_2$) denote the valuations of $n$ (resp. $m$) with respect to two primes. By applying Claim \ref{claim:forAmenable} to $p_1,p_2$, we find that $(|\beta(A_k)^{+K}|/ |A_k|)$ converges to $p_1/p_2$ because $|\partial A_k|/ |A_k| \to 0$. Similarly, by applying Claim \ref{claim:forAmenable} to $q_1,q_2$, we find that $(|\beta(A_k)^{+K}|/ |A_k|)$ converges to $q_1/q_2$. Thus, we have proved that there exists a rational $r/s$ such that, for every prime $p$, the quotient of the $p$-valuations of $n$ and $m$ is $r/s$, which implies that $n=k^{r}$ and $m=k^{s}$ for some $k \geq 1$. This proves the first assertion of our proposition.

\medskip \noindent
From now on, we set $\kappa:= r/s$. We want to prove that $\beta$ is quasi-$\kappa$-to-one. So fix a finite $A \subset Y$. By applying Claim \ref{claim:forAmenable} to $\beta^{-1}(A)$, we get 
$$\left| \kappa |A^{+K}| - |\beta^{-1}(A)| \right| = \left| \kappa |\beta(\beta^{-1}(A))^{+K}| - |\beta^{-1}(A)| \right| \leq M \cdot |\partial \beta^{-1}(A)|.$$
According to Facts \ref{fact:EasyOne} and \ref{fact:EasyTwo}, there exists a constant $U \geq 0$ that does not depend on $A$ such that $|\partial \beta^{-1}(A)| \leq U \cdot |\partial A|$ and $|A^{+K} \backslash A| \leq U \cdot |\partial A|$. Therefore,
$$\left| \kappa |A| - |\beta^{-1}(A)| \right| \leq \left| |A^{+K}|-|A| \right|+ \left| \kappa |A^{+K}| - |\beta^{-1}(A)| \right| \leq U(M+1) \cdot |\partial A|.$$
Thus, we have proved that $\beta$ (and a fortiori $q$ according to Lemma \ref{lem:TransQuasi}) is quasi-$\kappa$-to-one, as desired. 
\end{proof}

\noindent
We conclude this section by noticing that the necessary condition provided by Theorem~\ref{thm:Amenable} for the existence an aptolic quasi-isometry is also sufficient.

\begin{prop}\label{prop:AmConverse}
Let $n,m \geq 2$ be two integers and $X,Y$ two graphs of bounded degree. Assume that $n=k^{r}$, $m=k^{s}$ for some $k,r,s \geq 1$ and that there exists a quasi-$(s/r)$-to-one quasi-isometry $X \to Y$. Then there exists an aptolic quasi-isometry $\mathcal{L}_n(X) \to \mathcal{L}_m(Y)$. 
\end{prop}

\begin{proof}
According to Proposition \ref{prop:kappa}, there exist a partition $\mathcal{P}$ (resp. $\mathcal{Q}$) of $X$ (resp. of $Y$) with uniformly bounded pieces of size $s$ (resp. $r$), a bijection $\psi:\mathcal{P}\to \mathcal{Q}$, and a quasi-isometry $\beta : X \to Y$ satisfying $\beta(P) \subset \psi(P)$ for every $P \in \mathcal{P}$. Fix a bijection $\sigma : \mathbb{Z}_n^{s} \to \mathbb{Z}_m^{r}$ satisfying $\sigma(0)=0$, and define a bijection $\alpha : \mathbb{Z}_n^{(X)} \to \mathbb{Z}_n^{(Y)}$ in such a way that $\alpha$ sends $\mathcal{L}(P)$ to $\mathcal{L}(\psi(P))$ through $\sigma$ for every $P \in \mathcal{P}$. We claim that
$$q : (c,p) \mapsto (\alpha(c), \beta(p)), \ (c,p) \in \mathcal{L}_n(X)$$
is the quasi-isometry we are looking for. Let $(c_1,p_1),(c_2,p_2) \in \mathcal{L}_n(X)$ be two points. Let $P_1, \ldots, P_n$ denote the pieces of $\mathcal{P}$ containing points in $\mathrm{supp}(c_1-c_2)$. By construction, $\psi(P_1),\ldots, \psi(P_n)$ are the pieces of $\mathcal{Q}$ containing points in $\mathrm{supp}(\alpha(c_1)-\alpha(c_2))$. Because the pieces of $\mathcal{Q}$ are uniformly bounded, the Hausdorff distance between $\mathrm{supp}(\alpha(c_1)-\alpha(c_2))$ and $\psi(P_1) \cup \cdots \cup \psi(P_n)$ is finite. We also know by construction that $\beta(\mathrm{supp}(c_1-c_2))$ lies in $\psi(P_1) \cup \cdots \cup \psi(P_n)$ and has a point in each $\psi(P_1),\ldots, \psi(P_n)$. Once again because the pieces of $\mathcal{Q}$ are uniformly bounded, we deduce that the Hausdorff dimension between $\mathrm{supp}(\alpha(c_1)-\alpha(c_2))$ and $\beta(\mathrm{supp}(c_1-c_2))$ is bounded (by a bound that does not depend on $c_1,c_2$ but only on $\mathcal{P},\mathcal{Q}$). We conclude from Proposition \ref{prop:AptoQI} that $q$ is an aptolic quasi-isometry, as desired. 
\end{proof}

\subsection{Lamplighters over non-amenable groups}

\noindent
We saw in the previous section that, if there exists an aptolic quasi-isometry $\mathcal{L}_n(X) \to \mathcal{L}_m(Y)$ where $X,Y$ are amenable, then $n$ and $m$ must be powers of a common number, strengthening the observation made in Section \ref{section:Aptolic} that $n$ and $m$ must have the same prime divisors. In this section, our goal is to prove that this phenomenon is specific to the amenable case. More precisely:

\begin{prop}\label{prop:QInonamenable}
Let $X$ be a graph of bounded degree and $n,m \geq 2$ two integers. If $X$ is non-amenable and if $n,m$ have the same prime divisors, then there exists an aptolic quasi-isometry between $\mathcal{L}_n(X)$ and $\mathcal{L}_m(X)$.
\end{prop}

\noindent
We emphasize that, in this statement, we do not assume that $X$ is coarsely $1$-connected or one-ended. For instance, $X$ can be a (bushy) tree. 

\medskip \noindent
Let us illustrate the construction we use in order to prove Proposition \ref{prop:QInonamenable} by explaining why the lamplighter groups $\mathbb{Z}_6 \wr \mathbb{F}_2$ and $\mathbb{Z}_{24} \wr \mathbb{F}_2$ are quasi-isometric, where the free group $\mathbb{F}_2$ will be thought of as the $4$-regular tree. The first trick is to replace $\mathbb{Z}_6 \wr \mathbb{F}_2$ (resp. $\mathbb{Z}_{24} \wr \mathbb{F}_2$) with $(\mathbb{Z}_3 \oplus \mathbb{Z}_2) \wr \mathbb{F}_2$ (resp. $(\mathbb{Z}_{3} \oplus \mathbb{Z}_2^3) \wr \mathbb{F}_2$). Loosely speaking, we split each lamp into two half-lamps. Formally, each colouring $c : \mathbb{F}_2 \to \mathbb{Z}_6$ (resp. $c : \mathbb{F}_2 \to \mathbb{Z}_{24}$) becomes the sum $c_1 \oplus c_2$ of two colourings $c_1 : \mathbb{F}_2 \to \mathbb{Z}_3$ and $c_2 : \mathbb{F}_2 \to \mathbb{Z}_2$ (resp. $c_2 : \mathbb{F}_2 \to \mathbb{Z}_2^3$). The second trick is to notice that, given a point at infinity $\xi \in \partial \mathbb{F}_2$, one can associate a colouring $\bar{c} : \mathbb{F}_2 \to \mathbb{Z}_2^3$ to any colouring $c : \mathbb{F}_2 \to \mathbb{Z}_2$ in the following way: for every point $p \in \mathbb{F}_2$, we define $\bar{c}(p) \in \mathbb{Z}_2^3$ thanks to the three digits in $\mathbb{Z}_2$ provided by the values taken by $c$ at the three points $p_1,p_2,p_3 \in \mathbb{F}_2$ that are separated from $\xi$ by $p$, i.e. $\bar{c}(p):= (c(p_1),c(p_2),c(p_3))$. See Figure \ref{QIwreath}. Now, we define a map $(\mathbb{Z}_3 \oplus \mathbb{Z}_6) \wr \mathbb{F}_2 \to (\mathbb{Z}_{3} \oplus \mathbb{Z}_2^3) \wr \mathbb{F}_2$ by modifying the second halves of the lamps thanks to the previous operation and by leaving the arrow and the first halves as they were, i.e. $(c_1 \oplus c_2, p) \mapsto (c_1 \oplus \overline{c_2},p)$. This map turns out to define a quasi-isometry because the modifications on the colourings are local. 
\begin{figure}
\includegraphics[trim={0 0 15cm 0},clip,width=0.5\linewidth]{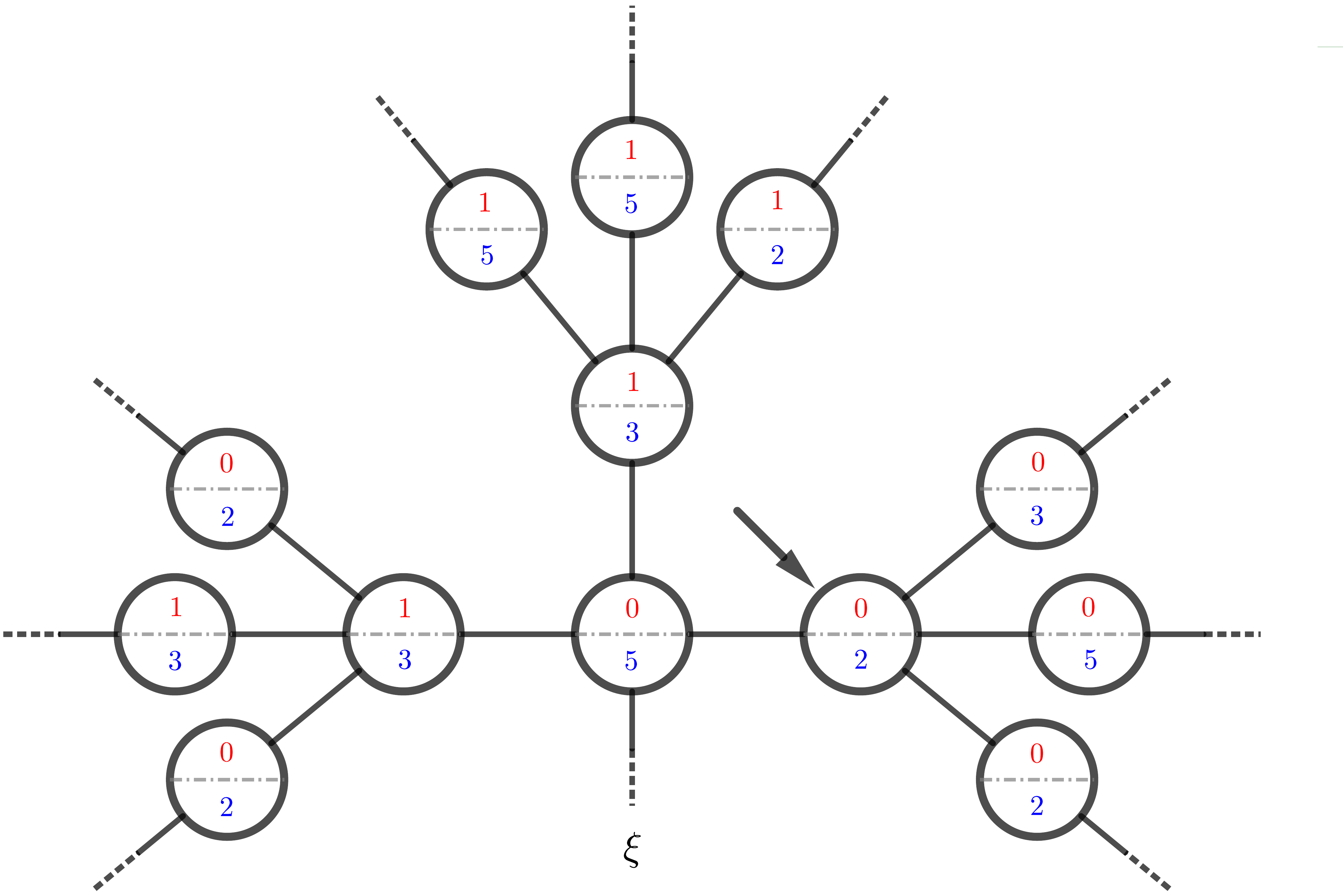}
\includegraphics[trim={0 0 15cm 0},clip,width=0.5\linewidth]{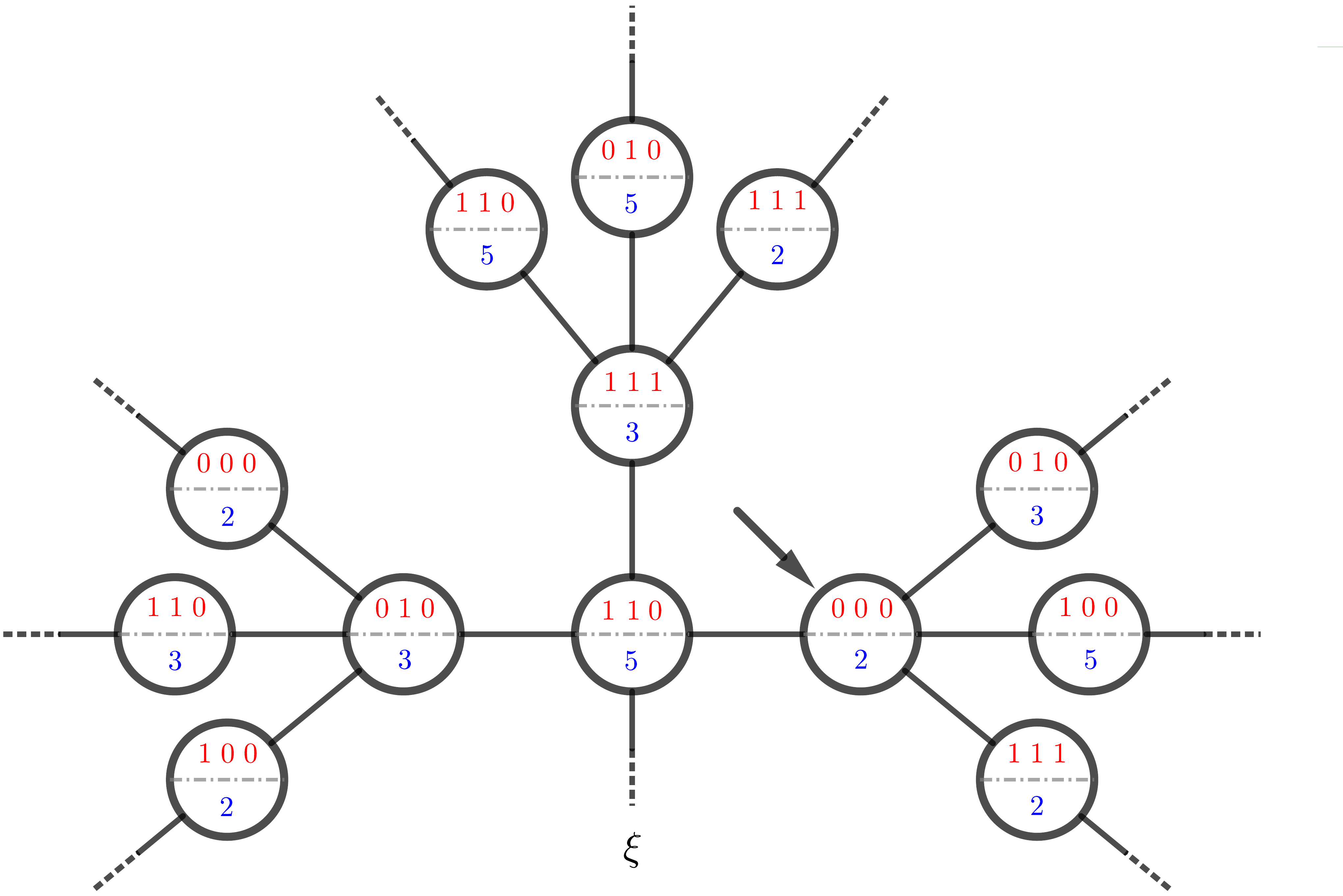}
\caption{A quasi-isometry $(\mathbb{Z}_6 \oplus \mathbb{Z}_2) \wr \mathbb{F}_2 \to (\mathbb{Z}_6\oplus \mathbb{Z}_2^3) \wr \mathbb{F}_2$.}
\label{QIwreath}
\end{figure}

\medskip \noindent
The key point in the previous construction is that there exists a $3$-to-$1$ map $\mathbb{F}_2 \to \mathbb{F}_2$ that lies at finite distance from the identity, namely the map that sends every vertex to its neighbour towards $\xi$. We generalise this idea for arbitrary non-amenable groups. 

\begin{proof}[Proof of Proposition \ref{prop:QInonamenable}.]
Given two integers $m \geq 1$, $n \geq 2$, and a prime $p$, we prove that $\mathcal{L}_{mp^n}(X)$ and $\mathcal{L}_{mp}(X)$ are quasi-isometric (through an aptolic quasi-isometry). This is sufficient to deduce our proposition.

\medskip \noindent
First of all, observe that there exists an $n$-to-$1$ map $f : X \to X$ at finite distance from the identity, i.e. there exists some $C \geq 0$ such that $d(f(x),x) \leq C$ for every $x \in X$. Indeed, as a consequence of \cite{MR1700742}, the embedding $\iota : X \hookrightarrow X \oplus \mathbb{Z}_n$ is at finite distance, say $C$, from a bijection $g : X \to X \oplus \mathbb{Z}_n$. If $p : X \oplus \mathbb{Z}_n \to X$ denotes the canonical projection, then $f:=p \circ g$ is $n$-to-$1$. Moreover,
$$d(f(x),x)=d(p(g(x)),p(\iota(x))) \leq d(g(x),\iota(x)) \leq C$$
for every $x \in X$. This proves our observation.

\medskip \noindent
From now on, we fix an enumeration of $X$ and we identify $\mathbb{Z}_{mp}$ (resp. $\mathbb{Z}_{mp^n}$) with $\mathbb{Z}_m \oplus \mathbb{Z}_p$ (resp. $\mathbb{Z}_m \oplus \mathbb{Z}_p^n$). Given a finitely supported colouring $c : X \to \mathbb{Z}_{m} \oplus \mathbb{Z}_{p}$, we construct a new finitely supported colouring $\bar{c} : X \to \mathbb{Z}_m \oplus \mathbb{Z}_p^n$ as follows. For convenience, we denote by $\pi_1$ and $\pi_2$ the projections on the first and second coordinates in both $\mathbb{Z}_m \oplus \mathbb{Z}_p$ and $\mathbb{Z}_m \oplus \mathbb{Z}_p^n$. Given an $x \in X$,
\begin{itemize}
	\item set $\pi_1(\bar{c}(x)):= \pi_1(c(x))$;
	\item enumerate $f^{-1}(x)$ as $\{x_1, \ldots, x_k\}$ by following the order of induced by our enumeration of $X$, and set $\pi_2(\bar{c}(x))=( \pi_2(c(x_1)), \ldots, \pi_2(c(x_k)))$.
\end{itemize}
We claim that
$$\Phi : \left\{ \begin{array}{ccc} \mathcal{L}_{mp}(X) & \to \mathcal{L}_{mp^n}(X) \\ (c,h) & \mapsto (\bar{c},h) \end{array} \right.$$
is a quasi-isometry. In the rest of the proof, our lamplighter graphs are endowed with diligent metrics. So fix two finitely supported colourings $c_1,c_2 : X \to \mathbb{Z}_m \oplus \mathbb{Z}_p$ and two points $k_1,k_2 \in X$. 

\medskip \noindent
Notice that, if $\bar{c}_1(x) \neq \bar{c}_2(x)$ for some $x \in X$, then either $\pi_1(c_1(x)) \neq \pi_1(c_2(x))$ (hence $x \in \mathrm{supp}(c_1-c_2)$) or $\pi_2(c_1(x')) \neq \pi_2(c_2(x'))$ (hence $x' \in \mathrm{supp}(c_1-c_2)$) for some $x' \in f^{-1}(x)$. Consequently, 
\begin{equation}\label{First}
\mathrm{supp}(\bar{c}_1-\bar{c}_2) \subset \mathrm{supp}(c_1-c_2) \cup f( \mathrm{supp}(c_1-c_2)).
\end{equation}
Next, if $c_1(x) \neq c_2(x)$ for some $x \in X$, then either $\pi_1(c_1(x)) \neq \pi_1(c_2(x))$, hence $\bar{c}_1(x) \neq \bar{c}_2(x)$; or $\pi_2(c_1(x)) \neq \pi_2(c_2(x))$, hence $\bar{c}_1(f(x)) \neq \bar{c}_2(f(x))$. Consequently,
\begin{equation}\label{Second}
\mathrm{supp}(c_1-c_2) \subset \mathrm{supp}(\bar{c}_1-\bar{c}_2) \cup f^{-1}(\mathrm{supp}(\bar{c}_1-\bar{c}_2) ).
\end{equation}
Now, fix a path $\alpha$ in $X$ that starts from $k_1$, that visits all the points in $\mathrm{supp}(c_1-c_2)$, that ends at $k_2$, and such that the length of $\alpha$ coincides with the distance between $(c_1,k_1)$ and $(c_2,k_2)$ in $\mathcal{L}_{mp}(X)$. For every point of $x \in \mathrm{supp}(c_1-c_2)$, we add to $\alpha$ a loop of length $\leq 2C$ based at $x$ and passing through $f(x)$. Thus, we obtain a new path $\alpha'$ that visits all the points in $\mathrm{supp}(c_1-c_2) \cup f( \mathrm{supp}(c_1-c_2))$ and whose length is at most 
$$\mathrm{length}(\alpha)+2C |\mathrm{supp}(c_1-c_2)| \leq (2C+1) \mathrm{length}(\alpha).$$
It follows from the inclusion (\ref{First}) that
$$d((\bar{c}_1,k_1),(\bar{c}_2,k_2)) \leq \mathrm{length}(\alpha') \leq (2C+1) d((c_1,k_1),(c_2,k_2)).$$
Next, fix a path $\beta$ in $X$ that starts from $k_1$, that visits all the points in $\mathrm{supp}(\bar{c}_1-\bar{c}_2)$, that ends at $k_2$, and such that the length of $\beta$ coincides with the distance between $(\bar{c}_1,k_1)$ and $(\bar{c}_2,k_2)$ in $\mathcal{L}_{mp^n}(X)$. For every point $x \in \mathrm{supp}(\bar{c}_1-\bar{c}_2)$ and every $x' \in f^{-1}(x)$, we add to $\beta$ a loop of length $\leq 2C$ based at $x$ and passing through $x'$. Thus, we obtain a new path $\beta'$ that visits all the points in $\mathrm{supp}(\bar{c}_1-\bar{c}_2) \cup f^{-1}(\mathrm{supp}(\bar{c}_1-\bar{c}_2) )$ and whose length is at most
$$\mathrm{length}(\beta) + 2nC |\mathrm{supp}(\bar{c}_1-\bar{c}_2)| \leq (2nC+1) \mathrm{length}(\beta).$$
It follows from the inclusion (\ref{Second}) that
$$d((c_1,k_1),(c_2,k_2)) \leq \mathrm{length}(\beta') \leq (2nC+1) d((\bar{c}_1,k_1),(\bar{c}_2,k_2)).$$
This concludes the proof that $\Phi$ is a quasi-isometry.
\end{proof}

\section{Leaf-preserving quasi-isometries are aptolic}\label{section:CosetAptolic}

\noindent
In this section, our goal is to characterise quasi-isometries that lie at finite distance from aptolic quasi-isometries. For this purpose, we introduce some terminology.

\begin{definition}
Let $n,m \geq 2$ be two integers, $X,Y$ two graphs, $q : \mathcal{L}_n(X) \to \mathcal{L}_m(Y)$ a quasi-isometry, and $\bar{q} : \mathcal{L}_m (Y) \to \mathcal{L}_n(X)$ a quasi-inverse. Then $q$ is \emph{leaf-preserving} if there exists some $C \geq 0$ such that $q$ (resp. $\bar{q}$) sends every leaf of $\mathcal{L}_n(X)$ (resp. of $\mathcal{L}_m(Y)$) at Hausdorff distance $\leq C$ from a leaf in $\mathcal{L}_m(Y)$ (resp. in $\mathcal{L}_n(X)$). 
\end{definition}

\noindent
An alternative characterisation of leaf-preserving quasi-isometries is:

\begin{lemma}
Let $n,m \geq 2$ be two integers, $X,Y$ two graphs, and $q : \mathcal{L}_n(X) \to \mathcal{L}_m(Y)$ a quasi-isometry. Then $q$ is leaf-preserving if and only if there exist a bijection $\alpha : \mathbb{Z}_n^{(X)} \to \mathbb{Z}_m^{(Y)}$ and a constant $C \geq 0$ such that, for every $c \in \mathbb{Z}_n^{(X)}$, the Hausdorff distance between $q(X(c))$ and $Y(\alpha(c))$ is $\leq C$.
\end{lemma}

\begin{proof}
Fix a quasi-inverse $\bar{q}$ of $q$. If $q$ is leaf-preserving, then there exists some $C \geq 0$ such that, for every $c \in \mathbb{Z}_n^{(X)}$, there exists some $\alpha(c) \in \mathbb{Z}_m^{(Y)}$ such that the Hausdorff distance between $q(X(c))$ and $Y(\alpha(c))$ is $\leq C$. Similarly, for every $c \in \mathbb{Z}_m^{(Y)}$ there exists some $\bar{\alpha}(c) \in \mathbb{Z}_n^{(X)}$ such that the Hausdorff distance between $\bar{q}(Y(c))$ and $X(\bar{\alpha}(c))$ is $\leq C$ (up to increasing $C$ if necessary). Because the Hausdorff distance between two distinct leaves in $\mathcal{L}_n(X)$ and $\mathcal{L}_m(Y)$ is infinite, we must have $\alpha \circ \bar{\alpha} = \mathrm{id}$ and $\bar{\alpha} \circ \alpha= \mathrm{id}$. Thus, we have proved that there exists a bijection $\alpha : \mathbb{Z}_n^{(X)} \to \mathbb{Z}_m^{(Y)}$ and a constant $C \geq 0$ such that, for every $c \in \mathbb{Z}_n^{(X)}$, the Hausdorff distance between $q(X(c))$ and $Y(\alpha(c))$ is $\leq C$. Conversely, if $q$ satisfies this property, then, for every $c_1 \in \mathbb{Z}_n^{(X)}$ (resp. $c_2 \in \mathbb{Z}_m^{(Y)}$), the Hausdorff distance between $q(X(c_1))$ and $Y(\alpha(c_1))$ (resp. $\bar{q}(Y(c_2))$ and $X(\alpha^{-1}(c_2))$) is $\leq C$. In other words, $q$ is leaf-preserving.
\end{proof}

\noindent
The objective of this section is to prove the following criterion:

\begin{thm}\label{thm:CosetAptolic}
Let $n,m \geq 2$ be two integers, $X,Y$ two graphs of bounded degree, and $q : \mathcal{L}_n(X) \to \mathcal{L}_m(Y)$ a quasi-isometry. The following statements are equivalent:
\begin{itemize}
	\item[(i)] $q$ lies at finite distance from an aptolic quasi-isometry;
	\item[(ii)] $q$ is leaf-preserving, i.e. there exists a bijection $\alpha : \mathbb{Z}_n^{(X)} \to \mathbb{Z}_m^{(Y)}$ and a constant $C \geq 0$ such that, for every $c \in \mathbb{Z}_n^{(X)}$, the Hausdorff distance between $q(X(c))$ and $Y(\alpha(c))$ is $\leq C$.
\end{itemize}
Moreover, if $(ii)$ holds, then the distance from $q$ to an aptolic quasi-isometry is bounded above by a constant that depends only on the integers $n,m$, the graphs $X,Y$, the parameters of $q$, and the constant $C$. 
\end{thm}

\noindent
The theorem is essentially a straightforward consequence of the following statement. Roughly speaking, Proposition~\ref{prop:Height} shows that leaf-preserving quasi-isometries $\mathcal{L}_n(X) \to \mathcal{L}_m(Y)$ preserve the projections onto $X$ and $Y$, and the proof below shows that the latter property amounts to being aptolic.

\begin{prop}\label{prop:Height}
Let $n,m \geq 2$ be two integers, $X,Y$ two graphs of bounded degree, and $q : \mathcal{L}_n(X) \to \mathcal{L}_m(Y)$ a quasi-isometry. Let $\pi_X$ (resp. $\pi_Y$) denote the canonical projection $\mathcal{L}_n(X) \twoheadrightarrow X$ (resp. $\mathcal{L}_m(Y) \twoheadrightarrow Y$). Assume that 
\begin{description}
	\item[($q$ leaf-preserving)] there exists a bijection $\alpha : \mathbb{Z}_n^{(X)} \to \mathbb{Z}_m^{(Y)}$ and a constant $C \geq 0$ such that, for every $c \in \mathbb{Z}_n^{(X)}$, the Hausdorff distance between $q(X(c))$ and $Y(\alpha(c))$ is $\leq C$. 
\end{description}
Then there exists a constant $Q \geq 0$ such that, for all $x,y \in \mathcal{L}_n(X)$, if $\pi_X(x)=\pi_X(y)$ then $d(\pi_Y(q(x)),\pi_Y(q(y))) \leq Q$. 
Moreover, $Q$ depends only on the integers $n,m$, the graphs $X,Y$, the parameters of $q$, and the constant $C$. 
\end{prop}

\begin{proof}[Proof of Theorem \ref{thm:CosetAptolic} assuming Proposition \ref{prop:Height}.]
The implication $(i) \Rightarrow (ii)$ is clear. Conversely, assume that $(ii)$ holds. Up to replacing $q$ with a new quasi-isometry at finite distance, we suppose without loss of generality that $q(c,p) \in Y(\alpha(c))$ for every $(c,p) \in \mathcal{L}_n(X)$. For every $p \in X$, let $\beta(p) \in Y$ be such that $q(0,p)=(\alpha(0),\beta(p))$; and set
$$\tilde{q} : (c,p) \mapsto (\alpha(c),\beta(p)), \ (c,p) \in \mathcal{L}_n(X).$$
Observe that, for every $(c,p) \in \mathcal{L}_n(X)$, we have
$$d(q(c,p),\tilde{q}(c,p)) = d \left( \pi_Y (q(c,p)), \beta(p) \right) = d \left( \pi_Y(c,p), \pi_Y(0,p) \right) \leq Q$$
where the first equality holds because $q(c,p),\tilde{q}(c,p)$ both belong to the same leaf, namely $Y(\alpha(c) )$. Therefore, $\tilde{q}$ lies at finite distance from $q$. Since $\beta$ is a quasi-isometry according to Fact \ref{fact:Beta}, it admits a quasi-inverse $\bar{\beta}$, and we can define
$$Q : (c,p) \mapsto \left( \alpha^{-1}(c), \bar{\beta}(p) \right), \ (c,p) \in \mathcal{L}_m(Y).$$
Clearly, $\tilde{q} \circ Q$ and $Q \circ \tilde{q}$ lie at finite distances from identities. Because $\tilde{q}$ is a quasi-isometry, this implies that $Q$ is also a quasi-isometry and a quasi-inverse of $\tilde{q}$. Thus, we have proved that $q$ is at finite distance from $\tilde{q}$, which is an aptolic quasi-isometry.
\end{proof}

\noindent
The proof of Proposition \ref{prop:Height}, in Section \ref{section:BigProof}, is rather technical but it relies on a simple idea. Therefore, for the reader's convenience, we begin with a general discussion in the next section.

\subsection{Warm up}

\noindent
Fix an integer $n \geq 2$ and a graph $X$. The geometric trick, in order to prove that two points in $\mathcal{L}_n(X)$ have close projections onto $X$, is that, if $X_0,X_1,X_2,X_3$ are four leaves such that $d(X_i,X_{i+1})$ is very small and $d(X_i,X_{i+2})$ very large ($i \in \mathbb{Z}_4$), then a point close to $X_0$ and $X_1$ must have its projection onto $X$ close to the projection of a point close to $X_2$ and $X_3$. A justification is the following. Consider a loop as illustrated by Figure \ref{B}. In terms of colourings, travelling along the loop corresponds to the following:
\begin{itemize}
	\item from $a$ to $a'$, one modifies the colouring of $a$ in a small region $A$ around the arrow;
	\item from $a'$ to $b$, one moves the arrow far away without modifying the colouring;
	\item from $b$ to $b'$, one modifies the colouring of $b$ in a small region $B$ around the arrow;
	\item from $b'$ to $c$, one moves the arrow far away without modifying the colouring;
	\item from $c$ to $c'$, one modifies the colouring of $c$ in a small region $C$ around the arrow;
	\item from $c'$ to $d$, one moves the arrow far away without modifying the colouring;
	\item from $d$ to $d'$, one modifies the colouring of $d$ in a small region $D$ around the arrow;
	\item from $d'$ to $a$, one moves the arrow far away without modifying the colouring.
\end{itemize}
Because we are following a loop, at some point we have to undo what we have done on the colouring. But the regions $A$ and $B$ are far away from each other, so we cannot undo anything between $a$ and $c$. Consequently, we have to undo in $C$ what we did in $A$ and undo in $D$ what we did in $B$. Since the regions $A$ and $C$ must be more or less the same, the projections onto $X$ must be approximatively the same between $a,a'$ and between $c,c'$. 

\medskip \noindent
However, considering paths of leaves of length $4$ is not sufficient in order to be able to identify any two points with the same projection onto $X$. We want to generalise the previous observation to longer paths. However, a naive extension does not work, as shown by Figure \ref{B}. The point is that the sequence $X_0,X_1,X_2,X_3$ corresponds to modifying the colouring in well separated regions $A,B,C$, but next the sequence $X_3,X_4,X_5,X_0$ undo what we did in a different order: $B$, $C$, and finally $A$. The fact that the points $x,y$ have very different projections onto $X$ comes from the fact that going from $X_5$ to $X_0$ corresponds to modifying the colouring inside the region $A$ while going from $X_2$ to $X_3$ corresponds to modifying the colouring inside a very different region, namely $C$. In order to avoid such a phenomenon, we require that the sequences $X_0,\ldots, X_3$ and $X_2,\ldots, X_5$ cannot be shortened when thought of as path of leaves. In the configuration illustrated by Figure \ref{B}, notice that $X_2, \ldots, X_5$ can be shortened as $X_2,X_5$. 
\begin{figure}
\begin{center}
\includegraphics[width=0.42\linewidth]{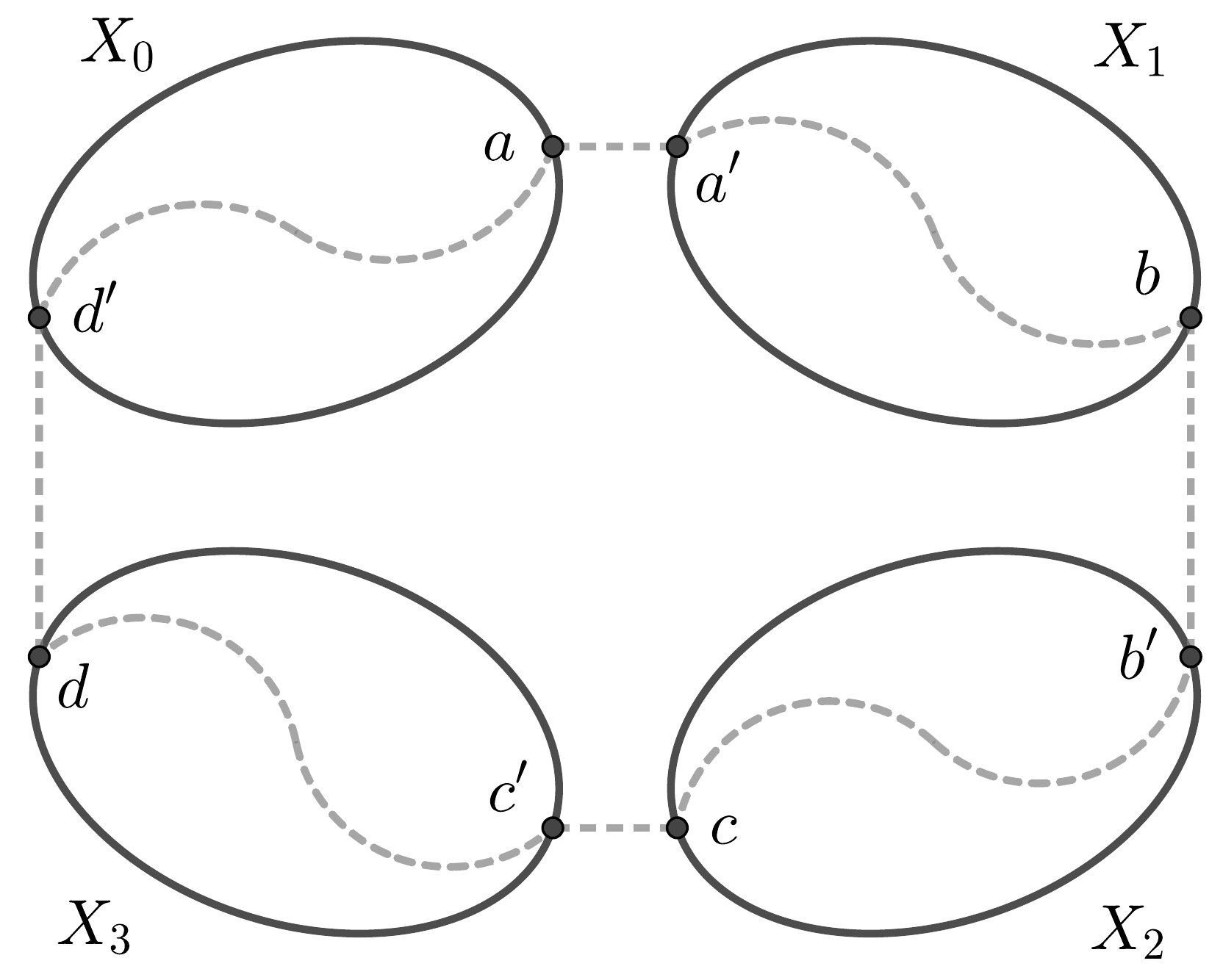} \hspace{0.8cm}
\includegraphics[width=0.5\linewidth]{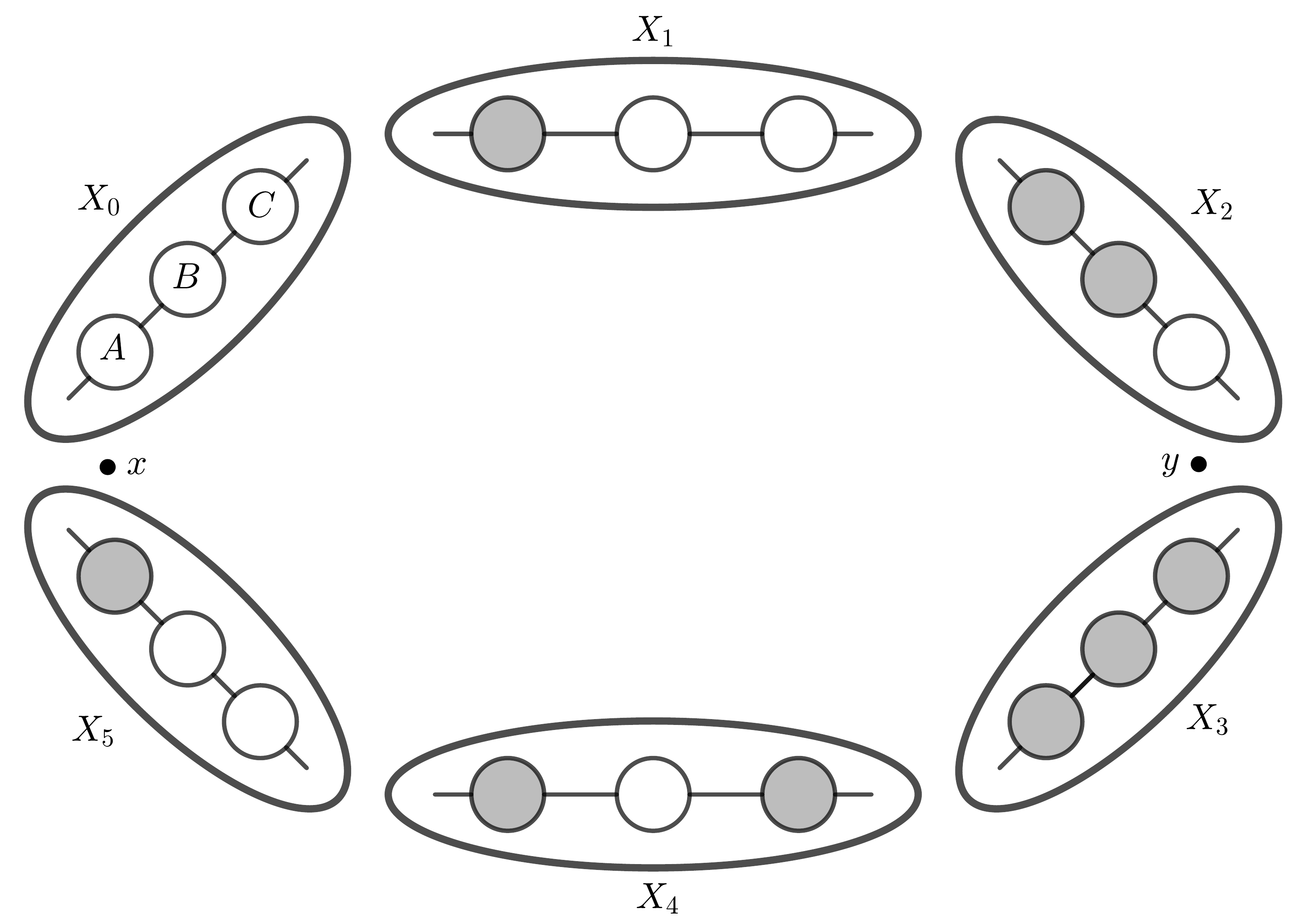}
\caption{On the left, the points $a$ and $c$ have close projections; on the right, the points $x$ and $y$ have very different projections.}
\label{B}
\end{center}
\end{figure}

\medskip \noindent
This is our fundamental tool in order to recognize two points with close projections onto $X$ : Let $X_0, \ldots, X_{2s-1}$ be a sequence of leaves of $X$ such that $d(X_i,X_{i+1})$ is very small for every $i \in \mathbb{Z}_{2s}$, such that $d(X_i,X_{i+2})$ is very large for every $i \in \mathbb{Z}_{2s}$, and such that $X_0, \ldots, X_s$ and $X_{n-1}, \ldots, X_{2s-1}$ cannot be shortened. Then a point close to $X_0$ and $X_{2s}$ must have its projection onto $X$ close to the projection of a point close to $X_{s-1}$ and $X_n$. See Lemma \ref{lem:CircleHeight} for a precise statement.

\medskip \noindent
Conversely, it is however not true that, if two points $x,y \in \mathcal{L}_n(X)$ have the same projection onto $X$, then there exists such a cycle of leaves $X_0, \ldots, X_{2s-1}$ such that $x$ is close to $X_0,X_{2s-1}$ and $y$ close to $X_{s-1},X_s$. But it is almost true. It turns out that there exists a constant $N \geq 1$ such that, for any two points $x,y \in \mathcal{L}_n(X)$ with the same projection onto $X$, we can find a sequence of points $x_1=x, \ x_2, \ldots, \ x_{N-1}, \ x_N=y$ such that, for every $1 \leq i \leq N-1$, there exists a cycle of leaves $X_0, \ldots, X_{2s-1}$ such that $x_i$ is close to $X_0,X_{2s-1}$ and $x_{i+1}$ close to $X_{s-1},X_s$. See Lemma \ref{lem:ChainCircles} for a precise statement.

\medskip \noindent
Thus, we are able to recognize geometrically when two points in $\mathcal{L}_n(X)$ have close projections onto $X$.

\subsection{Proof of the criterion}\label{section:BigProof}

\noindent
Until the proof of Proposition \ref{prop:Height}, we fix an integer $n \geq 2$ and a graph $X$ all of whose vertices have $\leq D$ neighbours. We denote by $\pi$ the canonical projection $\mathcal{L}_n(X) \twoheadrightarrow X$ and we endow the lamplighter graph with its diligent metric (see Section~\ref{section:prel}). 

\begin{definition}
Let $A,A',B,B' \geq 0$ be four real numbers. An \emph{$(A,B)$-line of leaves} is a sequence $X(c_1), \ldots, X(c_k)$ with $c_1, \ldots, c_k \in \mathbb{Z}_n^{(X)}$ such that:
\begin{enumerate}
	\item $d(X(c_i),X(c_{i+1})) \leq A$ for every $1 \leq i \leq k-1$;
	\item $d(X(c_i),X(c_{i+2})) \geq B$ for every $1 \leq i \leq k-2$.
\end{enumerate}
It is \emph{$(A',B')$-geodesic} if, for all $1 \leq i,j \leq k$, every $(A',B')$-line of leaves from $X(c_i)$ to $X(c_j)$ has length at least $|i-j|$. 
\end{definition}

\begin{definition}
Let $A,B \geq 0$ be two real numbers. An \emph{$(A,B)$-circle of leaves} is a sequence $X(c_0), \ldots, X(c_{k-1})$ with $c_1, \ldots, c_k \in \mathbb{Z}_n^{(X)}$ such that:
\begin{enumerate}
	\item $d(X(c_i),X(c_{i+1})) \leq A$ for every $i \in \mathbb{Z}_k$;
	\item $d(X(c_i),X(c_{i+2})) \geq B$ for every $i \in \mathbb{Z}_k$.
\end{enumerate}
\end{definition}

\noindent
Our proof of Proposition \ref{prop:Height} relies on two preliminary lemmas. The first one is a necessary condition for two points of $\mathcal{L}_n(X)$ to have close projections onto $X$.

\begin{lemma}\label{lem:CircleHeight}
Let $A,B,A',B' \geq 0$ be four real number satisfying
$$A<B, \ A' \geq 2D^{3A} \text{ and } B' \leq B-2D^A.$$
Let $X(c_0), \ldots, X(c_{2s-1})$ be an $(A,B)$-circle of leaves. Assume that both $X(c_0), \ldots, X(c_{s})$ and $X(c_{s-1}), \ldots, X(c_{2s-1})$ are $(A',B')$-geodesic, and fix a point $x$ (resp. $y$) within $A$ to both $X(c_0)$ and $X(c_{2s-1})$ (resp. $X(c_{s-1})$ and $X(c_s)$). Then $\pi(x)$ and $\pi(y)$ are within $14A$ in $X$.
\end{lemma}

\noindent
Before proving Lemma \ref{lem:CircleHeight}, we record two elementary observations. The first one shows that two leaves indexed by two colourings that differ in a small region must be close to each other.

\begin{fact}\label{fact:DistBall}
If two colourings $c_1,c_2 \in \mathbb{Z}_n^{(X)}$ only differ in a ball of radius $r$ in $X$, then the inequality $d(X(c_1),X(c_2)) \leq 2 D^r$ holds in $\mathcal{L}_n(X)$.
\end{fact}

\begin{proof}
Let $B(x,r)$ be a ball of radius $r$ in $X$ such that $c_1$ and $c_2$ only differ in $B(x,r)$. Fix a point $p \in B(x,r)$ and a maximal spanning tree $T \subset B(x,r)$. There exists a path $\gamma$ in $T$ that starts from $p$, that visits each vertex in $T$, and that passes through each edge at most twice. Let $q \in B(x,r)$ denote the final vertex of $\gamma$. Notice that
$$\mathrm{length}(\gamma) \leq 2 |E(T)| \leq 2|V(T)| = 2 | B(x,r)| \leq 2 D^r.$$
One can construct a path from $(c_1,p)$ to $(c_2,q)$ by moving the arrow from $p$ to $q$ following $\gamma$, and along the way we modify the colouring at each point where $c_1$ and $c_2$ differ from the value of $c_1$ to the value of $c_2$. The length of such a path coincides with $\mathrm{length}(\gamma)$.
Thus, we have constructed a path of length $\leq 2D^r$ from a point of $X(c_1)$ to a point of $X(c_2)$, concluding the proof of our fact.
\end{proof}

\noindent
Our second observation aims to show that, if a vertex is near to two close leaves, then it must also be close to any two vertices minimising the distance between the leaves.

\begin{fact}\label{fact:DistCoset}
Fix two distinct colourings $c_1,c_2 \in \mathbb{Z}_n^{(X)}$ and two points $a_1 \in X(c_1)$, $a_2 \in X(c_2)$. Then the inequality
$$d(x,a_1),d(x,a_2) \leq 2(d(x,X(c_1))+d(x,X(c_2))) + d(a_1,a_2)+d(X(c_1),X(c_2))$$
holds for every $x \in \mathcal{L}_n(X)$.
\end{fact}

\begin{proof}
Fix two points $(c_1,p)$ and $(c_2,q)$ such that $d(x,(c_1,p))=d(x,X(c_1))$ and similarly $d(x,(c_2,q))=d(x,X(c_2))$. Also, write $a_i = (c_i,f_i)$ for some $f_i \in X$, $i=1,2$. For convenience, set $\Delta:= \mathrm{supp}(c_2-c_1)$. The distance in $\mathcal{L}_n(X)$ between $(c_1,p)$ and $(c_2,q)$ coincides with the minimal length of a path in $X$ that starts from $p$, ends at $q$, and passes through all the points in $\Delta$. Consequently,
$$\begin{array}{lcl} d(x,X(c_1))+d(x,X(c_2)) & \geq & d((c_1,p),(c_2,q)) \geq d(p,\Delta) \\ \\ & \geq & d(p,f_1)-d(f_1,\Delta)- \mathrm{diam}(\Delta). \end{array}$$
Observe that we also have $d(f_1,\Delta) \leq d((c_1,f_1),(c_2,f_2))=d(a_1,a_2)$, because the distance in $\mathcal{L}_n(X)$ between $(c_1,f_1)$ and $(c_2,f_2)$ is at least the length of a path in $X$ from $f_1$ to $f_2$ passing through all the points in $\Delta$; and $\mathrm{diam}(\Delta) \leq d(X(c_1),X(c_2))$, because the distance in $\mathcal{L}_n(X)$ between $X(c_1)$ and $X(c_2)$ is at least the length of a path in $X$ visiting all the points in $\Delta$. Therefore,
$$\begin{array}{lcl} d(x,X(c_1))+d(x,X(c_2)) & \geq & d(p,f_1)- d(a_1,a_2)-d(X(c_1),X(c_2)) \\ \\ & \geq & d(a_1,(c_1,p))-d(a_1,a_2)-d(X(c_1),X(c_2)). \end{array}$$
Thus, we have proved that
$$d(a_1,(c_1,p)) \leq d(x,X(c_1))+d(x,X(c_2)) + d(a_1,a_2)+d(X(c_1),X(c_2)).$$
We conclude that
$$\begin{array}{lcl} d(x,a_1) & \leq & d(x,(c_1,p)) + d((c_1,p),a_1) \\ \\ & \leq & 2d(x,X(c_1))+d(x,X(c_2)) + d(a_1,a_2)+d(X(c_1),X(c_2)) \end{array}$$
as desired. The inequality for $d(x,a_2)$ is obtained similarly.
\end{proof}

\begin{proof}[Proof of Lemma \ref{lem:CircleHeight}.]
For every $i \in \mathbb{Z}_{2s}$, we fix two points $(c_i,q_i) \in X(c_i)$ and $(c_{i+1},p_{i+1}) \in X(c_{i+1})$ such that $d((c_i,q_i),(c_{i+1},p_{i+1})) \leq A$. This inequality implies that $d(q_i,p_{i+1}) \leq A$ and that the colourings $c_{i}$ and $c_{i+1}$ may only differ in $B(q_{i},A)$, i.e. $c_{i+1}=c_i+ L_i$ where $L_i$ is supported in $B(q_{i},A)$. 
\begin{figure}
\begin{center}
\includegraphics[width=\linewidth]{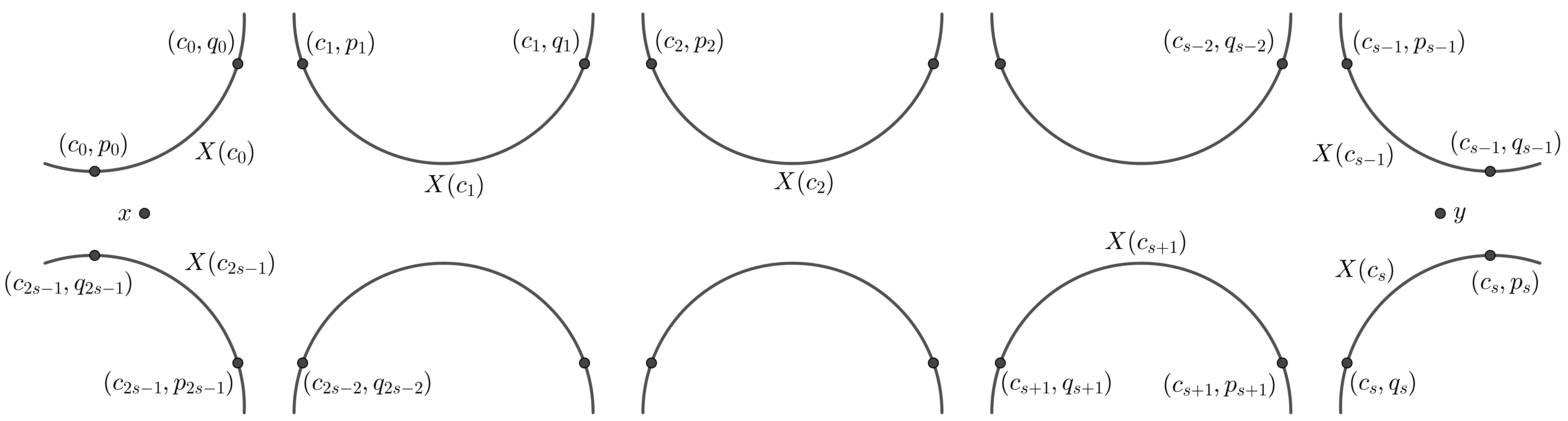}
\caption{Configuration from the proof of Lemma \ref{lem:CircleHeight}.}
\label{Conf}
\end{center}
\end{figure}

\begin{claim}\label{claim:InInf}
For all $0 \leq i < j \leq s-1$, we have $d(q_i,q_j) > 2A$.
\end{claim}

\noindent
Assume towards a contradiction that there exist two indices $0 \leq i< j \leq s-1$ such that $d(q_i,q_j) \leq 2A$. Now, observe that $L_i+L_j$ is supported in the ball $B(q_i,3A)$, hence 
$$d(X(c_i), X(c_i+L_j+L_i)) \leq 2 D^{3A}$$
according to Fact \ref{fact:DistBall}; also, observe that
$$\begin{array}{l} d(X(c_i+L_j+L_i+ \cdots +L_k),X(c_i+L_j+L_i+ \cdots +L_{k+1})) \\ \hspace{2cm} \begin{array}{cl} = & d(X(c_i+L_i+ \cdots +L_k),X(c_i+L_i+ \cdots +L_{k+1})) \\ = & d(X(c_k,X(c_{k+1})) \leq A \end{array} \end{array}$$
for every $i \leq k \leq j-1$; next that
$$\begin{array}{l} d(X(c_i),X(c_i+L_j+L_i+L_{i+1}) \\ \hspace{1cm} \begin{array}{cl} \geq & d(X(c_i),X(c_i+L_i+L_{i+1}))- d(X(c_i+L_i+L_{i+1}),X(c_i+L_j+L_i+L_{i+1})) \\ \geq & d(X(c_i),X(c_{i+2}))- d(H, X(L_j)) \geq B- 2 D^A \end{array}  \end{array}$$
where the last inequality if justified by Fact \ref{fact:DistBall}; and finally that
$$\begin{array}{l} d(X(c_i+L_j+L_i+\cdots + L_k),X(c_i+L_j+L_i+\cdots + L_{k+2})) \\ \hspace{2cm} \begin{array}{cl} = & d(X(c_i+L_i+\cdots + L_k),X(c_i+L_i+\cdots + L_{k+2})) \\ = & d(X(c_k),X(c_{k+2})) \geq B \end{array} \end{array}$$
for every $i \leq k \leq j-2$. It follows from all these inequalities that the sequence
$$X(c_i), X(c_i+L_j+L_i), X(c_i+L_j+L_i+L_{i+1}), \ldots, X(c_i+L_j+ L_i + \cdots + L_{j-1})$$
defines a $\left( 2D^{3A},B-2D^A \right)$-line of leaves from $X(c_i)$ to $X(c_{j+1})$ of length $j-i$. Since $A' \geq 2D^{3A}$ and $B' \leq B-2D^A$, we find a contradiction with the fact that $X(c_0), \ldots, X(c_{s})$ is $(A',B')$-geodesic. This completes the proof of our claim.

\begin{claim}\label{claim:InInfSym}
For all $s-1 \leq i < j \leq 2s-2$, we have $d(q_i,q_j) > 2A$.
\end{claim}

\noindent
Claim \ref{claim:InInfSym} is symmetric to Claim \ref{claim:InInf} by replacing the sequence $X(c_0),\ldots,X(c_s)$ with the sequence $X(c_{s-1}),\ldots, X(c_{2s-1})$. 

\begin{claim}\label{claim:InSup}
For every $0 \leq i \leq s-1$, there exists some $s \leq j \leq 2s-1$ such that $d(q_i,q_j) \leq 2A$.
\end{claim}

\noindent
Assume towards a contradiction that there exists some $0 \leq i \leq s-1$ such that, for every $s \leq j \leq 2s-1$, we have $d(q_i,q_j)> 2A$. We also know from Claim \ref{claim:InInf} that $d(q_i,q_j)>2A$ for every $0 \leq j \leq s-1$ distinct from $i$. Consequently, 
$$\mathrm{supp}(L_i) \cap \mathrm{supp}(L_j) \subset B(q_i,A) \cap B(q_j,A) = \emptyset$$
for every $0 \leq j \leq 2s-1$ distinct from $i$. Necessarily, $c_0$ and $c_0+L_0+ \cdots + L_{2s-1}$ must differ on $\mathrm{supp}(L_i)$. But, by construction, $c_0+L_0 + \cdots + L_{2s-1}=c_0$, hence $\mathrm{supp}(L_i)= \emptyset$. In other words, we have proved that $c_i=c_{i-1}+L_i=c_{i-1}$. It follows that
$$A \geq d(X(c_i),X(c_{i+1})) = d(X(c_{i-1}),X(c_{i+1})) \geq B,$$
a contradiction. This complete the proof of our claim.

\medskip \noindent
We are finally ready to conclude the proof of our lemma. According to Claim \ref{claim:InSup}, there exists some $s \leq i \leq 2s-1$ such that $d(q_{s-1},q_i) \leq 2A$. Claim \ref{claim:InInfSym} imposes that $i=2s-1$. So we have $d(q_{s-1},q_{2s-1}) \leq 2A$. We now apply Fact \ref{fact:DistCoset} with $c_1$ and $c_2$ being respectively $c_{2s-1}$ and $c_0$, and with $a_1$ and $a_2$ being respectively $(c_{2s-1},q_{2s-1})$ and $(c_{0},p_{0})$. Using that the quantities $d(x,X(c_{2s-1}))$, $d(x,X(c_0))$, $d((c_{2s-1},q_{2s-1}),(c_0,p_0))$, and $d(X(c_{2s-1}),X(c_{0}))$ are $\leq A$, we deduce that $ d(x,(c_{2s-1},q_{2s-1}))\leq 6A$. Similarly, we have $d(y,(c_{s-1},q_{s-1})) \leq 6A$. Therefore,
$$\begin{array}{lcl} d(\pi(x),\pi(y)) & \leq & d(\pi(x),q_{2s-1}) + d(q_{2s-1},q_{s-1}) + d(q_{s-1},\pi(y)) \\ \\ & \leq & d(x,(c_{2s-1},q_{2s-1})) +2A+ d(y,(c_{s-1},q_{s-1})) \leq 14A, \end{array}$$
concluding the proof of our lemma.
\end{proof}

\noindent
We are now ready to prove our second preliminary lemma. It essentially states that the sufficient condition provided by Lemma \ref{lem:CircleHeight} is almost necessary.

\begin{lemma}\label{lem:ChainCircles}
Fix four integers $A,B,A',B' \geq 2$. There exists an integer $N \geq 1$ such that, for any two points $x,y \in \mathcal{L}_n(X)$ satisfying $\pi(x)= \pi(y)$ and $d(x,y) > 2 D^{\max(A',B)}$, there exist $(A,B)$-circles of leaves
$$\begin{array}{c} X(c_0^1), \ X(c_1^1), \ldots, \ X(c_{2s_1-1}^1) \\ X(c_0^2), \ X(c_1^2), \ldots, \ X(c_{2s_2-1}^2) \\ \vdots \\ X(c_0^r), \ X(c_1^r), \ldots, \ X(c_{2s_r-1}^r) \end{array}$$
such that $r \leq N$ and such that
\begin{itemize}
	\item for every $1 \leq i \leq r$, $X(c_0^i), \ldots, C(c_{s_i}^i)$ and $X(c_{s_i-1}^i), \ldots, X(c_{2s_i-1}^i)$ are $(A',B')$-geodesic;
	\item for every $1 \leq i \leq r-1$, $X(c_{s_i-1}^i)$, $X(c_{s_i}^i)$, $X(c_0^{i+1})$, and $X(c_{2s_i-1}^i)$ are pairwise at distance $\leq A$;
	\item $x$ is at distance $\leq A$ to both $X(c_0^1)$ and $X(c_{2s_1-1}^1)$;
	\item there exists a point $y'$ at distance $\leq 2D^{\max(A',B)}$ from $y$ and at distance $\leq A$ to both $X(c_0^r)$ and $X(c_{2s_r-1}^r)$.
\end{itemize}
Moreover, $N$ depends only on $X$, $A'$ and $B$. 
\end{lemma}

\begin{proof}
For convenience, set $C:=\max(A',B)$. 

\begin{claim}\label{claim:Partition}
There exists a partition $X=\mathcal{X}_1 \sqcup \cdots \sqcup \mathcal{X}_N$ such that, for every $1 \leq i \leq N$, any two points in $\mathcal{X}_i$ are at distance $> C$. 
\end{claim}

\noindent
Let $\Xi$ be the graph whose vertex-set is $X$ and such that two elements in $X$ are linked by an edge if and only if they are at distance $\leq C$ in $X$. Because $X$ has bounded degree, so does $\Xi$, i.e. there exists some $N \geq 1$ such that every vertex of $\Xi$ has degree $\leq N-1$. Then one can colour the vertices of $\Xi$ with $N$ colours so that any two adjacent vertices have different colours. The partition of $X$ induced by this colouring is the partition we are looking for. Thus, our claim is proved.

\medskip \noindent
Now, fix two points $x,y \in \mathcal{L}_n(X)$ satisfying $\pi(x)= \pi(y)$. Without loss of generality, we assume that $x=(0,x_0)$ for some $x_0\in X$. We can write $y$ as $(c,x_0)$ for some colouring $c \in \mathbb{Z}_n^{(X)}$. Up to reindexing our partition $\mathcal{X}_1 \sqcup \cdots \sqcup \mathcal{X}_N$, we can assume that there exists some $0 \leq r \leq N$ such that $\mathrm{supp}(c)$ contains at least one point in each $\mathcal{X}_1, \ldots, \mathcal{X}_r$ that is at distance $\geq C$ from $1$. 

\medskip \noindent
Notice that, if $r=0$, then $\mathrm{supp}(c) \subset B(x_0,C)$, which implies that $d(x,y) \leq 2D^{C}$. This is a contradiction with our assumptions, so we must have $r \geq 1$. 

\medskip \noindent
We define inductively a sequence of colourings in $\mathbb{Z}_n^{(X)}$ as follows:
\begin{itemize}
	\item Set $c_0=0$.
	\item For every $1 \leq i \leq r$, $c_i$ is a colouring that agrees with $c$ over $(\mathcal{X}_1 \cup \cdots \cup \mathcal{X}_i) \backslash B(x_0,C)$, that differs from $c_{i-1}$ at $x_0$, and that is zero elsewhere. 
\end{itemize}
Notice that $c_r$ and $c$ may only disagree in $B(x_0,C)$, hence $d(y,y') \leq 2D^C$ where $y':=(c_r,x_0)$. 

\medskip \noindent To prove our lemma, it suffices to show that, given an index $0 \leq i \leq r-1$, there exists an $(A,B)$-circle of leaves $X(a_0),\ldots,X( a_{2s-1})$ such that:
\begin{itemize}
	\item $X(a_0),\ldots, X(a_s)$ and $X(a_{s-1}), \ldots, X(a_{2s-1})$ are $(A',B')$-geodesic;
	\item $(c_i,x_0)$ is at distance $\leq A/2$ from both $X(a_0)$ and $X(a_{2s-1})$;
	\item $(c_{i+1},x_0)$ is at distance $\leq A/2$ from both $X(a_{s-1})$ and $X(a_s)$.
\end{itemize}
By construction, we can write $c_{i+1}=c_i+ \delta(x_0) + \delta(x_1) + \cdots + \delta(x_{s-1})$ for some $x_1, \ldots, x_{s-1} \in X$ such that $x_0,x_1, \ldots, x_{s-1}$ are pairwise at distance $>C$ in $X$, where $\delta(x_0),\delta(x_1),\ldots, \delta(x_{s-1})$ are respectively supported in $\{x_0\},\{x_1\},\ldots, \{x_{s-1}\}$. Set
\begin{itemize}
	\item $a_0:= c_i+ \delta(x_0)$ and $a_j:= c_i+\delta(x_0)+\delta(x_1)+ \cdots + \delta(x_{j})$ for every $1 \leq j \leq s-1$;
	\item $a_s:= c_i+\delta(x_1)+ \cdots + \delta(x_{s-1})$ and $a_{s+j}:= c_i + \delta(x_{j+1})+ \cdots + \delta(x_{s-1})$ for every $1 \leq j \leq s-1$.
\end{itemize}
Notice that, for every $0 \leq j \leq 2s-1$, we have
$$d(X(a_j,)X(a_{j+1})) = d(X(a_j), X(a_j + \delta)) =d(X, X(\delta)) \leq 1$$
where $\pm \delta \in \{\delta(x_0),\delta(x_1), \ldots, \delta(x_{s-1})\}$. Also,
$$d(X(a_j),X(a_{j+2}))= d(X(a_j),X(a_j+\delta_1+\delta_2)) = d(X,X(\delta_1+\delta_2)) \geq d(g_1,g_2) >C$$
where $\pm \delta_1, \pm \delta_2$ are two colourings that belong to $\{\delta(x_0),\delta(x_1), \ldots, \delta(x_{s-1})\}$ and that are supported at two distinct points $g_1,g_2 \in \{x_0,x_1, \ldots, x_{s-1}\}$. Consequently, our sequence $X(a_0), \ldots, X(a_{2s-1})$ defines an $(A,B)$-circle of leaves. Next, notice that $(c_i,x_0)$ is at distance $\leq 1$ from $X(a_0)=X(c_i+\delta(1))$ and it belongs to $X(a_{2s-1})=X(c_i)$. Also, $(c_{i+1},x_0)$ belongs to $X(a_{s-1})=X(c_{i+1})$ and is at distance $\leq 1$ from $X(a_s)=X(c_{i+1}-\delta(x_0))$. 

\medskip \noindent
Now, we claim that $X(a_0),\ldots, X(a_s)$ is $(A',B')$-geodesic. So let $X(b_0),\ldots, X(b_t)$ be an $(A',B')$-line of leaves from $X(a_p)$ to $X(a_q)$ for some $0 \leq p < q \leq s$. We need to show that $t \geq q-p-1$. For every $0 \leq j \leq t-1$, let $L_j \in \mathbb{Z}_n^{(X)}$ be such that $b_{j+1}=b_j+L_j$; we denote by $S_j$ the support of $L_j$. Notice that
$$A' \geq d(X(b_j),X(b_{j+1})) \geq \mathrm{diam}(S_j)$$
for every $0 \leq j \leq t-1$. As a consequence, two distinct elements among $x_0,x_1, \ldots, x_{s-1}$ cannot belong to the same $S_j$ since they are at least $C+1$ apart. But $\mathrm{supp}(a_{q}-a_p) \cap \{x_0,x_1, \ldots, x_{s-1}\}$ has cardinality $q-p-1$ and
$$\mathrm{supp}(a_{q}-a_p) = \mathrm{supp}(b_t-b_0) \subset \bigcup\limits_{j=0}^{t-1} S_j,$$
so we must have $t \geq p-q-1$, proving our claim. 

\medskip \noindent
One can prove similarly that $X(a_{s-1}), \ldots, X(a_{2s-1})$ is $(A',B')$-geodesic, concluding the proof of our lemma.
\end{proof}

\begin{proof}[Proof of Proposition \ref{prop:Height}.]
Without loss of generality, we assume that $q$ sends a leaf of $\mathcal{L}_n(X)$ to a leaf of $\mathcal{L}_m(Y)$. Let $C_1,C_2 \geq 1$ be such that $q$ is a $(C_1,C_2)$-quasi-isometry, i.e.
$$\frac{1}{C_1} d(x,y)-C_2 \leq d(q(x),q(y)) \leq C_1 d(x,y)+C_2$$
for all $x,y \in \mathcal{L}_n(X)$ and every point in $\mathcal{L}_m(Y)$ is at distance $\leq C_2$ from the image of $q$. Define a quasi-inverse $\bar{q} : \mathcal{L}_m(Y) \to\mathcal{L}_n(X)$ by fixing, for every $x \in \mathcal{L}_m(Y)$, a point $\bar{q}(x)$ in the preimage under $q$ of a point in $q(\mathcal{L}_n(X))$ within $C_2$ from $x$. A straightforward computation shows that $\bar{q}$ is a $(C_1,3C_1C_2)$-quasi-isometry. Another straightforward computation shows that:

\begin{claim}\label{claim:GeodToGeod}
For all integers $A,B,A',B' \geq 0$, $q$ sends an $(A,B)$-line of leaves that is $(A',B')$-geodesic to a $(C_1A+C_2, \frac{B}{C_1}-C_2)$-line of leaves that is $(\frac{A'}{C_1}-3C_2, C_1B'+3C_1^2C_2)$-geodesic. 
\end{claim}

\noindent
Define integers $A_1,B_1,A_1',B_1',A_2,B_2,A_2',B_2'$ such that:
\begin{itemize}
	\item $A_1,B_1' \geq 2$ and $A_1'=C_1(2+3C_2)$;
	\item $B_1:= \max \left( 1+C_1(C_1A_1+2C_2), C_1(2+C_2+2D^{C_1A_1+C_2}) \right)$;
	\item $A_2=C_1A_1+C_2$ and $B_2= \frac{B_1}{C_1}-C_2$;
	\item $A_2'= \max \left( \frac{A_1'}{C_1}-3C_2, 2 D^{3A_2} \right)$;
	\item $B_2'= \min \left( C_1B_1'+3C_1^2C_2, B_2-2D^{A_2} \right)$.
\end{itemize}
Also, set $C:=14NA_2+2C_1D^{\max(A_1',B_1)}+C_2$ where $N$ is the constant given by Lemma~\ref{lem:ChainCircles} applied to $\mathcal{L}_n(X)$ with respect to $A_1,B_1,A_1',B_1'$. 

\medskip \noindent
Fix two points $x,y \in \mathcal{L}_n(X)$ satisfying $\pi_X(x)=\pi_X(y)$. We want to prove that $d(\pi_Y(q(x)),\pi_Y(q(y))) \leq C$. If $d(x,y) \leq 2 D^{\max(A_1',B_1)}$ then 
$$d(\pi_Y(q(x)),\pi_Y(q(y))) \leq d(q(x),q(y)) \leq C_1d(x,y)+C_2 \leq 2C_1D^{\max(A_1',B_1)} +C_2 \leq C$$
and there is nothing to prove. Therefore, we assume from now on that $d(x,y) > 2 D^{\max(A_1',B_1)}$. According to Lemma \ref{lem:ChainCircles}, there exist $(A_1,B_1)$-circles of leaves
$$\begin{array}{c} X(c_0^1), \ X(c_1^1), \ldots, \ X(c_{2s_1-1}^1) \\ X(c_0^2), \ X(c_1^2), \ldots, \ X(c_{2s_2-1}^2) \\ \vdots \\ X(c_0^r), \ X(c_1^r), \ldots, \ X(c_{2s_r-1}^r) \end{array}$$
such that $r \leq N$ and such that
\begin{itemize}
	\item for every $1 \leq i \leq r$, $X(c_0^i), \ldots, X(c_{s_i}^i)$ and $X(c_{s_i-1}^i), \ldots, X(c_{2s_i-1}^i)$ are $(A_1',B_1')$-geodesic;
	\item for every $1 \leq i \leq r-1$, there exists a point $z_i$ in the $A_1$-neighbourhood of $X(c_{s_i-1}^i)$, $X(c_{s_i}^i)$, $X(c_0^{i+1})$ and $X(c_{2s_i-1}^i)$;
	\item $x$ is at distance $\leq A_1$ to both $X(c_0^1)$ and $X(c_{2s_1-1}^1)$;
	\item there exists a point $y'$ at distance $\leq 2D^{\max(A_1',B_1)}$ from $y$ and at distance $\leq A_1$ to both $X(c_0^r)$ and $X(c_{2s_r-1}^r)$.
\end{itemize}
As a consequence of Claim \ref{claim:GeodToGeod}, 
$$\begin{array}{c} q(X(c_0^1)), \ q(X(c_1^1)), \ldots, \ q(X(c_{2s_1-1}^1)) \\ q(X(c_0^2)), \ q(X(c_1^2)), \ldots, \ q(X(c_{2s_2-1}^2)) \\ \vdots \\ q(X(c_0^r)), \ q(X(c_1^r)), \ldots, \ q(X(c_{2s_r-1}^r)) \end{array}$$
is a sequence of $(A_2,B_2)$-circles of leaves such that
\begin{itemize}
	\item for every $1 \leq i \leq r$, $q(X(c_0^i)), \ldots, q(X(c_{s_i}^i))$ and $q(X(c_{s_i-1}^i)), \ldots, q(X(c_{2s_i-1}^i))$ are $(A_2',B_2')$-geodesic;
	\item for every $1 \leq i \leq r-1$, $q(z_i)$ lies in the $A_2$-neighbourhood of $q(X(c_{s_i-1}^i))$, $q(X(c_{s_i}^i))$, $q(X(c_0^{i+1}))$ and $q(X(c_{2s_i-1}^i))$;
	\item $q(x)$ is at distance $\leq A_2$ to both $q(X(c_0^1))$ and $q(X(c_{2s_1-1}^1))$;
	\item $q(y')$ at distance $\leq 2C_1D^{\max(A_1',B_1)}+C_2$ from $q(y)$ and at distance $\leq A_2$ to both $q(X(c_0^r))$ and $q(X(c_{2s_r-1}^r))$.
\end{itemize}
For convenience, set $z_0=x$ and $z_r=y'$. According to Lemma \ref{lem:CircleHeight}, we have
$$d(\pi_Y(q(z_i)),\pi_Y(q(z_{i+1}))) \leq 14 A_2 \text{ for all } 0 \leq i \leq r-1.$$
Consequently,
$$d(\pi_Y(q(x)),\pi_Y(q(y'))) = \sum\limits_{i=0}^{r-1} d(\pi_Y(q(z_i)),\pi_Y(q(z_{i+1}))) \leq 14NA_2.$$
We conclude that
$$\begin{array}{lcl} d(\pi_Y(q(x)),\pi_Y(q(y))) & \leq & d(\pi_Y(q(x)),\pi_Y(q(y'))) + d(\pi_Y(q(y')),\pi_Y(q(y))) \\ \\ & \leq & 14NA_2 + d(q(y'),q(y)) \\ \\ & \leq & 14NA_2 + 2C_1D^{\max(A_1',B_1)}+C_2 \leq C \end{array}$$
as desired.
\end{proof}

\section{An embedding theorem}\label{bigsection:Embedding}

\noindent
This section is the core of the article, dedicated to the embedding theorem provided by the following statement:

\begin{thm}\label{thm:FullEmbeddingThm}
Let $X$ be a graph of bounded degree, $Z$ a graph that is uniformly one-ended and coarsely one-connected, and $n \geq 2$ an integer. For every coarse embedding $\rho : Z \to \mathcal{L}_n(X)$, there exists some $K \geq 0$ such that $\rho(Z)$ lies in the $K$-neighbourhood of a leaf in $\mathcal{L}_n(X)$. Moreover, $K$ depends only on $X$, $Z$, and the parameters of $\rho$. 
\end{thm}

\noindent
The strategy of the proof of Theorem~\ref{thm:FullEmbeddingThm} was outlined in the introduction. We now proceed to a more detailed description of its various steps. We refer to Section~\ref{section:approximation} for precise definitions.

\medskip \noindent First of all, we  introduce a new geometric interpretation of the lamplighter graph $\mathcal{L}_n(X)$. 
To every \emph{prism complex} $W$, i.e. to every cellular complex obtained by gluing products of simplices along faces, we associated a \emph{complex of pointed simplices} $\mathsf{PS}(W)$. Its vertices are pointed simplices in $W$ and two such pointed simplices are adjacent in $\mathsf{PS}(W)$ if one can pass from one to the other by some elementary moves, namely by either \emph{sliding} the point inside the simplex or by \emph{rotating} the simplex around its vertex. Then, we construct a prism complex $W(n,X)$ such that the lamplighter graph $\mathcal{L}_n(X)$ coincides with the one-skeleton of $\mathsf{PS}(W(n,X))$. 

\medskip \noindent Next, we exhibit some rigid structure on $W(n,X)$: its universal cover $\widetilde{W}(n,X)$ is a \emph{quasi-median complex}, which can be loosely thought of as the analogue of a CAT(0) cube complex where cubes are replaced with prisms. As a consequence of this structure, the complex $\mathsf{PS}(\widetilde{W}(n,X))$ inherits a wallspace structure. 

\medskip \noindent  The third step consists in showing that these walls satisfy nice properties, including the fundamental observation that the image of a wall in $\mathsf{PS}(\widetilde{W}(n,X))$ under the cover $\mathsf{PS}(\widetilde{W}(n,X)) \to \mathsf{PS}(W(n,X))$ induced by $\widetilde{W}(n,X) \to W(n,X)$ is bounded. 

\medskip \noindent  The combination of these three steps allows us to conclude as follows. We identify $\mathcal{L}_n(X)$ with the one-skeleton of $\mathsf{PS}(W(n,X))$. Using that $Z$ is coarsely one-connected, we show that the image of any loop in $Z$ by $\rho$  is homotopically trivial in $ \mathsf{PS}(W(n,X))$.  So $\rho$ lifts to a coarse embedding $\widetilde{\rho} : Z \to \mathsf{PS}(\widetilde{W}(n,X))$. A crucial property is that the walls in $\mathsf{PS}(\widetilde{W}(n,X))$ have bounded images in $\mathsf{PS}(W(n,X))$. Now since $Z$ is uniformly one-ended, we know that $\widetilde{\rho}(Z)$ cannot cross a wall of $\mathsf{PS}(\widetilde{W}(n,X))$ in an essential way. As a consequence, up to finite Hausdorff distance $\widetilde{\rho}(Z)$ is confined into a small region of $\mathsf{PS}(\widetilde{W}(n,X))$ that is not crossed by any wall. By construction, the image in $\mathsf{PS}(W(n,X))$ of such a region coincides with a leaf of $\mathcal{L}_n(X)$, allowing us to deduce the desired conclusion.

\subsection{Preliminaries on quasi-median geometry}\label{section:QM}

\noindent
Our proof of the embedding theorem provided by Theorem~\ref{thm:FullEmbeddingThm} relies fundamentally on quasi-median geometry, i.e. the geometry of quasi-median graphs and complexes. In this section, we record all the definitions and properties that will be needed later.

\paragraph{Quasi-median graphs.} There exist several equivalent definitions of quasi-median graphs. See for instance \cite{quasimedian}. Below, we give the definition used in \cite{QM}.

\begin{definition}
A connected graph $X$ is \emph{quasi-median} if it does not contain $K_4^-$ and $K_{3,2}$ as induced subgraphs, and if it satisfies the following two conditions:
\begin{description}
	\item[(triangle condition)] for every triplet of vertices $a, x,y \in X$, if $x$ and $y$ are adjacent and if $d(a,x)=d(a,y)$, then there exists a vertex $z \in X,$ which is adjacent to both $x$ and $y$, satisfying $d(a,z)=d(a,x)-1$;
	\item[(quadrangle condition)] for every quadruplet of vertices $a,x,y,z \in X$, if $z$ is adjacent to both $x$ and $y$ and if $d(a,x)=d(a,y)=d(a,z)-1$, then there exists a vertex $w \in X$ that is adjacent to both $x,y$ and that satisfies $d(a,w)=d(a,z)-2$.
\end{description}
\end{definition}

\noindent
The graph $K_{3,2}$ is the bipartite complete graph, corresponding to two squares glued along two adjacent edges; and $K_4^-$ is the complete graph on four vertices minus an edge, corresponding to two triangles glued along an edge.

\medskip \noindent
However, the definition of quasi-median graphs is only given for completeness. Similarly to median graphs (i.e. one-skeleta of CAT(0) cube complexes), the geometry of quasi-median graphs essentially reduces to the combinatorics of separating subspaces referred to as \emph{hyperplanes}. Understanding this interaction between hyperplanes and geometry will be sufficient for us.

\begin{definition}
Let $X$ be a quasi-median graph. A \emph{hyperplane} $J$ is an equivalence class of edges with respect to the transitive closure of the relation that identifies two edges in the same $3$-cycle and two opposite edges in the same $4$-cycle. The \emph{carrier} of $J$, denoted by $N(J)$, is the subgraph generated by the edges in $J$. The connected components of the graph $X\backslash \backslash J$ obtained from $X$ by removing the interiors of the edges in $J$ are the \emph{sectors delimited by $J$}; and the connected components of $N(J) \backslash \backslash J$ are the \emph{fibres} of $J$. Two distinct hyperplanes $J_1$ and $J_2$ are \emph{transverse} if $J_2$ contains an edge in $N(J_1) \backslash J_1$. 
\end{definition}
\begin{figure}
\begin{center}
\includegraphics[clip, trim= 0 16cm 10cm 0, scale=0.4]{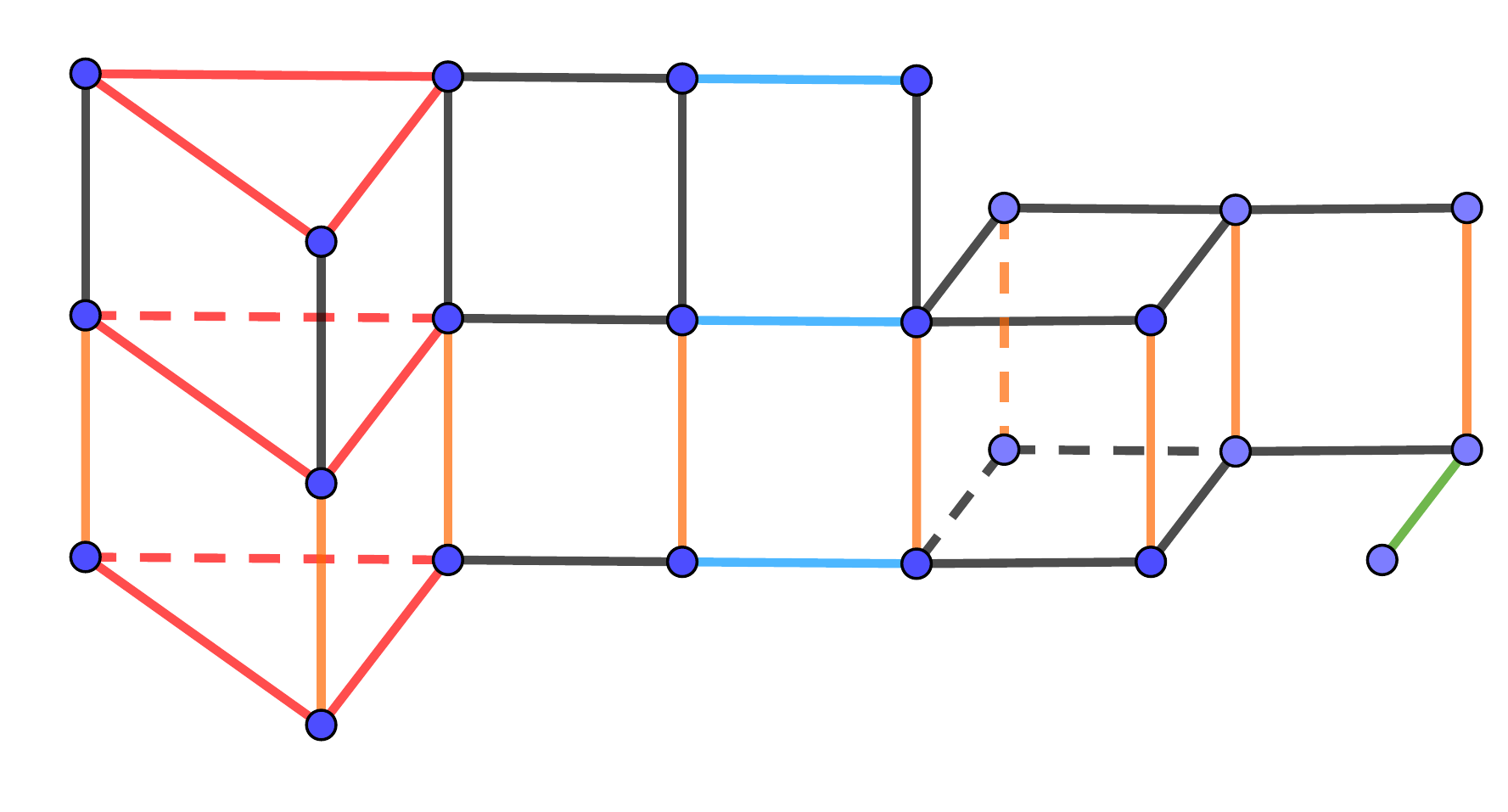}
\caption{Four hyperplanes in a quasi-median graphs. The orange hyperplane is transverse to the red and blue hyperplanes.}
\label{hyperplanes}
\end{center}
\end{figure}

\noindent
See Figure \ref{hyperplanes} for a few examples. The role of hyperplanes is highlighted by the following statement. For a self-contained proof, we refer to \cite[Propositions 2.15 and 2.30]{QM}; see also \cite[Theorem 3.19]{MR4033512} (or \cite[Theorem 2.14]{AutQM} for proofs specific to the quasi-median graphs that will be considered in the next sections). 

\begin{thm}\label{thm:MainQM}
Let $X$ be a quasi-median graph. The following assertions hold:
\begin{itemize}
	\item[(i)] For every hyperplane $J$, $X\backslash \backslash J$ contains at least two connected components.
	\item[(ii)] Carriers and fibres of hyperplanes are gated.
	\item[(iii)] A path is a geodesic if and only if it crosses each hyperplane at most once.
	\item[(iv)] The distance between two vertices coincides with the number of hyperplanes separating them.
\end{itemize}
\end{thm}

\noindent
Recall that a subgraph $Y \subset X$ is \emph{gated} if, for every vertex $x \in X$, there exists a vertex $y \in Y$ such that, for every $z \in Y$, there exists a geodesic between $x$ and $z$ passing through $y$. We refer to the vertex $y$ as the \emph{projection of $x$ onto $Y$}. The property of being gated can be thought of as a strong convexity condition. In particular, being gated implies being convex. Also, observe that the projection $y$ coincides with the unique vertex of $Y$ minimising the distance to $x$. A useful statement, showing how hyperplanes interact with projections, is the following. We refer to \cite[Lemma 2.34]{QM} or \cite[Corollary~3.20]{MR4033512} for a proof. 

\begin{lemma}\label{lem:ProjHyp}
Let $X$ be a quasi-median graph, $Y \subset X$ a gated subgraph, and $x \in X$ a vertex. If a hyperplane separates $x$ from its projection onto $Y$, then it separates $x$ from~$Y$.
\end{lemma}

\noindent
Quasi-median graphs are quite similar to median graphs. Essentially, the only differences are that edges are replaced with \emph{cliques} (i.e. maximal complete subgraphs), that cubes are replaced with \emph{prisms} (i.e. subgraphs which are products of cliques), and that cutting along a hyperplane may disconnected the graph into more than two connected components (possibly infinitely many). 

\medskip \noindent
We conclude this subsection with a few statements that will be needed later.

\begin{lemma}\label{lem:SpanPrism}
Let $X$ be a quasi-median graph and $C_1,C_2$ two cliques. If $C_1 \cap C_2 \neq \emptyset$ and if the hyperplanes containing $C_1$ and $C_2$ are transverse, then $C_1$ and $C_2$ span a prism. 
\end{lemma}

\begin{proof}
Fix a vertex $a \in C_1 \cap C_2$ and let $J_1,J_2$ denote the hyperplanes containing respectively $C_1,C_2$. We begin by proving the following observation:

\begin{claim}
For all $b_1 \in C_1 \backslash \{a\}$ and $b_2 \in C_2 \backslash \{a\}$, the edges $[a,b_1]$ and $[a,b_2]$ span a square. 
\end{claim}

\noindent
Let $F$ denote the fibre of $J_2$ that contains $b_2$ and let $c$ denote the projection of $b_1$ onto $F$. Observe that $[b_1,a] \cup [a,b_2]$ is a geodesic as it crosses only two distinct hyperplanes (namely, $J_1$ and $J_2$). As a consequence, $b_1$ does not belong to $F$, since otherwise the convexity of $F$ would imply that $a \in [b_1,a] \cup [a,b_2] \subset F$. On the other hand, we know that $d(b_1,c) = d(b_1,F) \leq d(b_1,b_2) = 2$, so $d(b_1,c) \in \{1,2\}$. If $d(b_1,c)=2$ then we must have $c=b_2$, which is impossible according to Lemma \ref{lem:ProjHyp} because $J_1$ separates $b_1$ from $b_2$ but not from $F$. Therefore, $d(b_1,c)=1$; in other words, $c$ is adjacent to $b_1$. Because $J_1$ contains the edges $[b_1,c],[a,b_2]$ and does not separate $c,b_2$, necessarily the hyperplanes separating $c$ and $b_2$ must separate $a$ and $b_1$. In other words, $J_2$ is the only hyperplane separating $c$ and $b_2$, which implies that $c$ is also adjacent to $b_2$. We conclude that $[a,b_1]$ and $[a,b_2]$ span a square, namely $\{a,b_1,b_2,c\}$, as desired.

\medskip \noindent
For all $b_1 \in C_1 \backslash \{a\}$ and $b_2 \in C_2 \backslash \{a\}$, we denote by $c(b_1,b_2)$ the vertex of the square spanned by $[a,b_1]$ and $[a,b_2]$ opposite to $a$. Consider the map
$$\Phi : \left\{ \begin{array}{ccc} C_1 \times C_2 & \to & X \\ (x,y) & \mapsto & \left\{ \begin{array}{cl} c(x,y) & \text{if $x \neq a$ and $y \neq a$} \\ x & \text{if $y=a$} \\ y & \text{if $x=a$} \end{array} \right. \end{array} \right..$$
We claim that $\Phi$ is a graph embedding with an induced image, which will conclude the proof of our lemma since it will follow that $C_1$ and $C_2$ span the prism $\Phi(C_1 \times C_2)$. 

\medskip \noindent
So let $(x_1,y_1),(x_2,y_2) \in C_1 \times C_2$ be two adjacent vertices. Up to symmetry, we assume without loss of generality that $x_1,x_2$ are adjacent in $C_1$ and that $y_1=y_2$ (from now on, we denote by $y$ this common vertex). 
\begin{itemize}
	\item If $y=a$, then $\Phi(x_1,y)=x_1$ and $\Phi(x_2,y)=x_2$ are adjacent in $X$.
	\item If $y \neq a$ and $x_1=a$, then $\Phi(x_1,y)=y$ and $\Phi(x_2,y)=c(x_2,y)$ are adjacent.
	\item If $y \neq a$ and $x_2=a$, then $\Phi(x_1,y)=c(x_1,y)$ and $\Phi(x_2,y)=y$ are adjacent.
	\item If $y,x_1,x_2 \neq a$, then $\Phi(x_1,y)=c(x_1,y)$ and $\Phi(x_2,y)=c(x_2,y)$ are adjacent since otherwise the two squares $\{a,x_1,y,c(x_1,y)\}$ and $\{a,x_2,y,c(x_2,y)\}$ would define an induced copy of $K_{3,2}$ in $X$.
\end{itemize}
Therefore, $\Phi$ sends an edge to an edge. Next, observe that $c : (C_1 \backslash \{a\} ) \times (C_2 \backslash \{a\} ) \to X$ is injective since otherwise $X$ would contain an induced copy of $K_{3,2}$. Consequently, $\Phi$ is also injective. Finally, notice that, if $\Phi(C_1 \times C_2)$ contains an edge that does not come from $C_1 \times C_2$, then a square in $C_1 \times C_2$ would have a diagonal in $X$, implying that $J_1=J_2$, a contradiction. This concludes the proof of our lemma.
\end{proof}

\noindent
For our next lemma, we need the following definition. The \emph{cubical dimension} of a quasi-median graph $X$, denoted by $\dim_\square(X)$, is the maximal size of a collection of pairwise transverse hyperplanes. Although this observation will not be used in the sequel, we know from \cite[Proposition 2.79]{QM} that the cubical dimension coincides with the maximal number of factors of a prism in $X$. 

\begin{lemma}\label{lem:MetricInf}
Let $X$ be a quasi-median graph and $x,y \in X$ two vertices. Let $N$ denote the maximal number of pairwise non-transverse hyperplanes separating $x$ and $y$. Then $d(x,y) \leq \dim_\square(X) \cdot N$.
\end{lemma}

\begin{proof}
Let $\mathscr{S}$ denote the collection of the sectors delimited by hyperplanes separating $x,y$ that contain $x$, partially ordered by the inclusion. Let $\mathscr{S}= \mathscr{S}_1 \sqcup \cdots \sqcup \mathscr{S}_r$ be a partition with $r$ minimal such that each $\mathscr{S}_i$ is totally ordered by the inclusion. Observe that a subcollection of sectors that are pairwise incomparable with respect to the inclusion corresponds to a collection of pairwise transverse hyperplanes, so it must have cardinality $\leq \dim_\square(X)$. It follows from Dilworth's theorem that $r \leq \dim_\square(X)$. We conclude that
$$d(x,y) = | \mathscr{S}| = | \mathscr{S}_1| + \cdots + | \mathscr{S}_r| \leq \dim_\square(X) \cdot N,$$
as desired.
\end{proof}

\begin{lemma}\label{lem:InterSectors}
Let $X$ be a quasi-median graph of finite cubical dimension. For every hyperplane $J$, choose a sector $J^+$ delimited by $J$. Assume that
\begin{itemize}
	\item for any two hyperplanes $J_1$ and $J_2$, $J_1^+ \cap J_2^+ \neq \emptyset$;
	\item every non-increasing sequence $J_1^+ \supset J_2^+ \supset \cdots$ is eventually constant. 
\end{itemize}
Then $\bigcap\limits_{\text{$J$ hyperplane}} J^+$ is non-empty and reduced to a single vertex.
\end{lemma}

\begin{proof}
Assume for contradiction that the intersection is empty. We define by induction a sequence $x_0,x_1,\ldots$ of vertices and a sequence $J_1,J_2,\ldots$ of hyperplanes as follows:
\begin{itemize}
	\item Fix an arbitrary vertex $x_0 \in X$. 
	\item Assume that $x_n$ and $J_1,\ldots, J_n$ are defined. There must exist some hyperplane $J_{n+1}$ such that $x_n \notin J_{n+1}^+$. Define $x_{n+1}$ as the projection of $x_n$ onto the intersection $I_{n+1}:= \bigcap\limits_{i=1}^{n+1} J_{i}^+$. 
\end{itemize}
Observe that $(I_n)$ defines a decreasing sequence of gated subgraphs such that, for every $n \geq 1$, $x_n$ belongs to $I_n$ and, according to Lemma \ref{lem:ProjHyp}, $J_{n+1}$ separates $x_n$ from $I_{n+1}$. As a consequence, $x_m \in J_n^+$ for every $m \geq n$. Because every collection of pairwise transverse hyperplane has cardinality at most $\dim_\square(X)<\infty$, it follows from Ramsey's theorem that there exists a subsequence $(J_{k_n})$ of pairwise non-transverse hyperplanes. We deduce from the previous observation that $J_{k_1}^+ \supset J_{k_2}^+ \supset \cdots$, contradicting our assumptions. 

\medskip \noindent
Thus, we have proved that $\bigcap\limits_{\text{$J$ hyperplane}} J^+$ is non-empty. Fix a vertex $a$ in this intersection and let $b \in X$ be an arbitrary vertex distinct from $a$. Because $a \neq b$, there exists some hyperplane $J$ separating $a$ and $b$; and, because $a \in J^+$, necessarily $b \notin J^+$, hence $b \notin \bigcap\limits_{\text{$J$ hyperplane}} J^+$. Thus, we have proved that $a$ is the unique vertex in $\bigcap\limits_{\text{$J$ hyperplane}} J^+$, concluding the proof of our lemma.
\end{proof}

\paragraph{Prism complexes.} In the same way that median graphs can be thought of as one-skeleta of cubical complexes, quasi-median graphs are naturally one-skeleta of \emph{prism complexes}. Here, a \emph{prism} refers to a product of simplices and a \emph{prism complex} to a cellular complex obtained by gluing prisms along \emph{faces} (i.e. products of maximal simplices). We emphasize that, with our definition, gluing two triangles along a common edge does not define a prism complex since the faces of a triangle, when thought of as a prism, are its vertices and itself. 

\medskip \noindent
Because two intersecting cliques in a quasi-median graph either coincide or intersect along a single vertex (see \cite[Lemma 2.11]{QM}), filling in the prisms of the graph with products of simplices yields a prism complex. We refer to such a prism complex as a \emph{quasi-median complex}. The only thing we need to know about quasi-median complexes is that they are simply connected. In fact, it is not difficult to deduce from the product structure of hyperplanes that they are contractible. A less straightforward property is that quasi-median complexes can be endowed with CAT(0) metrics \cite[Theorem~2.120]{QM}, which also implies that they are contractible.

\paragraph{Graph products of groups.} Given a simplicial graph $\Gamma$ and a collection $\mathcal{G}=\{G_u \mid u \in V(\Gamma)\}$ of groups indexed by the vertex-set $V(\Gamma)$ of $\Gamma$, the \emph{graph product} $\Gamma \mathcal{G}$ is the quotient
$$\left( \underset{u \in V(\Gamma)}{\ast} G_u \right) / \langle \langle [g,h]=1 \text{ if $g \in G_u$ and $h \in G_v$ for some $\{u,v\} \in E(\Gamma)$} \rangle \rangle$$
where $E(\Gamma)$ denotes the edge-set of $\Gamma$. The groups in $\mathcal{G}$ are referred to as \emph{vertex-groups}. If all the vertex-groups are isomorphic to a single group $G$, we denote the graph product by $\Gamma G$ instead of $\Gamma \mathcal{G}$. 

\medskip \noindent
Vertex-groups embed into the graph product. More generally, given an induced subgraph $\Lambda \subset \Gamma$, the subgroup generated by the vertex-groups indexed by the vertices in $\Lambda$, which we denote by $\langle \Lambda \rangle$, is naturally isomorphic to the graph product $\Lambda \mathcal{G}_{|V(\Lambda)}$ where $\mathcal{G}_{|V(\Lambda)}:= \{ G_u \mid u \in V(\Lambda) \} \subset \mathcal{G}$. 

\medskip \noindent
As observed in \cite{QM}, graph products naturally act on quasi-median graphs. More precisely, it acts by left-multiplication on the following Cayley graph:

\begin{prop}\label{prop:QMisquasimedian}
Given a simplicial graph $\Gamma$ and a collection of groups $\mathcal{G}$ indexed by $V(\Gamma)$, consider the Cayley graph
$$\mathrm{QM}(\Gamma, \mathcal{G}):= \mathrm{Cayl} \left( \Gamma \mathcal{G}, \bigcup\limits_{u \in V(\Gamma)} G_u \backslash \{1\} \right).$$
Then $\mathrm{QM}(\Gamma, \mathcal{G})$ is a quasi-median graph of cubical dimension $\mathrm{clique}(\Gamma)$.
\end{prop}

\noindent
We refer to \cite[Proposition 8.2]{QM} for a proof. The cliques, prisms, and hyperplanes of $\mathrm{QM}(\Gamma, \mathcal{G})$ are described as follows (see \cite[Lemma~8.6 and Corollaries~8.7 and~8.10]{QM} or \cite[Lemmas~2.4 and~2.6, and Theorem~2.10]{AutQM}):
\begin{figure}
\begin{center}
\includegraphics[scale=0.4]{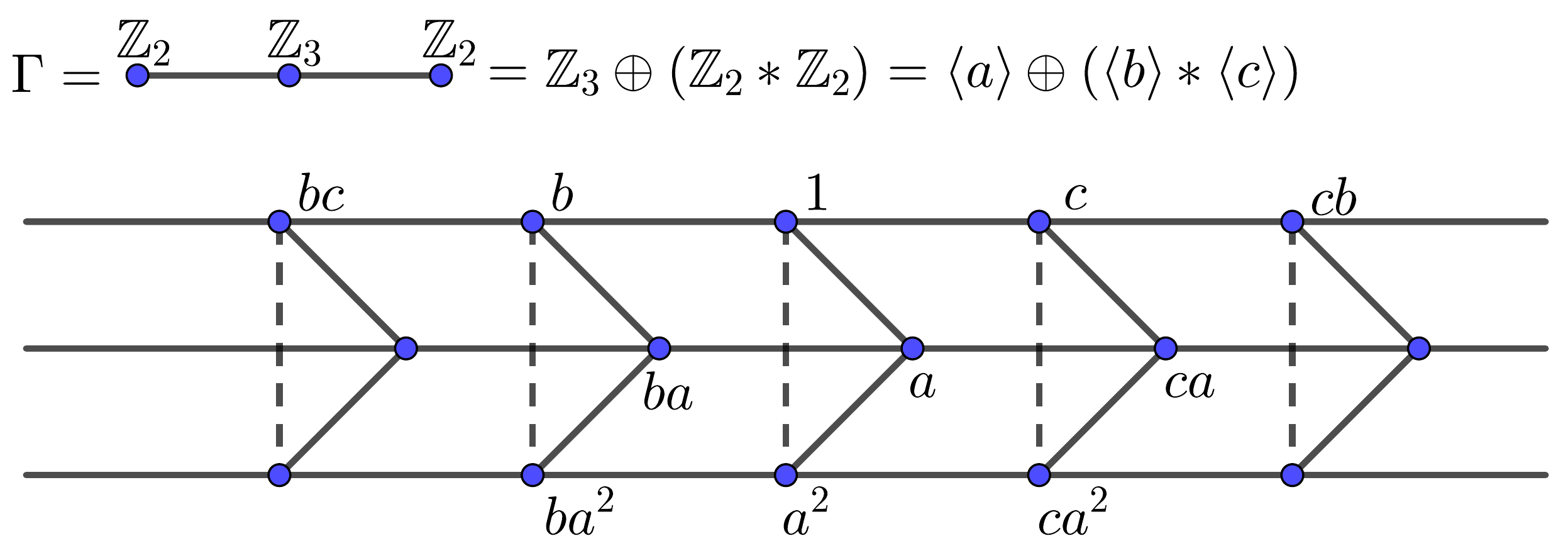}
\caption{Example of a quasi-median graph $\mathrm{QM}(\Gamma, \mathcal{G})$.}
\label{Cayl}
\end{center}
\end{figure}

\begin{lemma}\label{lem:QMcliques}
Let $\Gamma$ be a simplicial graph and $\mathcal{G}$ a collection of groups indexed by $V(\Gamma)$. The cliques of $\mathrm{QM}(\Gamma, \mathcal{G})$ coincide with the cosets of vertex-groups.
\end{lemma}

\begin{lemma}\label{lem:QMprisms}
Let $\Gamma$ be a simplicial graph and $\mathcal{G}$ a collection of groups indexed by $V(\Gamma)$. The prisms of $\mathrm{QM}(\Gamma, \mathcal{G})$ coincide with the cosets of the $\langle \Lambda \rangle$, where $\Lambda \subset \Gamma$ is a complete subgraph.
\end{lemma}

\begin{lemma}\label{lem:QMhyp}
Let $\Gamma$ be a simplicial graph and $\mathcal{G}$ a collection of groups indexed by $V(\Gamma)$. Fix a vertex $u \in V(\Gamma)$ and let $J_u$ denote the hyperplane containing the clique $G_u$. The cliques in $J_u$ are the cosets $gG_u$ where $g \in \langle \mathrm{star}(u) \rangle$. Consequently, the carrier $N(J_u)$ coincides with the subgraph $\langle \mathrm{star}(u) \rangle$ and splits at the Cartesian product $G_u \times \langle \mathrm{link}(u) \rangle$; and the stabiliser of $J_u$ in $\Gamma \mathcal{G}$ coincides with the subgroup $\langle \mathrm{star}(u) \rangle$.
\end{lemma}

\noindent
Recall that, given a graph $\Gamma$ and a vertex $u \in V(\Gamma)$, the \emph{star} (resp.\ the \emph{link}) of $u$ is the subgraph induced by $u$ and its neighbours (resp.\ by its neighbours). We record the following straightforward consequence of Lemma~\ref{lem:QMhyp}:

\begin{fact}\label{fact:PrismGeneration}
Let $\Gamma$ be a simplicial graph and $\mathcal{G}$ a collection of groups indexed by $V(\Gamma)$. Fix a hyperplane $J$, a vertex $x \in N(J)$, and two distinct cliques $C_1,C_2 \subset N(J)$ containing $x$. If $C_1 \subset J$, then $C_1$ and $C_2$ span a prism.
\end{fact}

\begin{proof}
Up to translating by $x^{-1}$, we assume that $x=1$. Consequently, there exist two vertices $u,v \in V(\Gamma)$ such that $J=J_u$, $C_1=G_u$, and $C_2=G_v$. It follows from Lemma~\ref{lem:QMhyp} that $v \in \mathrm{star}(u)$. Since the cliques $C_1$ and $C_2$ are distinct by assumption, we conclude that $v$ is adjacent to $u$, so $C_1$ and $C_2$ span the prism $G_u \oplus G_v$. 
\end{proof}

\noindent
In $\mathrm{QM}(\Gamma, \mathcal{G})$, like in any other Cayley graph, edges are naturally labelled by generators. As a consequence, edges in $\mathrm{QM}(\Gamma, \mathcal{G})$ are naturally labelled by vertices of $\Gamma$, corresponding to the vertex-groups the generators belong to. A direct consequence of Lemma~\ref{lem:QMhyp}~is:

\begin{lemma}\label{lem:CliqueLabel}
Let $\Gamma$ be a simplicial graph and $\mathcal{G}$ a collection of groups indexed by $V(\Gamma)$. In $\mathrm{QM}(\Gamma, \mathcal{G})$, two cliques in the same hyperplane have the same label.
\end{lemma}

\subsection{An alternative viewpoint on lamplighter graphs}\label{section:approximation}

\noindent
In this section, we propose an alternative description of lamplighter graphs that will be central in the proof of our embedding theorem.

\begin{definition}
Let $W$ be a finite-dimensional prism complex. The \emph{graph of pointed simplices} $\mathsf{PS}_1(W)$ is the graph whose vertices are the pointed simplices $(S,x)$, where $S$ is a maximal simplex (with respect to the inclusion) and $x$ a vertex, and whose edges link two pointed simplices $(S_1,x_1),(S_2,x_2)$ if either $S_1=S_2$ and $x_1 \neq x_2$ or $x_1=x_2$ and $S_1,S_2$ span a prism in $W$. 
\end{definition}

\noindent
The idea to keep in mind is that we are moving a pointed simplex $(S,x)$ in $W$ by applying two elementary moves: either we slide the vertex $x$ to another vertex of $S$, or we rotate the simplex $S$ around the vertex $x$ through a prism. See Figure~\ref{PS} for an explicit example. 
\begin{figure}
\begin{center}
\includegraphics[width=0.6\linewidth]{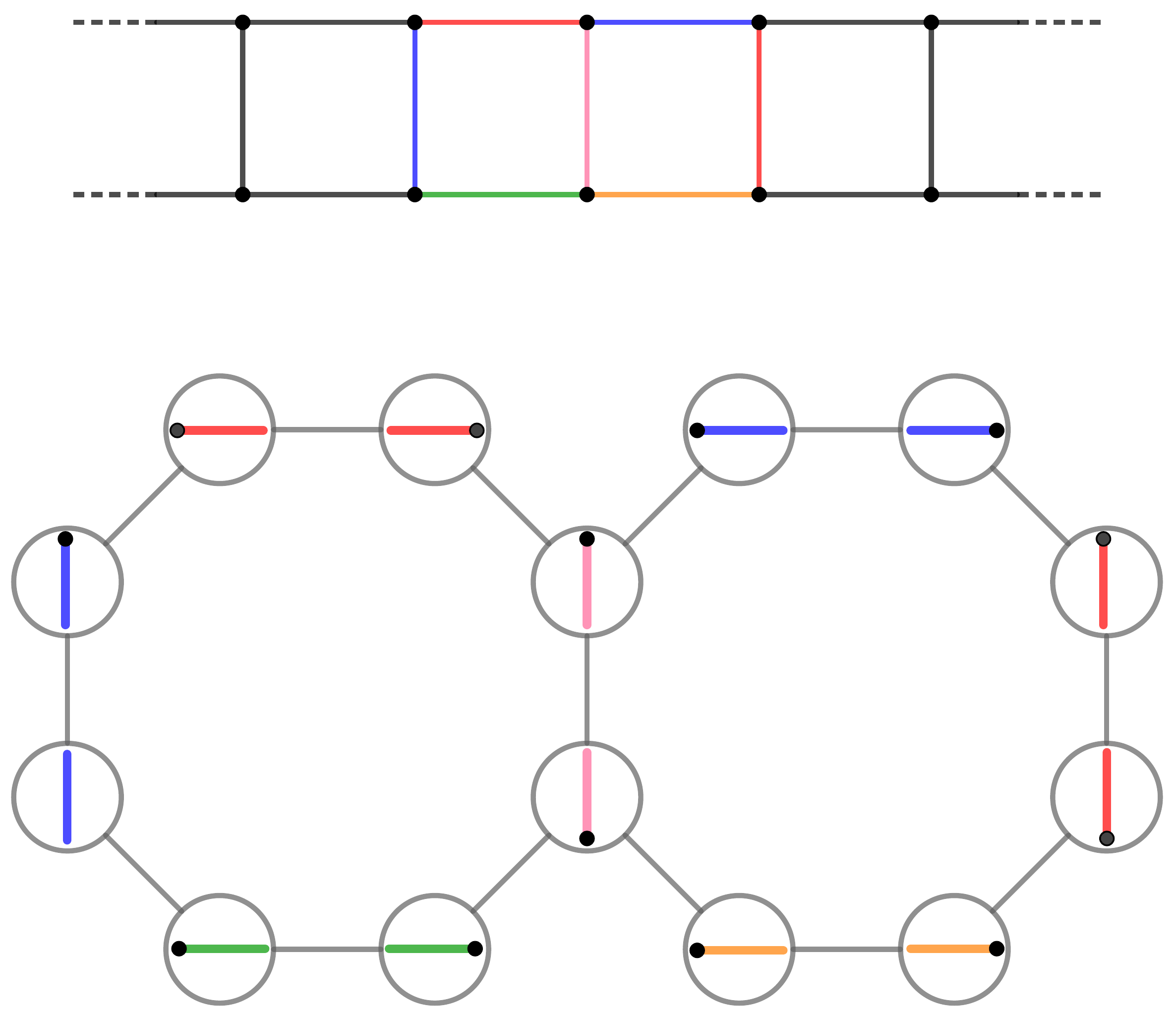}
\caption{A square complex and a piece of its graph of pointed edges.}
\label{PS}
\end{center}
\end{figure}

\medskip \noindent
Now, let us show that any lamplighter graph can be described as the graph of pointed simplices of some prism complex. So let $X$ be a locally finite graph and $n \geq 2$ an integer. Define the prism complex $W(n,X)$ as follows:
\begin{itemize}
	\item the vertices of $W(n,X)$ are the finitely supported colourings in $\mathbb{Z}_n^{(X)}$;
	\item two colourings are linked by an edge in $W(n,X)$ if they differ at a single vertex of~$X$;
	\item for every vertex $x \in X$ and every colouring $\varphi \in \mathbb{Z}_n^{(X)}$, the complete subgraph $\{ \psi \in \mathbb{Z}_n^{(X)} \mid \mathrm{supp}(\psi) \vartriangle \mathrm{supp}(\varphi)  \subset\{x\} \}$ in $W(n,X)$ spans a simplex $S(\varphi,x)$;
	\item simplices $S(\varphi, x_1), \ldots, S(\varphi,x_k)$ span a prism in $W(n,X)$ if $x_1,\ldots, x_k$ are pairwise adjacent in~$X$.
\end{itemize}
In other words, the prism complex $W(n,X)$ is a sub-complex of the prism $\bigoplus_X \Delta^n$: it has same underlying simplicial complex, but only contains prisms that are spanned by simplices labelled by pairwise adjacent vertices of $X$. While $\bigoplus_X \Delta^n$ has infinite dimension as soon as $X$ is infinite,  $W(n,X)$ has finite dimension as soon as $X$ has bounded degree. 

\medskip \noindent
We now state our main observation.

\begin{prop}\label{prop:WreathIsPS}
The map $(S,x) \mapsto (x, \text{ vertex of $X$ labelling $S$})$ induces a graph isomorphism $\mathsf{PS}_1(W(n,X)) \to \mathcal{L}_n(X)$. Moreover, it sends a leaf to a leaf.
\end{prop}

\noindent
Here, a leaf of $\mathsf{PS}_1(W(n,X))$ refers to a fibre of the canonical projection $\mathsf{PS}_1(W) \twoheadrightarrow W$ induced by $(S,x) \mapsto x$. In other words, a leaf is a subgraph spanned by $\{ (S,x) \mid \text{ $S$ maximal simplex}\}$ for some vertex $x$.

\begin{proof}[Proof of Proposition~\ref{prop:WreathIsPS}]
It is clear that our map induces a bijection from the vertices of $\mathsf{PS}_1(W(n,X))$ to the vertices of $\mathcal{L}_n(X)$. Let $(S_1,x_1),(S_2,x_2) \in \mathsf{PS}_1(W(n,X))$ be two vertices. Observe that $S_1=S_2$ and $x_1 \neq x_2$ amounts to saying that the colourings $x_1,x_2$ differ at a single vertex which is also the common label of $S_1,S_2$. Also, $x_1=x_2$ and $S_1,S_2$ span a prism amounts to saying that the colourings $x_1,x_2$ coincide and that the labels of $S_1,S_2$ are adjacent. Consequently, $(S_1,x_1)$ and $(S_2,x_2)$ are adjacent in $\mathsf{PS}_1(W(n,X))$ if and only if so are their images in $\mathcal{L}_n(X)$.
\end{proof}

\noindent
In the rest of the section, we study in more details graphs of pointed simplices of prism complexes. Our main objective is to endow them with structures of $2$-complexes, in a way that is sufficiently natural so that a covering map between prism complexes will induce a covering map between the corresponding complexes of pointed simplices. 

\begin{definition}
Let $W$ be a finite-dimensional prism complex. The \emph{complex of pointed simplices} $\mathsf{PS}(W)$ is the $2$-complex obtained from $\mathsf{PS}_1(W)$ by gluing triangles and octagons along the following cycles in $\mathsf{PS}_1(W)$:
\begin{itemize}
	\item $((S,x_1),(S,x_2),(S,x_3))$ where $x_1,x_2,x_3$ are pairwise distinct vertices of a maximal simplex $S$;
	\item $((S_1,x),(S_2,x),(S_3,x))$ where $S_1,S_2,S_3$ are three pairwise distinct maximal simplices that contain a vertex $x$ and that span a prism;
	\item $((S_1,x_1), (S_2,x_1), (S_2,x_2), (S_3,x_2), (S_3,x_3), (S_4,x_3), (S_4,x_4), (S_1,x_4))$ where the vertices $x_1,\ldots, x_4$ are pairwise distinct and where $S_1,S_2$ span a prism such that $S_3$ (resp.\ $S_4$) is parallel to $S_1$ (resp.\ $S_2$). 
\end{itemize} 
\end{definition}

\noindent
As an illustration, the $3$- and $8$-cycles in the graphs of pointed simplices given by Figures~\ref{PS} and~\ref{C16} bound polygons.

\medskip \noindent
A natural generalisation would be to investigate the structure of $\mathsf{PS}_1(\text{prism})$ in order to define a higher dimensional complex structure on $\mathsf{PS}_1(W)$. Figure~\ref{C16} illustrates $\mathsf{PS}_1(3-\text{cube})$, which turns out to coincides with the one-skeleton of a convex polyhedron. More generally, it can be shown that $\mathsf{PS}_1(n-\text{cube})$ coincides with the one-skeleton of a convex polytope in $\mathbb{E}^n$, namely a \emph{truncated $n$-cube}. As a consequence, the graph of pointed edges of a cube complex can be naturally endowed with the structure of cellular complex. However, the situation is less clear when triangles are allowed. For instance, observe that the link of a vertex in $\mathsf{PS}(\text{triangle}^2)$ is a $3$-cycle, and not the expected complete graph $K_4$ for a convex polytope in $\mathbb{E}^4$. Anyway, no higher dimensional structure on $\mathsf{PS}(W)$ is required in the sequel, so we do not pursue further these questions and restrict ourselves to the following observation:
\begin{figure}
\begin{center}
\includegraphics[trim={0 15cm 43cm 0},clip,width=0.3\linewidth]{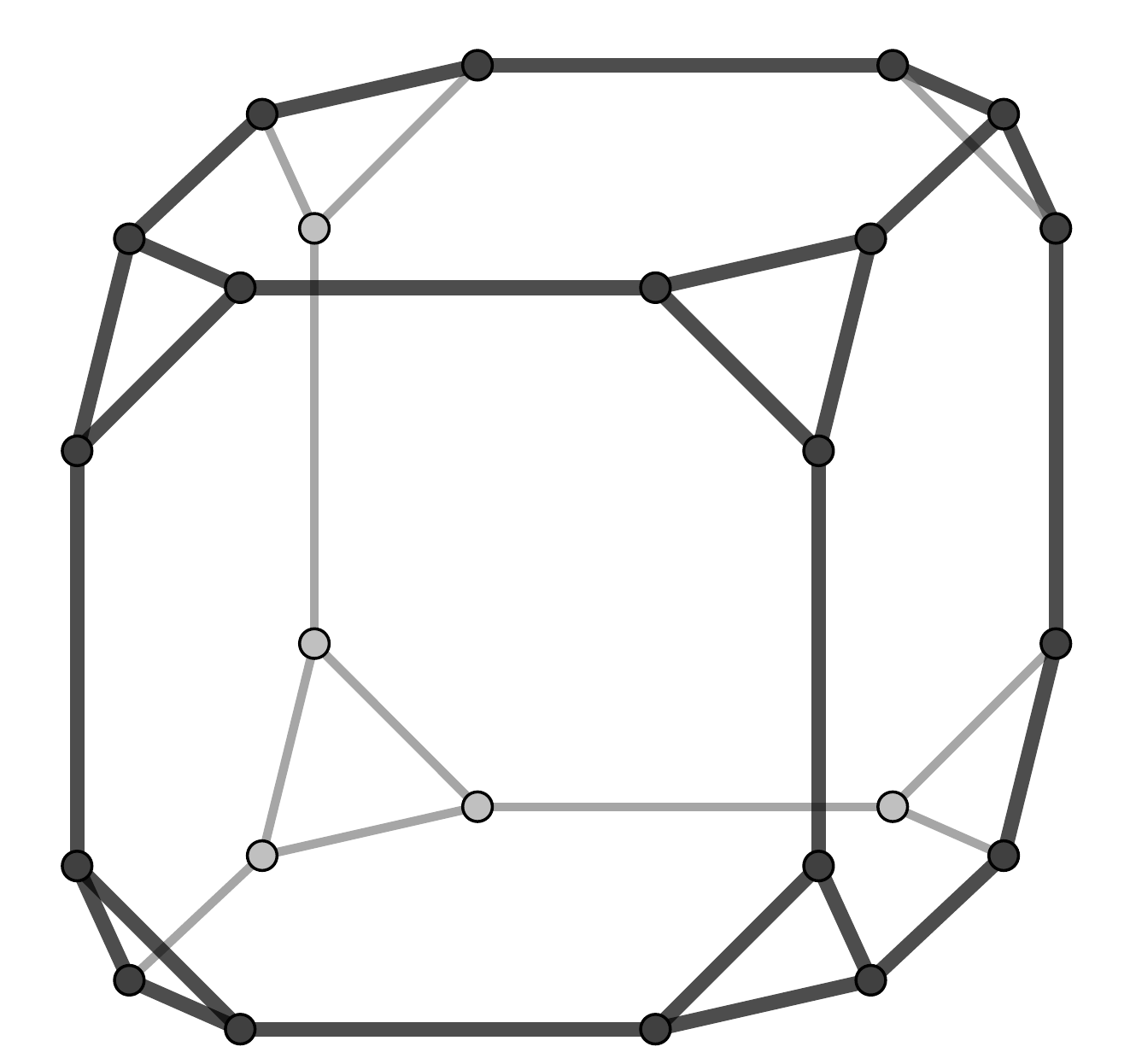}
\caption{The graph of pointed edges of a $3$-cube.}
\label{C16}
\end{center}
\end{figure}

\begin{lemma}\label{lem:PSconnected}
If $W$ is a prism, then $\mathsf{PS}(W)$ is a simply connected $2$-complex.
\end{lemma}

\begin{proof}
Because any two intersecting maximal simplices in $W$ span a prism, for every vertex $x \in W$ the subcomplex $\mathsf{PS}(x)$ spanned by $\{ (S,x) \mid \text{ $S$ maximal simplex}\}$, called \emph{leaf} below, coincides with the $2$-skeleton of a simplex, and so is simply connected. Notice that the subcomplexes $\mathsf{PS}(x)$ are pairwise disjoint, so we can collapse them without modifying the fundamental group of the space. The complex thus obtained coincides with the $2$-skeleton of $W$, which is simply connected.
\end{proof}

\noindent
Observe that, given a prism complex $W$, the canonical projection $\mathsf{PS}(W) \twoheadrightarrow W$ induced by $(S,x) \mapsto x$ is \emph{combinatorial}, i.e. it sends a cell to a cell. As shown by Proposition~\ref{prop:WreathIsPS}, a lamplighter graph can be described as a graph of pointed simplices of some prism complex, and, under such an identification, leaves of the lamplighter graph will correspond to the fibres of the previous projection. This motivates the following terminology: 

\begin{definition}
Let $W$ be a finite-dimensional prism complex. For every vertex $x \in W$, we refer to the subgraph (resp. the subcomplex) $\mathsf{PS}(x)$ generated by $\{(S,x) \mid \text{ $S$ simplex}\}$ in $\mathsf{PS}_1(W)$ (resp. $\mathsf{PS}(W)$) as a \emph{leaf}. 
\end{definition}

\noindent
Observe that $\mathsf{PS}(x)$ can be thought of as a \emph{link} of $x$ in $W$. Indeed, the vertices of $\mathsf{PS}(x)$ are given by the maximal simplices of $W$ containing $x$ and two such simplices are linked by an edge in $\mathsf{PS}(x)$ if they span a prism in $W$. In $\mathsf{PS}(W)$, the vertices of a leaf are obtained by only rotating a given pointed simplex around its distinguished vertex.

\medskip \noindent
Next, observe that our construction of complexes of pointed simplices is compatible with covering maps. More precisely:

\begin{lemma}\label{lem:PScovering}
Let $W_1,W_2$ be two finite-dimensional prism complexes and $\xi : W_1 \to W_2$ a covering map. Then the map $\overline{\xi} : \mathsf{PS}(W_1) \to \mathsf{PS}(W_2)$ induced by $(S,x) \mapsto (\xi(S),\xi(x))$ also defines a covering map. Moreover, it sends a leaf to a leaf.
\end{lemma}

\begin{proof}
It suffices to show that, given an arbitrary pointed simplex $(S,x)$, $\overline{\xi}$ induces an isomorphism from the link of $(S,x)$ to the link of $\xi(S,x)$. The link of $(S,x)$ is the graph whose vertices are $(S,x_1), \ldots, (S,x_p)$, where $x_1,\ldots, x_p$ are the vertices in $S \backslash \{x\}$, and  $(S_1,x), \ldots, (S_q,x)$, where $S_1, \ldots, S_q$ are the maximal simplices containing $x$ and spanning prisms with $S$; and whose edges connect $(S,x_i)$ with $(S,x_j)$ for all $1 \leq i<j \leq p$, $(S,x_i)$ with $(S_j,x)$ for all $1 \leq i \leq p$ and $1 \leq j \leq q$, and $(S_i,x)$ with $(S_j,x)$ whenever $S, S_i, S_j$ span a prism in $W_1$. The link of $\overline{\xi}(S,x)= (\xi(S),\xi(x))$ is described similarly. Because $\xi$ is a covering map, $\xi(x_1),\ldots, \xi(x_p)$ are the vertices in $\xi(S) \backslash \{\xi(x)\}$ and $\xi(S_1), \ldots, \xi(S_q)$ are the maximal simplices containing $\xi(x)$ and spanning prisms with $\xi(S)$. The desired conclusion follows.
\end{proof}

\subsection{Proof of the embedding theorem}\label{section:EmbeddingProof}

\noindent
Recall from Proposition~\ref{prop:WreathIsPS} that the lamplighter graph $\mathcal{L}_n(X)$ can be described as the graph of pointed simplices of some prism complex $W(n,X)$. The first ingredient towards the proof of Theorem~\ref{thm:FullEmbeddingThm} is that the universal cover $\widetilde{W}(n,X)$ of $W(n,X)$ has a very rigid structure, namely it turns out to be quasi-median complex. 

\begin{lemma}\label{lem:PSisQM}
$\widetilde{W}(n,X)$ is a quasi-median complex of cubical dimension $\mathrm{clique}(X)$.
\end{lemma}

\noindent
As a consequence, the hyperplanes of $\widetilde{W}(n,X)$ induce a wallspace structure on the graph $\mathsf{PS}_1(\widetilde{W}(n,X))$. The second ingredient towards the proof of Theorem~\ref{thm:FullEmbeddingThm} is that these walls satisfy convenient properties. First, the walls in $\mathsf{PS}_1(\widetilde{W}(n,X))$ have bounded images in $\mathsf{PS}_1(W(n,X))$:

\begin{lemma}\label{lem:HypareBounded}
For every hyperplane $J$ in $\widetilde{W}(n,X)$, the image of $\{(S,x) \mid S \subset J \} \subset \mathsf{PS}_1(\widetilde{W}(n,X))$ in $\mathsf{PS}_1(W(n,X))$ has bounded diameter. 
\end{lemma}

\noindent
And second, along a given wall, the metrics induced by $\widetilde{W}(n,X)$ and $\mathsf{PS}_1(\widetilde{W}(n,X))$ turn out to be biLipschitz equivalent:

\begin{lemma}\label{lem:BiLipEquivMetrics}
For every hyperplane $J$ in $\widetilde{W}(n,X)$ and pointed simplices $(S_1,x_1),(S_2,x_2) \in \mathsf{PS}_1(\widetilde{W}(n,X))$ satisfying $S_1,S_2 \subset J$, the inequalities
$$d_{\widetilde{W}} (x_1,x_2) \leq d_{\mathsf{PS}}((S_1,x_1),(S_2,x_2)) \leq 3 d_{\widetilde{W}}(x_1,x_2)$$
hold.
\end{lemma}

\noindent
We postpone the proofs of these three lemmas to the next section, and show how to deduce Theorem~\ref{thm:FullEmbeddingThm} from them.

\begin{proof}[Proof of Theorem \ref{thm:FullEmbeddingThm}.]
Without loss of generality, we assume that $\rho$ is continuous. According to Proposition~\ref{prop:WreathIsPS}, there exist a prism complex $W(n,X)$ and an isomorphism $\tau : \mathcal{L}_n(X) \to \mathsf{PS}_1(W(n,X))$ that sends leaves to leaves. Set $\eta:= \tau \circ \rho$. We think of $\mathsf{PS}_1(W(n,X))$ as the one-skeleton of $\mathsf{PS}(W(n,X))$. According to Lemma~\ref{lem:PScovering}, the universal cover $\widetilde{W}(n,X) \to W(n,X)$ induces a covering map $\pi : \mathsf{PS}(\widetilde{W}(n,X)) \to \mathsf{PS}(W(n,X))$.

\medskip \noindent
The first step of the proof is to notice that we can assume without loss of generality that the image of any loop in $Z$ under $\eta$ is homotopically trivial in $\mathsf{PS}(W(n,X))$.

\medskip \noindent
Let $R\geq 0$ be such that filling in cycles of $Z$ of length $\leq R$ with discs produces a simply connected $2$-complex. Let $L \geq 0 $ be such that the image under $\rho$ of every cycle of length $\leq R$ in $Z$ has diameter at most $L$ in $\mathcal{L}_n(X)$. Observe that $L$ depends only on $R$ and the parameters of $\rho$. Let $X'$ denote the graph obtained from $X$ by adding an edge between any two vertices at distance $\leq L$. The inclusion $X \hookrightarrow X'$ induces a $(L,0)$-quasi-isometry $q : \mathcal{L}_n(X) \to \mathcal{L}_n(X')$ that sends leaves to leaves. Therefore, up to replacing $X$ with $X'$ and $\rho$ with $q \circ \rho$, we can assume without loss of generality that
\begin{itemize}
	\item[$(\ast)$] for every cycle $\gamma \subset Z$ of length $\leq R$, there exist a colouring $\varphi \in \mathbb{Z}_n^{(X)}$ and a complete subgraph $Y \subset X$ such that $\rho(\gamma) \subset \{ (\psi,x) \mid x \in Y, \; \mathrm{supp}(\psi) \vartriangle \mathrm{supp}(\varphi)  \subset Y \}$.
\end{itemize}
As a consequence of $(\ast)$, $\eta$ sends every cycle of length $\leq R$ in $Z$ inside $\mathsf{PS}(P) \subset \mathsf{PS}(W(n,X))$ for some prism $P \subset W(n,X)$. We conclude from Lemma~\ref{lem:PSconnected} that $\eta(Z)$ is simply connected in $\mathsf{PS}(W(n,X))$, as desired.

\medskip \noindent
Therefore, $\eta : Z \to \mathsf{PS}(W(n,X))$ lifts to $\widetilde{\eta} : Z \to \mathsf{PS}(\widetilde{W}(n,X))$. So we have the following commutative diagram:

\begin{center}
\hspace{0cm} \xymatrix{
\mathcal{L}_n(X) \ar[r]^\tau & \mathsf{PS}_1(W(n,X)) \ar[r]^\iota & \mathsf{PS}(W(n,X)) & \ar[l]_\pi \mathsf{PS}(\widetilde{W}(n,X)) \\ Z \ar[u]^\rho \ar[ur]^\eta \ar[urrr]_{\widetilde{\eta}}  & & &
}
\end{center}

\noindent
According to Lemma~\ref{lem:CoarseLift}, $\widetilde{\eta}$ is a coarse embedding whose parameters depends only on those of $\rho$. Because the maps $\tau$ and $\pi$ send leaves to leaves, it suffices to show that $\widetilde{\eta}(Z)$ lies in the neighbourhood of a leaf in $\mathsf{PS}(\widetilde{W}(n,X))$ in order to conclude the proof of our theorem. Recall from Lemma~\ref{lem:PSisQM} that $\widetilde{W}(n,X)$ is a quasi-median complex of cubical dimension $\mathrm{clique}(X)$.

\begin{claim}\label{claim:Orientation}
There exists a constant $B \geq 0$ only depending on $X$, $Z$, and the parameters of $\rho$ such that the following holds. Every hyperplane $J$ in $\widetilde{W}(n,X)$ delimits a sector $H$ such that $\widetilde{\eta}(Z) \cap \{ \text{pointed simplices in $H$} \}$ is unbounded and $\widetilde{\eta}(Z) \cap \{ \text{pointed simplices in $H'$}\}$ has diameter $\leq B$ for every sector $H' \neq H$ delimited by~$J$.
\end{claim}

\noindent
Set $\mathcal{E}:= \{ (S,p) \in \widetilde{\eta}(Z) \mid p \in N(J), S \subset J\}$. The connected components of $\widetilde{\eta}(Z) \backslash \mathcal{E}$ are
$$\{ (S,p) \in \widetilde{\eta}(Z) \mid S \subset K\}, \text{ $K$ sector delimited by $J$}.$$
Our claim follows from the combination of the facts that $\pi(\mathcal{E})$ is bounded, according to Lemma~\ref{lem:HypareBounded}, and that $Z$ is uniformly one-ended.

\medskip \noindent
For every hyperplane $J$ of $\mathrm{QM}(\Gamma,X)$, let $J^+$ denote the sector given by Claim \ref{claim:Orientation}. 

\begin{claim}\label{claim:Vertex}
The intersection $\bigcap\limits_{\text{$J$ hyperplane}} J^+$ is reduced to a single vertex.
\end{claim}

\noindent
Our goal is to apply Lemma \ref{lem:InterSectors}.

\medskip \noindent
Let $J_1,J_2$ be two hyperplanes. If $J_1$ and $J_2$ are transverse, then clearly $J_1^+$ and $J_2^+$ intersect. Next, if $J_1$ and $J_2$ are not transverse and if $J_1^+$ contains $J_2$, then clearly $J_1^+$ and $J_2^+$ intersect. Finally, if $J_1$ and $J_2$ are not transverse and if $J_1^+$ does not contain $J_2$, then $J_2^+$ must be the sector delimited by $J_2$ that contains $J_1$ as a direct consequence of Claim \ref{claim:Orientation}. It follows that $J_1^+$ and $J_2^+$ intersect. 

\medskip \noindent
Next, assume for contradiction that $J_1^+ \supsetneq J_2^+ \supsetneq \cdots$ for some collection of hyperplanes $J_1,J_2, \ldots$ in $\widetilde{W}(n,X)$. Fix a vertex $(S,x) \in \widetilde{\eta}(Z)$ such that $S \subset J_1^+$. Because two vertices in $\widetilde{W}(n,X)$ are always separated by only finitely many hyperplanes, there must exist some $r \geq 1$ such that $S$ is disjoint from $J_r^+$. Next, fix an $s > r+B$ and a vertex $(Q,y) \in \widetilde{\eta}(Z)$ such that $Q \subset J_s^+$. As before, there exists  $t \geq s$ such that $Q$ is disjoint from $J_t^+$. We conclude that
$$B< s-r \leq d(x,y) \leq d((S,x),(Q,y)) \leq B,$$
where the last inequality is justified by Claim~\ref{claim:Orientation}, a contradiction.

\medskip \noindent
Lemma \ref{lem:InterSectors} completes the proof of Claim \ref{claim:Vertex}.

\medskip \noindent
Let $x$ denote the vertex of $\widetilde{W}(n,X)$ given by Claim \ref{claim:Vertex}, and consider the corresponding leaf of $ \mathsf{PS}_1(\widetilde{W}(n,X))$:
$$\mathcal{H}:=\{(T,x) \mid \text{$T$ maximal simplex containing $x$} \}.$$
In order to conclude the proof of our theorem, it suffices to show that every vertex in $\widetilde{\eta}(Z)$ lies at bounded distance from $\mathcal{H}$. 

\medskip \noindent
Let $(S,y)$ be a vertex in $\widetilde{\eta}(Z)$ and pick a geodesic path from $y$ to $x$ in the one-skeleton of $\widetilde{W}(n,X)$. Let $K$ denote the maximal simplex containing the last edge of that geodesic. Note that $(K,x)\in \mathcal{H}$. Because the hyperplane $J$ containing $K$ separates $y$ and $x$, it follows from the definition of $J^+$ that $J$ separates $y$ from a pointed simplex $(S',x')\in \widetilde{\eta} (Z)$. 
Recall that, since $\widetilde{\eta}$ is continuous, $\widetilde{\eta}(Z)$ is connected. Every path in $\widetilde{\eta}(Z)$ joining $(S,y)$ to $(S',x')$ must therefore cross $J$. We let $(Q,z) \in \widetilde{\eta}(Z)$ be the first vertex along that path such that  $Q \subset J$.
The initial segment of that path that ends just before hitting $J$ is contained in a single sector delimited by $J$. Since this sector is distinct from $J^+$, we deduce that $d((S,y),(Q,z)) \leq B+1$. Thanks to Claim~\ref{claim:InTwo} below, we have
$$\begin{array}{lcl} d((S,y),\mathcal{H}) & \leq & d((S,y),(K,x)) \leq d((S,y),(Q,z))+d((Q,z),(K,x))\\ \\ & \leq & B+1 + 3d(z,x) \leq 2(B+2) + \mathrm{clique}(X) \cdot (B+3), \end{array}$$
where penultimate inequality is justified by Lemma~\ref{lem:BiLipEquivMetrics}. This concludes the proof of our theorem.

\begin{claim}\label{claim:InTwo}
$d(z,x) \leq \mathrm{clique}(X) \cdot (B+3)+B+2$.
\end{claim}

\noindent
Let $J_1, \ldots, J_k$ be a maximal collection of pairwise non-transverse hyperplanes separating $y$ and $x$. Up to reindexing our collection, we assume that $J_i$ separates $J_{i-1}$ and $J_{i+1}$ for every $2 \leq i \leq k-1$ and that $J_1$ separates $y$ from $J_k$. Because $S$ is disjoint from $J_k^+$ and that $\widetilde{\eta}(Z)$ is connected, there exists a vertex $(P,w) \in \widetilde{\eta}(Z)$ such that $P \subset J_k$. According to Claim \ref{claim:Orientation}, $d((S,y),(P,w)) \leq B+2$. On the other hand, $d((S,y),(P,w)) \geq d(y,w)$. Because $y$ and $w$ are separated by $J_1, \ldots, J_{k-1}$, we deduce that
$$\begin{array}{lcl} d(y,x) & \leq & \mathrm{clique}(X) \cdot k \leq \mathrm{clique}(X) \cdot (d(y,w)+1) \\ \\ & \leq & \mathrm{clique}(X) \cdot (d((S,y),(P,w))+1) \leq \mathrm{clique}(X) \cdot (B+3), \end{array}$$
where the first inequality is justified by Lemma \ref{lem:MetricInf}. We also have
$$d(z,y) \leq d((Q,z),(C,y)) \leq B+2,$$
where the second inequality has been observed earlier. Therefore,
$$d(z,x) \leq d(z,y)+d(y,x) \leq \mathrm{clique}(X) \cdot (B+3) + B+2$$
as desired.
\end{proof}

\subsection{Proofs of the lemmas}\label{section:ProofsLemmas}

\noindent
This section is dedicated to the proofs of Lemmas \ref{lem:PSisQM}, \ref{lem:HypareBounded}, and \ref{lem:BiLipEquivMetrics}. Given a locally finite graph $X$ and an integer $n \geq 2$, the quasi-median structure claimed by Lemma~\ref{lem:PSisQM} of the universal cover $\widetilde{W}(n,X)$ of $W(n,X)$ can be easily shown by verifying on $W(n,X)$ the local condition given in \cite[Section~2.12]{QM}. However, it will be more convenient to identify $W(n,X)$ with the quotient of the quasi-median complex $\mathrm{QM}(X,\mathbb{Z}_n)$ by some specific subgroup of the graph product $X \mathbb{Z}_n$ in order to deduce easily Lemmas~\ref{lem:HypareBounded} and~\ref{lem:BiLipEquivMetrics}, thanks to the description of quasi-median complexes given in Section~\ref{section:QM}.

\begin{proof}[Proofs of Lemmas \ref{lem:PSisQM}, \ref{lem:HypareBounded}, and \ref{lem:BiLipEquivMetrics}.]
Consider the graph product $X\mathbb{Z}_n$ and its subgroup $K(X\mathbb{Z}_n)$ defined as the kernel of the morphism $\dag : X \mathbb{Z}_n \to \bigoplus_X \mathbb{Z}_n$ (which coincides with the epimorphism from $X\mathbb{Z}_n$ to its abelianisation). Observe that, for every complete subgraph $Y \subset X$, $\dag$ is injective on the subgroup $\langle Y \rangle \leq X \mathbb{Z}_n$. Consequently, we deduce from Lemma~\ref{lem:QMprisms} that $K(X \mathbb{Z}_n)$ intersects trivially the prism stabilisers of the action $X \mathbb{Z}_n \curvearrowright \mathrm{QM}(X, \mathbb{Z}_n)$. In other words, $K(X \mathbb{Z}_n)$ acts freely on $\mathrm{QM}(X, \mathbb{Z}_n)$, so the quotient map
$$\rho : \mathrm{QM}(X, \mathbb{Z}_n) \twoheadrightarrow \mathrm{qm}(X, \mathbb{Z}_n):= \mathrm{QM}(X, \mathbb{Z}_n)/ K(X \mathbb{Z}_n)$$
is a universal covering map. Let us observe that:

\begin{claim}
The prism complexes $\mathrm{qm}(X,\mathbb{Z}_n)$ and $W(n,X)$ are isomorphic.
\end{claim}

\noindent
Because $\dag$ is the quotient map $X\mathbb{Z}_n \twoheadrightarrow \bigoplus_X \mathbb{Z}_n$, quotienting by the kernel yields an isomorphism $X\mathbb{Z}_n / K(X\mathbb{Z}_n) \to \bigoplus_X \mathbb{Z}_n$. By identifying $\bigoplus_X \mathbb{Z}_n$ with $\mathbb{Z}_n^{(X)}$, we get a map $\ddag : \mathrm{qm}(X,\mathbb{Z}_n) \to W(n,X)$ that induces a bijection between the vertices. Because moving a vertex in $\mathrm{QM}(X,\mathbb{Z}_n)$ amounts to right-multiplying by a generator, moving a vertex in $\mathrm{qm}(X,\mathbb{Z}_n)$ amounts to modifying one coordinate in $\bigoplus_X \mathbb{Z}_n$. Therefore, moving a vertex in the image of $\ddag$ amounts to modifying a colouring at a single point. It follows that $\ddag$ induces an isomorphism between the one-skeleta of $\mathrm{qm}(X,\mathbb{Z}_n)$ and $W(n,X)$. Finally, spanning a prism in $\mathrm{QM}(X,\mathbb{Z}_n)$ amounts to right-multiplying by pairwise commuting generators, so spanning a prism in $\mathrm{qm}(X,\mathbb{Z}_n)$ amounts to modifying coordinates in $\bigoplus_X \mathbb{Z}_n$ indexed by pairwise adjacent vertices of $X$. Therefore, spanning a prism in the image of $\ddag$ amounts to modifying a colouring at pairwise adjacent vertices of $X$. We conclude that $\ddag$ induces an isomorphism between the prism complexes $\mathrm{qm}(X,\mathbb{Z}_n)$ and $W(n,X)$, as desired.

\medskip \noindent
As a consequence of the claim, we have a commutative diagram

\begin{center}
\hspace{0cm} \xymatrix{
\mathrm{QM}(X,\mathbb{Z}_n) \ar[r]^\sim \ar[d]_\rho & \widetilde{W}(n,X) \ar[d]^\pi \\ \mathrm{qm}(X,\mathbb{Z}_n) \ar[r]^\sim & W(n,X)
}
\end{center}

\noindent
Thus, it suffices to prove Lemmas \ref{lem:PSisQM}, \ref{lem:HypareBounded}, and \ref{lem:BiLipEquivMetrics} for $\mathrm{QM}(X,\mathbb{Z}_n) \twoheadrightarrow \mathrm{qm}(X,\mathbb{Z}_n)$. First, Lemma~\ref{lem:PSisQM} follows from Proposition~\ref{prop:QMisquasimedian}. Next, let $J$ be a hyperplane of $\mathrm{QM}(X,\mathbb{Z}_n)$. It follows from Lemma~\ref{lem:QMhyp} that the image under $\dag$ of the stabiliser of $J$ in $X\mathbb{Z}_n$ is finite, so the stabiliser of $J$ in $K(X \mathbb{Z}_n)$ has finite index in the stabiliser of $J$ in $X\mathbb{Z}_n$. Because the latter acts transitively on the vertices of $J$, Lemma~\ref{lem:HypareBounded} follows. 

\medskip \noindent
We now turn to the proof of Lemma \ref{lem:BiLipEquivMetrics}. The inequality $d_{\mathsf{PS}}((S_1,x_1),(S_2,x_2)) \geq d_{\mathrm{QM}}(x_1,x_2)$ is clear, so we focus on the other one. 
Let $(S_1,x_1),(S_2,x_2) \in \mathsf{PS}_1(\mathrm{QM}(X,\mathbb{Z}_n))$ be two pointed simplices such that $S_1$ and $S_2$ are contained in $J$. 
Fix a geodesic $y_1, \ldots, y_k \in \mathrm{QM}(X,\mathbb{Z}_n)$ from $x_1$ to $x_2$, and, for every $1 \leq i \leq k-1$, let $T_i$ denote the unique maximal simplex that contains the edge connecting $y_i$ and $y_{i+1}$. As a consequence of Theorem~\ref{thm:MainQM}, our geodesic lies in $N(J)$. So, for every $2 \leq i \leq k-1$, there exists a maximal simplex $U_i \subset J$ containing $y_i$. Observe that, as a consequence of Fact~\ref{fact:PrismGeneration}, $U_i$ spans two prisms with $T_{i-1}$ and $T_i$ for every $2 \leq i \leq k-1$, and, similarly, $S_1$ and $T_1$ (resp. $S_k$ and $T_{k-1}$) span a prism. 
It follows that
$$(S_1,x_1)=(S_1,y_1), (T_1,y_1), (T_1,y_2), (U_2,y_2), (T_2,y_2), \ldots, (T_{k-1},y_{k}), (S_2,y_k)=(S_2,x_2)$$
defines a path of length $\leq 3(k-1)$ in $\mathsf{PS}_1(\mathrm{QM}(X,\mathbb{Z}_n)$, hence the inequality 
$$d_{\mathsf{PS}}((S_1,x_1),(S_2,x_2)) \leq 3 \cdot d_{\mathrm{QM}}(x_1,x_2).$$
 This concludes the proofs of our lemmas.
\end{proof}

\section{Theorems of rigidity}

\subsection{Coarse simple connectivity of lamplighters}

\noindent
In this section, our goal is to distinguish geometrically lamplighters over one-ended and multi-ended groups. Our argument is based on the following characterisation of coarsely $1$-connected lamplighter graphs.

\begin{prop}\label{prop:LampGraph}
Let $X$ be a graph and $n \geq 2$ an integer. The lamplighter graph $\mathcal{L}_n(X)$ is coarsely simply connected if and only if $X$ is bounded.
\end{prop}

\begin{proof}
First, assume that $X$ is unbounded. So, for every $k \geq 1$, there exist two vertices $a_k,b_k \in X$ satisfying $d(a_k,b_k) \geq k$. Denote by $\alpha_k$ (resp. $\beta_k$) the colouring that is $1$ at $a_k$ (resp. $b_k$) and $0$ elsewhere. Also, denote by $Y_k \subset \mathcal{L}_n(X)$ the subgraph generated by $X \cup X(\alpha_k) \cup X(\beta_k) \cup X(\alpha_k + \beta_k)$. Notice that the map
$$\left\{ \begin{array}{ccc} \mathcal{L}_n(X) & \to & Y_k \\ (c,x) & \mapsto & \left( c_k : p \mapsto \left\{ \begin{array}{cl} c(a_k) & \text{if $p=a_k$} \\ c(b_k) & \text{if $p=b_k$} \\ 0 & \text{otherwise} \end{array} \right., \ x \right) \end{array} \right.$$
is $1$-Lipschitz. Therefore, if $\mathcal{L}_n(X)$ is coarsely simply connected, then the subgraphs $Y_k$ must be uniformly coarsely simply connected. However, $Y_k$ is the disjoint union of $X$, $X(\alpha_k)$, $X(\beta_k)$ and $X(\alpha_k+\beta_k)$, connected together by four edges: one between $(0,a_k)$ and $(\alpha_k,a_k)$, one between $(0,b_k)$ and $(\beta_k,b_k)$, one between $(\alpha_k,b_k)$ and $(\alpha_k+\beta_k,b_k)$, and one between $(\beta_k,a_k)$ and $(\alpha_k+\beta_k,a_k)$. Clearly, every simple loop of length $<4(d(a_k,b_k)+1)$ in $Y_k$ must lie inside $X$, $X(\alpha_k)$, $X(\beta_k)$ or $X(\alpha_k+\beta_k)$, so $Y_k$ is not $(4d(a_k,b_k)+3)$-coarsely simply connected, and a fortiori not $(4k+3)$-coarsely simply connected. We conclude that $\mathcal{L}_n(X)$ is not coarsely simply connected, as desired.

\medskip \noindent
Now, assume that $X$ is bounded. Let $P(X,n)$ denote the graph whose vertices are the finitely supported colourings $X \to \mathbb{Z}_n$ and whose edges connect two colourings if they differ at a single vertex. Notice that the distance between two colourings in $P(X,n)$ coincides with the number of vertices where they differ. The map
$$\left\{ \begin{array}{ccc} \mathcal{L}_n(X) & \to & P(X,n) \\ (c,x) & \mapsto & c \end{array} \right.$$
is a quasi-isometry since
$$d(c_1,c_2) \leq d((c_1,x_1),(c_2,x_2)) \leq \mathrm{diam}(X) \left( 2 d(c_1,c_2)+1 \right)$$
for all $(c_1,x_1),(c_2,x_2) \in \mathcal{L}_n(X)$. Therefore, it suffices to show that $P(X,n)$ is coarsely simply connected in order to deduce that $\mathcal{L}_n(X)$ is coarsely simply connected. But $P(X,n)$ is just (a connected component of) a product of infinitely many complete graphs (in other words, it is the $1$-skeleton of an infinite-dimensional prism), so it is $4$-coarsely simply connected.
\end{proof}

\begin{cor}\label{cor:NumberOfEnds}
Let $F_1,F_2$ be two finite groups and $H_1,H_2$ two finitely presented groups. If $H_1$ is one-ended and $H_2$ multi-ended, then $F_1 \wr H_1$ and $F_2 \wr H_2$ are not quasi-isometric.
\end{cor}

\begin{proof}
Assume towards a contradiction that there exists a quasi-isometry $q : F_1 \wr H_1 \to F_2 \wr H_2$. According to Theorem \ref{thm:MainThm}, $q(H_1)$ lies in a neighbourhood of some $H_2$-coset, say $H_2$ itself. In fact, $q(H_1)$ must lie in a neighbourhood of some coset of a one-ended factor $M \leq H_2$ coming from the Stallings-Dunwoody decomposition of $H_2$, say $M$ itself. Notice that, because there exists a Lipschitz quasi-retraction $H_2 \to M$, $M$ must be finitely presented. By applying Theorem \ref{thm:MainThm} once again, it follows that the image of $M$ under a quasi-inverse $\bar{q}$ of $q$ lies in a neighbourhood of some coset $hH_1$. Notice that $\bar{q}(q(H_1))$ lies in a neighbourhood of both $H_1$ and $hH_1$. But the intersection in $F_1 \wr H_1$ of two neighbourhoods of distinct $H_1$-cosets has finite diameter, so we must have $hH_1=H_1$. We also deduce that $q$ sends $H_1$ at finite Hausdorff distance from $M$. 

\medskip \noindent
As a consequence, $q$ induces a quasi-isometry from the space $X$ obtained from $F_1 \wr H_1$ by conning-off the cosets of $H_1$ and the space $Y$ obtained from $F_2 \wr H_2$ by conning-off the cosets of $M$ which are at finite Hausdorff distance from the images under $q$ of the cosets of $H_1$.  Observe that $X$ coincides with the lamplighter graph over the conning-off of $H_1$ over $H_1$ (which is bounded), and $Y$ with the lamplighter graph over the conning-off of $H_2$ over cosets of $M$ (which is unbounded since $H_2$ is multi-ended). Thus, we get a contradiction with Proposition \ref{prop:LampGraph}, proving that $F_1 \wr H_1$ and $F_2 \wr H_2$ cannot be quasi-isometric, as desired.
\end{proof}

\subsection{Proofs of the theorems}\label{section:Proofs}

\noindent
We are finally ready to prove the main result of this article, namely Theorem \ref{thm:LampRigid} from the introduction, as well as all its corollaries, by combining the various statements proved in Sections~\ref{bigsection:Aptolic},~\ref{section:CosetAptolic} and~\ref{bigsection:Embedding}. We begin by proving a quantitative version of Theorem \ref{thm:StructureQI}.

\begin{thm}\label{thm:CloseToAptolic}
Let $n,m \geq 2$ be two integers and $X,Y$ two coarsely $1$-connected uniformly one-ended graphs of bounded degree. For all $A,B \geq 0$, there exists a constant $Q \geq 0$ such that every $(A,B)$-quasi-isometry $\mathcal{L}_n(X) \to \mathcal{L}_m(Y)$ lies at distance $\leq Q$ from an aptolic quasi-isometry.
\end{thm}

\begin{proof}
Fix an $(A,B)$-quasi-isometry $q : \mathcal{L}_n(X) \to \mathcal{L}_m(Y)$ and one of its quasi-inverses $\bar{q} : \mathcal{L}_m(Y) \to \mathcal{L}_n(X)$ (whose parameters depend only of $A,B$). As a consequence of Theorem \ref{thm:FullEmbeddingThm}, for every leaf $L_1$ in $\mathcal{L}_n(X)$, the image $q(L_1)$ lies in $K$-neighbourhood of some leaf $L_2$ in $\mathcal{L}_m(Y)$ for some constant $K \geq 0$, which only depends on $X$, $Y$, $A$ and $B$. Similarly, $\bar{q}(L_2)$ lies in the $K$-neighbourhood of some leaf $L_3$ in $\mathcal{L}_n(X)$ (up to increasing the constant $K$). Consequently, $\bar{q}(q(L_1))$ lies in a neighbourhood of both $L_1$ and $L_3$. But the intersection in $\mathcal{L}_n(X)$ of two neighbourhoods of distinct leaves has finite diameter, so we must have $L_3=L_1$. We conclude that $q$ sends $L_1$ at Hausdorff distance at most $C$ from $L_2$, for some constant $C\geq 0$, which only depends on $X$, $Y$, $A$ and $B$. Thus, we have proved that $q$ sends every leaf of $\mathcal{L}_n(X)$ at Hausdorff distance $\leq C$ from a leaf of $\mathcal{L}_m(Y)$. In other words, there exists a map $\alpha : \mathbb{Z}_n^{(X)} \to \mathbb{Z}_m^{(Y)}$ such that the Hausdorff distance between $q(X(c))$ and $Y(\alpha(c))$ is $\leq C$ for every $c \in \mathbb{Z}_n^{(X)}$. 

\medskip \noindent
Similarly, there must exist a map $\bar{\alpha} : \mathbb{Z}_m^{(Y)} \to \mathbb{Z}_n^{(X)}$ such that the Hausdorff distance between $\bar{q}(Y(c))$ and $X(\bar{\alpha}(c))$ is $\leq C$ for every $c \in \mathbb{Z}_m^{(Y)}$. For every $c \in \mathbb{Z}_n^{(X)}$, the Hausdorff distance between $\bar{\alpha} \circ \alpha(X(c))$ and $Y(c)$ must be finite; and, for every $c \in \mathbb{Z}_m^{(Y)}$, the Hausdorff distance between $\alpha \circ \bar{\alpha} (Y(c))$ and $X(c)$ must be finite as well. Hence $\alpha \circ \bar{\alpha}= \mathrm{id}$ and $\bar{\alpha} \circ \alpha = \mathrm{id}$. In other words, $\alpha$ is a bijection.

\medskip \noindent
We conclude from Theorem \ref{thm:CosetAptolic} that $q$ lies at finite distance from an aptolic quasi-isometry, as desired, where the constant depends only on $A,B$. 
\end{proof}

\begin{proof}[Proof of Theorem \ref{thm:LampRigid}.]
Assume that there exists a quasi-isometry $q : \mathcal{L}_n(X) \to \mathcal{L}_m(Y)$. As a consequence of Theorem \ref{thm:StructureQI}, we can suppose without loss of generality that $q$ is aptolic. If $X$ is amenable, then the desired conclusion follows from Theorem \ref{thm:Amenable} and Proposition \ref{prop:AmConverse}.

\medskip \noindent
Next, assume that $X$ is non-amenable. According to Proposition \ref{prop:general}, $n$ and $m$ must have the same prime divisors; and according to Proposition \ref{prop:AptoQI}, $X$ and $Y$ must be quasi-isometric. Conversely, if $X$ and $Y$ are quasi-isometric, then they are biLipschitz equivalent according to Theorem \ref{thm:QIvsBil}, so $\mathcal{L}_n(X)$ and $\mathcal{L}_m(Y)$ are quasi-isometric; and, if we know that $n,m$ have the same prime divisors, then $\mathcal{L}_n(Y)$ and $\mathcal{L}_m(Y)$ are quasi-isometric according to Proposition \ref{prop:QInonamenable}. The desired conclusion follows.
\end{proof}

\begin{proof}[Proof of Theorem \ref{thm:QIfactors}.]
If $F_1$ is trivial, then $F_2 \wr H_2$ must be finitely presented (because quasi-isometric to the finitely presented group $H_1$) which implies that either $F_2$ is trivial or $H_2$ is finite. In both cases, $H_1$ and $H_2$ are quasi-isometric. The same conclusion holds if $F_2$ is trivial, so from now on we assume that $F_1$ and $F_2$ are both non-trivial. The conclusion is also clear if $H_1$ or $H_2$ is finite, so we assume that they are both infinite. We distinguish two cases.

\medskip \noindent
If $H_1$ is one-ended, it follows from Corollary \ref{cor:NumberOfEnds} that $H_2$ is one-ended as well. So Theorem~\ref{thm:StructureQI} applies and shows that there exists an aptolic quasi-isometry $F_1 \wr H_1 \to F_2 \wr H_2$. We conclude from Proposition \ref{prop:AptoQI}(ii) that $H_1$ and $H_2$ are quasi-isometric. 

\medskip \noindent
Next, assume that $H_1$ is multi-ended. According to Corollary \ref{cor:NumberOfEnds}, $H_2$ is multi-ended as well. If $H_1$ is two-ended (or equivalently if $H_1$ is virtually infinite cyclic), then $F_1 \wr H_1$ is amenable and so must be $F_2 \wr H_2$. Since infinitely-ended groups contain non-abelian free subgroups, necessarily $H_2$ must be two-ended (or equivalently virtually infinite cyclic). A fortiori, $H_1$ and $H_2$ are quasi-isometric. The same conclusion holds if $H_2$ is two-ended, so from now on we assume that $H_1$, $H_2$ are both infinitely-ended. According to \cite{MR1898396}, it suffices to show that $H_1$, $H_2$ have the same one-ended factors (up to quasi-isometry) in their Stallings-Dunwoody decompositions in order to deduce that they are quasi-isometric. Let $M \leq H_1$ be such a factor. Notice that, since there exists a Lipschitz quasi-retraction $H_1 \to M$, the subgroup $M$ must be finitely presented. Therefore, Theorem \ref{thm:MainThm} applies and shows that $q(M)$ lies in a neighbourhood of a coset of $H_2$, say $H_2$. In fact, $q(M)$ must lie in a neighbourhood of a coset of a one-ended factor $M'$ of $H_2$, say $M'$ itself. Similarly, the image of $M'$ under a quasi-inverse $\bar{q}$ lies in a neighbourhood of a coset $hM''$ of some one-ended factor of $H_1$. But the intersection in $F_1 \wr H_1$ of two neighbourhoods of distinct $H_1$-cosets has finite diameter, and the intersection in $H_1$ of two neighbourhoods of distinct cosets of one-ended factors has finite diameter as well, so necessarily $hM''=M$. In other words, $q$ sends $M$ at finite Hausdorff distance from $M'$. Thus, we have proved that every one-ended factor of $H_1$ is quasi-isometric to a one-ended factor of $H_2$. The converse follows by symmetry, proving that $H_1$ and $H_2$ are quasi-isometric, as desired. 
\end{proof}

\noindent
Now, we focus on the corollaries mentioned in the introduction.

\begin{proof}[Proof of Corollary \ref{cor:LampRigid}.]
Assume that there exists a quasi-isometry $q : F_1 \wr H_1 \to F_2 \wr H_2$. Notice that, as a consequence of Theorem \ref{thm:QIfactors}, $H_2$ must be one-ended. Therefore, Theorem \ref{thm:LampRigid} applies and leads to the desired conclusion.
\end{proof}

\begin{proof}[Proof of Corollary \ref{cor:Abelian}.]
According to Corollary \ref{cor:QIkappa}, it suffices to show that, for every $n \geq 1$, the inclusion $\mathbb{Q}_{>0} \subset \kappa(\mathbb{Z}^n)$ holds. Notice that, for every $k \geq 1$, the embedding $k \mathbb{Z} \oplus \mathbb{Z}^{n-1} \hookrightarrow \mathbb{Z}^n$ induces a quasi-isometry $q_k : \mathbb{Z}^n \to \mathbb{Z}^n$ that is quasi-$(1/k)$-to-one (apply for instance Proposition \ref{prop:kappa}(iii)); let $\bar{q}_k$ denote a quasi-inverse of $q_k$. It follows from Proposition \ref{prop:EasyKappa} that, for every $a,b \geq 1$, the quasi-isometry $\bar{q}_a \circ q_b$ is quasi-$(a/b)$-to-one, concluding the proof.
\end{proof}

\begin{proof}[Proof of Corollary \ref{cor:NotBiL}.]
If $|F_1|=|F_1|$, then $F_1 \wr H$ and $F_2 \wr H$ admits isomorphic Cayley graphs, namely $\mathrm{Cayl}(F_1 \wr H, F_1 \cup S)$ and $\mathrm{Cayl}(F_2 \wr H, F_2 \cup S)$ where $S$ is an arbitrary finite generating set of $H$. A fortiori, these wreath products must be biLipschitz equivalent. Conversely, assume that there exists a biLipschitz equivalence $q : F_1 \wr H \to F_2 \wr H$. According to Theorem \ref{thm:StructureQI}, $q$ is at finite distance from an aptolic quasi-isometry $\tilde{q} : F_1 \wr H \to F_2 \wr H$; and, according to Theorem \ref{thm:Amenable}, there exist $k,n_1,n_2 \geq 1$ such that $|F_1|=k^{n_1}$, $|F_2|=k^{n_2}$, and such that $\tilde{q}$ is quasi-$(n_2/n_1)$-to-one. We conclude from Lemma \ref{lem:KappaWellDefined} and Proposition \ref{prop:EasyKappa} that $n_1=n_2$, hence $|F_1|=|F_2|$ as desired.
\end{proof}

\begin{proof}[Proof of Corollary \ref{cor:kappatrivial}.]
Let $q : F \wr H \to F\wr H$ be a quasi-isometry. According to Theorem \ref{thm:StructureQI}, $q$ lies at finite distance from an aptolic quasi-isometry $\tilde{q}  : F \wr H \to F \wr H$; and it follows from Theorem \ref{thm:Amenable} that $\tilde{q}$ is quasi-one-to-one. We conclude from Proposition \ref{prop:QIdistBij} that $\tilde{q}$, and a fortiori $q$, lies at finite distance from a bijection. 
\end{proof}

\noindent
We now turn our attention to Corollary \ref{cor:Commensurability}. Actually, we are going to prove a more general statement, but we need to introduce some vocabulary first. Given two finitely generated groups $A,B$, a \emph{biLipschitz commensurability} is the data of two finite-index subgroups $\dot{A} \leq A$, $\dot{B} \leq B$ and a biLipschitz equivalence $\varphi : \dot{A} \to \dot{B}$. When $\varphi$ is an isomorphism, we recover the usual notion of commensurability. The \emph{index} of a biLipschitz commensurability is the quotient $[A:\dot{A}]/[B:\dot{B}]$. 

\begin{prop}\label{prop:commensurable}
Let $F$ be a non-trivial finite group and $H$ a finitely presented one-ended amenable group. Fix two groups $G_1,G_2$ in the biLipschitz commensurability class of $F \wr H$. If $G_1$ and $G_2$ are biLipschitz equivalent, then every biLipschitz commensurability between $G_1$ and $G_2$ has index one.
\end{prop}

\begin{proof}
By assumption, there exist finite-index subgroups $\dot{G}_1 \leq G_1$, $L \leq F\wr H$, $K_1 \leq G_1$, $K_2 \leq G_2$, and biLipschitz equivalences $\varphi : L \to \dot{G}_1$, $\phi : K_1 \to K_2$, $\psi : G_2 \to G_1$. For every group $B$ and every finite-index subgroup $A \leq B$, we denote by $\iota_{A,B}$ the inclusion $A \hookrightarrow B$ and we fix a quasi-inverse $\bar{\iota}_{A,B} : B \to A$. Observe that $\iota_{A,B}$ is quasi-$(1/[B:A])$-to-one and $\bar{\iota}_{A,B}$ quasi-$[B:A]$-to-one. (For instance, apply Proposition \ref{prop:kappa}(iii).) As a consequence of Proposition \ref{prop:EasyKappa}, the quasi-isometry $F\wr H \to F \wr H$ defined by
$$\Psi := \iota_{L,F\wr H} \circ \varphi^{-1} \circ \bar{\iota}_{\dot{G}_1,G_1} \circ \left( \psi \circ \phi \circ \bar{\iota}_{K_1,G_1}  \right) \circ \iota_{\dot{G}_1,G_1} \circ \varphi \circ \bar{\iota}_{L,F\wr H}$$
is quasi-$([G_1:K_1]/[G_2:K_2])$-to-one. The combination of Corollary \ref{cor:kappatrivial} and Lemma~\ref{lem:KappaWellDefined} implies that $[G_1:K_1]=[G_2:K_2]$ as desired.
\end{proof}

\subsection{Permutational wreath products}

\noindent
It is worth noticing that Theorem \ref{thm:LampRigid} also applies to another kind of wreath products.

\begin{definition}
Let $F,H$ be two groups and $S$ a set on which $H$ acts. The \emph{permutational product} of $F,H$ with respect to $H \curvearrowright S$ is
$$F \wr_S H:= \left( \bigoplus\limits_S F \right) \rtimes H$$
where $H$ acts on the direct sum by permuting the coordinates through its action on $S$.
\end{definition}

\noindent
Observe that, if $H$ acts freely and transitively on $S$, then $F \wr_S H$ coincides with $F \wr H$. Corollary \ref{cor:PermutationalW} below shows that the classification provided by Corollary \ref{cor:LampRigid} extends to the case where $H$ acts on $S$ transitively and with finite stabilisers, which coincides with a permutational wreath product $F \wr_{H/L} H$ for some finite subgroup $L \leq H$.

\begin{cor}\label{cor:PermutationalW}
Let $F_1,F_2$ be two finite groups, $H_1,H_2$ two finitely presented one-ended groups, and $X_1,X_2$ two sets on which $H_1,H_2$ respectively act with finitely many orbits (say $m_1,m_2$) and with finite stabilisers. For every $1 \leq i \leq m_1$ (resp. $1 \leq i \leq m_2$), let $p_i$ (resp. $q_i$) denote the size of point-stabilisers in the $i$th $H_1$-orbit of $X_1$ (resp. in the $i$th $H_2$-orbit of $X_2$). 
\begin{itemize}
	\item If $H_1$ is amenable, then $F_1 \wr_{X_1} H_1$ and $F_2 \wr_{X_2} H_2$ are quasi-isometric if and only if $|F_1|=k^{n_1}$, $|F_2|=k^{n_2}$ for some $k,n_1,n_2 \geq 1$ and if there exists a quasi-$\kappa$-to-one quasi-isometry $H_1 \to H_2$ where $\kappa:=\frac{n_2}{n_1} \left( \frac{1}{p_1}+ \cdots + \frac{1}{p_{m_1}} \right) \left( \frac{1}{q_1} + \cdots + \frac{1}{q_{m_2}} \right)^{-1}$.
	\item If $H_1$ is non-amenable, then $F_1 \wr_{X_1} H_1$ and $F_2 \wr_{X_2} H_2$ are quasi-isometric if and only if $|F_1|,|F_2|$ have the same prime divisors and $H_1,H_2$ are quasi-isometric.
\end{itemize}
\end{cor}

\noindent
The assertion will be an easy consequence of the following observation:

\begin{lemma}\label{lem:QIperm}
Let $F$ be a finite group, $H$ a group generating by finite set $S$, and $X$ a set on which $H$ acts with finite stabilisers and with finitely many (say $n$) orbits. Then there exist a graph $Y$ and a quasi-$\left( \sum\limits_{i=1}^n \frac{1}{\ell_i} \right)$-to-one quasi-isometry $H \to Y$ such that $\mathrm{Cayl}(F \wr_{X} H, F \cup S)$ and $\mathcal{L}_{|F|}( Y)$ are quasi-isometric, where $\ell_i$ denotes the size of point-stabilisers in the $i$th $H$-orbit of $X$. 
\end{lemma}

\begin{proof}
It is clear that the graphs $\mathrm{Cayl}(F\wr_X H, F \cup S)$ and $\mathrm{Cayl}(\mathbb{Z}_{|F|} \wr_X H, \mathbb{Z}_{|F|} \cup S)$ are isomorphic. Consequently, from now on we assume that $F$ is a cyclic group, say of order $r$. Let $x_1, \ldots, x_n \in X$ be representatives modulo the $H$-action. For every $1 \leq i \leq n$, let $L_i$ denote $\mathrm{stab}_H(x_i)$. We define $Y$ as the graph
\begin{itemize}
	\item whose vertex-set is the disjoint union $H/L_1 \sqcup \cdots \sqcup H/L_n$;
	\item whose edges link $hL_i,hs^{\pm 1}L_i$ for all $1 \leq i \leq n$, $h \in H$, $s \in S$ and $hL_i,hL_j$ for all $h \in H$, $1 \leq i,j \leq n$.
\end{itemize}
The graph $Y$ is clearly quasi-isometric to $H$. Fix an arbitrary map $\sigma : H/L_1 \sqcup \cdots \sqcup H/L_n \to H$ satisfying $\sigma(hL_i) \in hL_i$ for all $h \in H$ and $1 \leq i \leq n$. Finally, notice that the map $\beta : Y \to X$ defined by $hL_i \mapsto h \cdot x_i$ ($h \in H$ and $1 \leq i \leq n$) is clearly a bijection. We claim that
$$\Phi : \left\{ \begin{array}{ccc} \mathcal{L}_r(Y) & \to & \mathbb{Z}_r \wr_X H \\ (c,hL_i) & \mapsto & \left( c \circ \beta^{-1}, \sigma(hL_i) \right) \end{array} \right. \text{ and } \Psi : \left\{ \begin{array}{ccc} \mathbb{Z}_r \wr_X H & \to & \mathcal{L}_r(Y) \\ (c,h) & \mapsto & \left( c \circ \beta, hL_1 \right) \end{array} \right.$$
are quasi-isometries, quasi-inverses of each other. Let $a,b \in \mathbb{Z}_r \wr_X H$ be two adjacent vertices (in our Cayley graph). Two cases may happen:
\begin{itemize}
	\item There exist $c \in \mathbb{Z}_r^{(X)}$, $h \in H$, $s \in S$ such that $a=(c,h)$ and $b=(c,hs^{\pm 1})$. Then $\Psi(a)=(c \circ \beta,hL_1)$ and $\Psi(b)=(c \circ \beta,hs^{\pm 1}L_1)$ are either identical or adjacent in~$Y$.
	\item There exist $c \in \mathbb{Z}_r^{(X)}$, $h \in H$, $1 \leq i \leq n$ and $d \in \mathbb{Z}_r^{(X)}$ supported at $h \cdot x_i$ such that $a=(c,h)$ and $b=(c+d,h)$. Then $\Psi(a)=(c \circ \beta, hL_1)$ and $\Psi(b)=(c \circ \beta + d',hL_1)$, where $d' \in \mathbb{Z}_r^{(Y)}$ is supported at $hL_i$, are two adjacent vertices of $Y$.
\end{itemize}
Thus, we have proved that $\Psi$ sends an edge to either a point or an edge, which implies that it is $1$-Lipschitz. Similarly, given two adjacent vertices $a,b \in Y$, three cases may happen:
\begin{itemize}
	\item There exist $c \in \mathbb{Z}_r^{(Y)}$, $h \in H$, $1 \leq i,j \leq n$ such that $a=(c,hL_i)$ and $b=(c,hL_j)$. Then $\Phi(a)=(c \circ \beta^{-1}, \sigma(hL_i))$ and $\Phi(b)=(c \circ \beta^{-1}, \sigma(hL_j))$ are at distance at most $\mathrm{diam}_H(L_i \cup L_j)$ in $\mathbb{Z}_r \wr_X H$. 
	\item There exist $c \in \mathbb{Z}_r^{(Y)}$, $h \in H$, $1 \leq i \leq n$, $s \in S$ such that $a=(c,hL_i)$ and $b=(c,hs^{\pm 1}L_i)$. Then $\Phi(a)=(c \circ \beta^{-1}, \sigma(hL_i))$ and $\Phi(b)=(c \circ \beta^{-1}, \sigma(hs^{\pm 1}L_i))$ are at distance at most $\mathrm{diam}_H(L_i \cup s^{\pm 1}L_j)$ in $\mathbb{Z}_r \wr_X H$.
	\item There exist $c \in \mathbb{Z}_r^{(Y)}$, $h \in H$, $1 \leq i \leq n$ and $d \in \mathbb{Z}_r^{(Y)}$ supported at $hL_i$ such that $a=(c,hL_i)$ and $b=(c+d,hL_i)$. Then $\Phi(a)=(c \circ \beta^{-1},\sigma(hL_i))$ and $\Phi(b)=(c \circ \beta^{-1} + d', \sigma(hL_i))$, where $d' \in \mathbb{Z}_r^{(Y)}$ is supported at $h \cdot x_i$, are at distance at most $1+ \mathrm{diam}_H(L_i)$ in $\mathbb{Z}_r \wr_X H$.
\end{itemize}
Thus, we have also proved that $\Phi$ is Lipschitz. Finally, observe that
$$\Phi \circ \Psi(c,h)=(c, \sigma(hL_1)) \text{ for all $(c,h) \in \mathbb{Z}_r \wr_X H$,}$$
proving that $\Phi \circ \Psi$ lies at distance $\leq \mathrm{diam}_H(L_1)$ from the identity; and that
$$\Psi \circ \Phi(c,hL_i) = (c, \sigma(hL_i)L_1) \text{ for all $c \in \mathbb{Z}_r^{(Y)}$, $h \in H$ and $1 \leq i \leq n$,}$$
proving that $\Psi \circ \Phi$ lies at distance $\leq 1+\mathrm{diam}_H(L_1)$ from the identity. We conclude that $\Phi, \Psi$ are quasi-isometries and that $\Phi$ is a quasi-inverse of $\Psi$ (and vice-versa). 

\medskip \noindent
Now, set $\kappa:= (1/|L_1| + \cdots + 1/|L_n|)^{-1}$. We claim that the map $f : Y \to H$ defined by $hL_i \mapsto \sigma(hL_i)$, which is clearly a quasi-isometry, is quasi-$\kappa$-to-one. As a consequence of Proposition \ref{prop:EasyKappa}, this will imply that there exists a quasi-$(1/\kappa)$-to-one quasi-isometry $H \to Y$, as desired.

\medskip \noindent
For every $1 \leq i \leq n$, let $f_i : H/L_i \to H$ denote the restriction of $f$ to the subgraph $H/L_i \subset Y$. Observe that $f_i$ is a quasi-inverse of the $|L_i|$-to-one projection $H \to H/L_i$ (where $H/L_i$ is thought of as a subgraph of $Y$), so $f_i$ is quasi-$(1/|L_i|)$-to-one according to Proposition \ref{prop:EasyKappa}; and that, for every finite subset $A \subset H$, $f^{-1}(A)$ coincides with the disjoint union $f^{-1}_1(A) \sqcup \cdots \sqcup f^{-1}_n(A)$. Therefore,
$$\begin{array}{lcl} \displaystyle \left| \kappa |f^{-1}(A)| - |A| \right| & = & \displaystyle \kappa \left| \sum\limits_{i=1}^n |f_i^{-1}(A)| - \sum\limits_{i=1}^n \frac{|A|}{|L_i|} \right| \leq \kappa \sum\limits_{i=1}^n \left| |L_i| \cdot |f_i^{-1}(A) | - |A| \right| \\ \\ & \leq & \displaystyle \kappa \sum\limits_{i=1}^n \frac{C_i}{|L_i|} \cdot | \partial A| \end{array}$$
for some constants $C_1, \ldots, C_n$ that do not depend on $A$. Thus, we have shown that $f$ is quasi-$\kappa$-to-one, concluding the proof of our lemma.
\end{proof}

\begin{proof}[Proof of Corollary \ref{cor:PermutationalW}.]
According to Lemma~\ref{lem:QIperm}, there exist a graph $Y_1$ and a quasi-$\kappa_1$-to-one quasi-isometry $H_1 \to Y_1$, where $\kappa_1:= 1/p_1 + \cdots + 1/p_{m_1}$, such that $F_1 \wr_{X_1} H_1$ and $\mathcal{L}_{|F_1|}(Y_1)$ are quasi-isometric; and there exist a graph $Y_2$ and a quasi-$\kappa_2$-to-one quasi-isometry $H_2 \to Y_2$, where $\kappa_2 := 1/q_1+ \cdots +1/q_{m_2}$, such that $F_2 \wr_{X_2} H_2$ and $\mathcal{L}_{|F_2|}(Y_2)$ are quasi-isometric. Observe that, as a consequence of Proposition \ref{prop:EasyKappa}, there exists a quasi-$(n_2/n_1)$-to-one quasi-isometry $Y_1 \to Y_2$ if and only if there exist a quasi-$(n_2\kappa_1/n_1\kappa_2)$-to-one quasi-isometry $H_1 \to H_2$. The desired conclusion follows from Theorem~\ref{thm:LampRigid}.
\end{proof}

\begin{remark}\label{rem:locallycompact}
Let us indicate a further generalisation. Assume that $H$ is a locally compact group and that $X$ is a set on which $H$ acts with open stabilisers. If $F$ is a discrete group, then we observe that $F \wr_{X} H$ is a locally compact group. It turns out that  Corollary \ref{cor:PermutationalW} holds under the assumption that $H_1$ and $H_2$ are locally compact, compactly presented, one-ended, and assuming that the stabilisers are compact open (instead of finite). The statements and proofs are identical, using the definition of quasi-$\kappa$-to-one quasi-isometries between locally compact groups given in \cite{Kappa}. We leave the details to the reader. 
\end{remark}

\section{Further results and open questions}\label{section:Last}

\noindent
The main question addressed in the article was the following:

\begin{question}\label{QOne}
Let $F_1,F_2$ be two non-trivial finite groups and $H_1,H_2$ two finitely generated groups. When are $F_1 \wr H_1$ and $F_2 \wr H_2$ quasi-isometric?
\end{question}

\noindent
Let us discuss the state of this problem regarding the results proved in this article and the rest of the literature. We distinguish different cases according to the number of ends of $H_1$ and $H_2$.

\paragraph{Two-ended case.} When $H_1$ and $H_2$ are both infinite cyclic, the problem was solved by Eskin, Fisher, and Whyte: $F_1 \wr \mathbb{Z}$ and $F_2 \wr \mathbb{Z}$ are quasi-isometric if and only if $|F_1|$ and $|F_2|$ are powers of a common number \cite{EFWI, EFWII}. In fact, because virtually infinite cyclic groups are always biLipschitz equivalent, their theorem answers Question~\ref{QOne} when $H_1$ and $H_2$ are both two-ended. Moreover, as a consequence of the following observation, assuming that only $H_1$ is two-ended suffices.

\begin{prop}\label{prop:cones}
Let $F$ be a non-trivial finite group and $H$ a finitely generated group. The asymptotic cones of $F \wr H$ have finite topological dimension if and only if $H$ is virtually cyclic.
\end{prop}

\noindent
We were informed by private communication that this phenomenon has been also noticed by Y. Cornulier, independently.

\begin{proof}[Proof of Proposition \ref{prop:cones}.]
If $H$ is finite then there is nothing to prove, so from now on we assume that $H$ is infinite. We distinguish two cases. First, assume that $H$ has linear growth. In other words, $H$ is virtually cyclic, and it must be biLipschitz equivalent to $\mathbb{Z}$ so that $F \wr H$ is quasi-isometric to $F \wr \mathbb{Z}$. It is well-known that $F \wr \mathbb{Z}$ quasi-isometrically embed into a product of two simplicial trees, so the asymptotic cones of $F \wr \mathbb{Z}$ topologically embed into a product of two real trees. We conclude that the asymptotic cones of $F \wr \mathbb{Z}$ have finite topological dimension.

\medskip \noindent
Next, assume that $H$ has super-linear growth. Fix an ultrafilter $\omega$ over $\mathbb{N}$, a sequence of basepoints $o=(o_n)$, and a sequence of scaling factors $\lambda=(\lambda_n)$. Without loss of generality, we assume that $\lambda_n \in \mathbb{N}$ for every $n \geq 0$. Given a $k \geq 1$, our goal is to construct a topological embedding of $[0,1]^k$ into $\mathrm{Cone}_\omega(F\wr H,o,\lambda)$. Taking $k$ arbitrarily large will prove that the topological dimension of the asymptotic cone is infinite.

\medskip \noindent
For every $n \geq 0$, let $R_n$ denote the smallest integer such that the ball $B(o_n,R_n)$ has cardinality $\geq k \lambda_n$. Notice that $(R_n/\lambda_n)$ tends to zero as $n \to + \infty$ because the fact that $H$ has super-linear growth implies that
$$\frac{R_n}{\lambda_n}= \frac{R_n}{ |B(o_n,R_n-1)|} \cdot \frac{| B(o_n,R_n-1)|}{\lambda_n} \leq k \frac{R_n}{ |B(o_n,R_n-1)|} \underset{n \to + \infty}{\longrightarrow} 0.$$
Now, fix an index $n \geq 0$. Because $B(o_n,R_n)$ has size $\geq k\lambda_n$, there exist $k$ pairwise disjoint subsets $L_n(1), \ldots, L_n(k) \subset B(o_n,R_n)$ of size $\lambda_n$. For every $1 \leq i \leq k$, we fix an enumeration $L_n(i)= \{\ell_n^1(i), \ldots, \ell_n^{\lambda_n}(i) \}$. Define
$$\Psi : \left\{ \begin{array}{ccc} [0,\lambda_n]^k & \to & F\wr H \\ (r_1,\ldots, r_k) & \mapsto & \left( h \mapsto \left\{ \begin{array}{cl} 1& \text{if $h=\ell_n^j(i)$ for some $1 \leq i \leq k$, $1 \leq j \leq r_i$} \\ 0 & \text{otherwise} \end{array} \right., \ o_n \right) \end{array} \right..$$
Observe that
$$\sum\limits_{i=1}^k |r_i-s_i| \leq d \left( \Psi(r), \Psi(s) \right) \leq \sum\limits_{k=1}^k |r_i-s_i| + 2kR_n$$
for all $r=(r_i),s=(s_i) \in [0,\lambda_n]^k$. As a consequence, $\Psi$ induces a biLipschitz embedding
$$\Psi_\infty : \underset{\simeq [0,1]^k}{\underbrace{\mathrm{Cone}_\omega \left( [1,\lambda_n]^k, 0, \lambda \right)}} \hookrightarrow \mathrm{Cone}_\omega (F \wr H, o,\lambda),$$
concluding the proof of our proposition.
\end{proof}

\noindent
In summary, Question \ref{QOne} is completely solved in the two-ended case:

\begin{thm}
Let $F_1,F_2$ be two finite groups and $H_1,H_2$ two finitely generated groups. Assume that $H_1$ is two-ended. The groups $F_1 \wr H_1$ and $F_2 \wr H_2$ are quasi-isometric if and only if $H_2$ is two-ended and $|F_1|,|F_2|$ are powers of a common number. 
\end{thm}

\noindent
As a consequence of Proposition \ref{prop:cones}, the two-endedness of $H$ can be detected from the asymptotic geometry of $F\wr H$. And, according to Theorem \ref{thm:QIfactors}, the number of ends of $H$ can also be detected if we restrict ourselves to finitely presented groups. So a natural question is the following:

\begin{question}
Let $F_1,F_2$ be two non-trivial finite groups and $H_1,H_2$ two finitely generated groups. If $F_1 \wr H_1$ and $F_2 \wr H_2$ are quasi-isometric, do $H_1$ and $H_2$ have the same number of ends?
\end{question}

\paragraph{One-ended case.} For finitely presented groups, Theorem \ref{thm:LampRigid} provides a complete answer to Question \ref{QOne} in the one-ended case. Now, a natural problem is to determine what happens for infinitely presented groups. For instance, does Theorem \ref{thm:QIfactors} still holds? More precisely:

\begin{question}
Do there exist finitely generated groups $H_1$ and $H_2$ that are not quasi-isometric but such that $\mathbb{Z}_2 \wr H_1$ and $\mathbb{Z}_2 \wr H_2$ are quasi-isometric?
\end{question}

\noindent
In this perspective, iterated wreath products are examples of interest. For instance:

\begin{question}
Let $H$ be a finitely presented group and $n,m,p,q \geq 2$ four integers. When are $\mathbb{Z}_n \wr (\mathbb{Z}_m \wr H)$ and $\mathbb{Z}_p \wr (\mathbb{Z}_q \wr H)$ quasi-isometric?
\end{question}

\noindent
In particular, the cases $H=\mathbb{Z}$ and $H= \mathbb{Z}^2$ would be already interesting. Observe that iterated wreath products with different numbers of factors can be often distinguished. For instance, if $A$ is a finitely generated abelian group, then $\mathbb{Z}_2 \wr ( \mathbb{Z}_2 \wr A)$ and $\mathbb{Z}_2 \wr A$ are not quasi-isometric since they have different isoperimetric profiles (see \cite{MR2011120} for more information). However, it is not clear that such a distinction is always possible.

\begin{question}
Does there exist a finitely presented group $H$ such that $\mathbb{Z}_2 \wr (\mathbb{Z}_2 \wr H)$ and $\mathbb{Z}_2 \wr H$ are quasi-isometric? $\mathbb{Z}_2 \wr (\mathbb{Z}_2 \wr (\mathbb{Z}_2 \wr H))$ and $\mathbb{Z}_2 \wr (\mathbb{Z}_2 \wr H)$?
\end{question}

\paragraph{Infinitely-ended case.} We are not aware of any answer, even partial, regarding Question \ref{QOne} in the infinitely-ended case. The situation is different from \cite{EFWI, EFWII} since, as proved by Proposition \ref{prop:QInonamenable}, if $H$ is infinitely-ended and if $F_1,F_2$ are two finite groups whose cardinalities have the same prime divisors, then $F_1 \wr H$ and $F_2 \wr H$ are quasi-isometric. But our arguments do not seem to provide any valuable information either. The main reason is that, in the multi-ended case, quasi-isometries may not be at finite distance from aptolic quasi-isometries, as shown by Example \ref{ex:NonAptolic}. Even worse, it can be shown thanks to the quasi-isometries given by Proposition~\ref{prop:NonAptolic} that there exists a sequence of (uniform) quasi-isometries that are not at finite distance from aptolic quasi-isometries pointwise converging to the identity. As a consequence, quasi-isometries that are not at finite distance from aptolic quasi-isometries define a dense $G_\delta$ in the topological space of quasi-isometries (endowed with the topology of pointwise convergence). Loosely speaking, non-aptolic quasi-isometries are generic in the multi-ended case.

\medskip \noindent
The following problem should be the next step towards a full understanding of the asymptotic geometry of lamplighter groups:

\begin{question}
Let $\mathbb{F}$ be a finitely generated free group and $n,m \geq 2$ two integers. When are $\mathbb{Z}_n \wr \mathbb{F}$ and $\mathbb{Z}_m \wr \mathbb{F}$ quasi-isometric?
\end{question}

\paragraph{Quasi-isometric rigidity.} A natural question following our quasi-isometric classification is: given a finite group $F$ and a finitely presented one-ended group $H$, which finitely generated groups are quasi-isometric to the lamplighter group $F \wr H$? In the case $H=\mathbb{Z}$, it turns out that such a group must be a uniform lattice in the isometry group of a Diestel-Leader graph \cite{EFWI,EFWII} and these lattices are described in \cite{MR2991419} as \emph{cross-wired lamplighter groups}. Although a group $G$ quasi-isometric to $F \wr \mathbb{Z}$ may not be commensurable to a lamplighter group, its structure is well-understood and close to the structure of lamplighter groups over $\mathbb{Z}$. In the opposite direction, simple uniform lattices in automorphism groups of Cayley graphs of lamplighter groups over free groups are constructed in \cite{leBoudec}, exhibiting a strong lack of rigidity. Therefore, it is not clear whether it is reasonable to expect rigidity or flexibility among groups quasi-isometric to lamplighter groups in general. However, it is reasonable to expect some rigidity at least for free abelian groups:

\begin{question}
Let $n,m \geq 2$ be two integers. Which finitely generated groups are quasi-isometric to $\mathbb{Z}_n \wr \mathbb{Z}^m$?
\end{question}

\noindent
In the general one-ended case, it is not clear whether a strong rigidity, like for $\mathbb{Z}$, has to be expected. Nevertheless, strong restrictions on groups quasi-isometric to lamplighter groups over one-ended finitely presented groups can be deduced from our work. For instance:

\begin{thm}\label{thm:QItoFwrH}
Let $F$ be a non-trivial finite group, $H$ a finitely presented one-ended group, and $G$ a finitely generated group. If $G$ is quasi-isometric to $F \wr H$, then there exist finitely many subgroups $H_1, \ldots, H_n \leq G$ such that:
\begin{itemize}
	\item $H_1, \ldots, H_n$ are all quasi-isometric to $H$;
	\item the collection $\{H_1, \ldots, H_n\}$ is almost malnormal;
	\item for every finitely presented one-ended subgroup $K\leq G$, there exist $g \in G$ and $1 \leq i \leq n$ such that $K \leq gH_ig^{-1}$.
\end{itemize}
\end{thm}

\noindent
The following elementary observation will be needed in our proof:

\begin{lemma}\label{lem:commensurator}
Let $F,H$ be two groups. Then $H$ is an almost malnormal subgroup in $F \wr H$. As a consequence, if $H$ is infinite, it coincides with its commensurator. 
\end{lemma}

\begin{proof}
Fix an element $g \in F \wr H$. We can write $g$ as a product $ah$ where $a \in \bigoplus_H F$ and $h \in H$. Observe that
$$\begin{array}{lcl} aHa^{-1} \cap H & = & \{aha^{-1} \mid h \in H\} \cap H= \{ aha^{-1} h^{-1} \cdot h \mid H \} \cap H \\ \\ & = & \{ h \in H \mid aha^{-1}h^{-1}=1 \} \subset \mathrm{stab}_H(\mathrm{supp}(a)) \end{array}$$
is finite unless $\mathrm{supp}(a)= \emptyset$, i.e. $g \in H$. 
\end{proof}

\begin{proof}[Proof of Theorem \ref{thm:QItoFwrH}.]
Fix a quasi-isometry $q : G \to F\wr H$ and a quasi-inverse $\bar{q} : F\wr H \to G$. Then $G$ quasi-acts properly and cocompactly on $F \wr H$ via
$$g \mapsto \left( (c,p) \mapsto q( g \cdot \bar{q}(c,p)) \right), \ g \in G, (c,p) \in F \wr H.$$
We know from Lemma \ref{lem:commensurator} that $H$ has finite index in its commensurator, and we deduce from Theorem~\ref{thm:StructureQI} that every auto-quasi-isometry of $F\wr H$ sends every $H$-coset at (uniform) finite Hausdorff distance from another $H$-coset. It follows from \cite[Theorem 1.1]{QIsubgroup} (and its proof) that, with respect to our quasi-action $G \curvearrowright F\wr H$,
\begin{itemize}
	\item[(i)] there exist a constant $C \geq 0$ and finitely many $H$-cosets $\mathcal{H}_1, \ldots, \mathcal{H}_n$ such that, for every $H$-coset $\mathcal{H}$, there exist $g \in G$ and $1 \leq i \leq n$ such that the Hausdorff distance between $\mathcal{H}$ and $g \mathcal{H}_i$ is finite;
	\item[(ii)] for every $H$-coset $\mathcal{H}$, the quasi-stabiliser $$\mathrm{qstab}(\mathcal{H}):= \left\{ g \in G \mid \text{$\mathcal{H}$ and $g \mathcal{H}$ at finite Hausdorff distance} \right\}$$ quasi-acts cocompactly on $\mathcal{H}$.
\end{itemize}
Without loss of generality, we assume that $\mathcal{H}_1, \ldots, \mathcal{H}_n$ are pairwise distinct. For every $1 \leq i \leq n$, let $H_i$ denote the quasi-stabiliser of $\mathcal{H}_i$. It follows from $(ii)$ that $H_1, \ldots, H_n$ are all quasi-isometric to $H$. 

\medskip \noindent
Let $g \in G$ and $1 \leq i,j \leq n$ be such that $gH_ig^{-1} \cap H_j$ is infinite. Notice that $gH_ig^{-1}$ (resp. $H_j$) quasi-stabilises $g \mathcal{H}_i$ (resp. $\mathcal{H}_j$). Because the intersection between two neighbourhoods of distinct $H$-cosets must be bounded in $F \wr H$, we must have $g \mathcal{H}_i= \mathcal{H}_j$. In other words, $i=j$ and $g \in H_i$. Thus, we have proved that $\{H_1, \ldots, H_n\}$ is almost malnormal.

\medskip \noindent
Finally, let $K \leq G$ be a finitely presented one-ended subgroup. If we denote by $\iota : K \hookrightarrow G$ the inclusion, then $q \circ \iota$ defines a coarse embedding $K \to F \wr H$. According to Theorem~\ref{thm:FullEmbeddingThm}, $q(\iota(K))$ lies in the neighbourhood of an $H$-coset $\mathcal{H}$. As a consequence, for every $k \in K$, $k \cdot \mathcal{H}$ and $\mathcal{H}$ both contain $q(\iota(K))$ in a neighbourhood. But $k \cdot \mathcal{H}$ lies at finite Hausdorff distance from an $H$-coset, and again the intersection between two neighbourhoods of two distinct $H$-cosets is bounded, so $k$ has to quasi-stabilise $\mathcal{H}$. Therefore, $K \leq g H_ig^{-1}$ where $g \in G$ and $1 \leq i \leq n$ are given by (i), i.e. the Hausdorff distance between $\mathcal{H}$ and $g \mathcal{H}_i$ is finite.
\end{proof}

\noindent
As an illustration of Theorem \ref{thm:QItoFwrH}, let us prove the following observation (which, together with Corollary \ref{cor:PermutationalW}, characterises which finitely generated permutational wreath products between finite and finitely presented one-ended groups are quasi-isometric to standard wreath products between finite and finitely presented one-ended groups):

\begin{cor}\label{cor:PermutationW}
Let $F_1,F_2$ be two non-trivial finite groups, $H_1,H_2$ two finitely presented one-ended groups, and $X_1,X_2$ two sets on which $H_1,H_2$ respectively act with finitely many orbits. Assume that $H_1$ acts on $X_1$ with finite stabilisers. If $F_1 \wr_{X_1} H_1$ and $F_2 \wr_{X_2} H_2$ are quasi-isometric, then $H_2$ acts on $X_2$ with finite stabilisers.
\end{cor}

\begin{proof}
Let $\mathcal{P}$ denote the collection of the conjugates of the subgroups in $F_2 \wr_{X_2} H_2$ provided by Theorem \ref{thm:QItoFwrH}. As a consequence, there exists some $P \in \mathcal{P}$ such that $H_2 \leq P$. Assume that there exists a point $x \in X_2$ with infinite stabiliser in $H_2$. We claim that $\bigoplus_{H_2 \cdot x} F_2 \leq P$. 

\medskip \noindent
We fix an $a\in \bigoplus_{H_2 \cdot x} F_2$ and we argue by induction over the size of $\mathrm{supp}(a)$. If $\mathrm{supp}(a)$ is empty, there is nothing to prove. Otherwise, we can write $a$ as a product $bc$ where $b,c \in \bigoplus_{H_2 \cdot x} F_2$ are such that the support of $c$ has size the support of $a$ minus one and such that the support of $b$ has size one. We know from our induction hypothesis that $c \in P$. Moreover, we have
$$P \cap bPb^{-1} \supset H_2 \cap  bH_2b^{-1} \supset \mathrm{stab}_{H_2}(\mathrm{supp}(b))$$
where $\mathrm{stab}_{H_2}(\mathrm{supp}(b))$ is infinite since $\mathrm{supp}(b)$ contains only an $H_2$-translate of $x$. Because $\mathcal{P}$ is almost malnormal, necessarily $b \in P$. There, $a=bc \in P$ as desired.

\medskip \noindent
Therefore, $P$ contains the infinite normal subgroup $N:= \langle \bigoplus_{H_2 \cdot x} F_2 , H_2 \rangle \lhd F_2 \wr_{X_2} H_2$. For every $g \in F_2 \wr_{X_2} H_2$, we must have $P \cap gPg^{-1} \supset N$. And, because $\mathcal{P}$ is almost malnormal, this implies that $g \in P$. In other words, we have proved that $P = F_2 \wr_{X_2} H_2$. We get a contradiction since $F_2 \wr_{X_2} H_2$ is not quasi-isometric to $H_1$: otherwise, $H_1$ would be quasi-isometric to $F_1 \wr_{X_1} H_1$, which is impossible since $F_1 \wr_{X_1} H_1$ is not finitely presented. 
\end{proof}

\begin{remark}
Using Remark \ref{rem:locallycompact}, we deduce that a version of Corollary \ref{cor:PermutationW} holds (with the same proof) under the more general assumption that $H_1$ and $H_2$ are locally compact, compactly presented, one-ended, and that their actions on $X_1$ and $X_2$ have open stabilisers. Assuming that the stabilisers of $H_1$ are compact, we conclude that the same holds for $H_2$.  \end{remark}

\paragraph{Further applications.} Finally, let us mention that, although we mainly focused in this article on wreath products of the form $F \wr H$ where $F$ is a finite group and $H$ a finitely presented one-ended group, the techniques we developed can be generalised in several directions. 

\medskip \noindent
For instance, the embedding theorem provided by Theorem~\ref{thm:FullEmbeddingThm} can be proved for some wreath products $F \wr H$ where $F$ is infinite, providing interesting information about the quasi-isometries of $F \wr H$ when $H$ cannot be coarsely separated by a subspace of $F \times \cdots \times F$. We expect to investigate this subject in the future. 

\medskip \noindent
As another application, it is worth noticing that other infinitely presented groups can be studied in a similar way. For instance, given a group $G$, let $\mathscr{H}(G)$ be the semidirect product $\mathrm{Bij}_c(G) \rtimes G$ where $\mathrm{Bij}_c(G)$ denotes the set of all the finitely supported bijections $G \to G$ and where $G$ acts on $\mathrm{Bij}_c(G)$ by precomposition. Observe that $\mathscr{H}(\mathbb{Z})$ coincides with the second Houghton group $H_2$. In the same way that truncating a presentation of $\mathbb{Z}_2 \wr H$ leads to a semidirect product $C(\Gamma) \rtimes H$ for some right-angled Coxeter group $C(\Gamma)$, truncating a presentation of $\mathscr{H}(G)$ leads to a semidirect product $C(\Gamma) \rtimes G$ for some Coxeter group $C(\Gamma)$ (which is not right-angled). Following the construction sketched in the introduction, $C(\Gamma) \rtimes G$ can be thought of geometrically as a pointed edge moving by elementary moves in the canonical Cayley graph of $C(\Gamma)$. Since the latter graph has a natural wallspace structure, most of the arguments from Section~\ref{section:EmbeddingProof} apply, leading to a embedding theorem similar to Theorem~\ref{thm:FullEmbeddingThm}.

\addcontentsline{toc}{section}{References}

\bibliographystyle{alpha}
{\footnotesize\bibliography{AsGeomLamp}}

\def\polhk#1{\setbox0=\hbox{#1}{\ooalign{\hidewidth
  \lower1.5ex\hbox{`}\hidewidth\crcr\unhbox0}}}
\begin{thebibliography}{MPJSS20}

\bibitem[AGS06]{MR2271228}
G.~Arzhantseva, V.~Guba, and M.~Sapir.
\newblock Metrics on diagram groups and uniform embeddings in a {H}ilbert
  space.
\newblock {\em Comment. Math. Helv.}, 81(4):911--929, 2006.

\bibitem[BK98]{MR1616135}
D.~Burago and B.~Kleiner.
\newblock Separated nets in {E}uclidean space and {J}acobians of bi-{L}ipschitz
  maps.
\newblock {\em Geom. Funct. Anal.}, 8(2):273--282, 1998.

\bibitem[BMW94]{quasimedian}
H.-J. Bandelt, H.M. Mulder, and E.~Wilkeit.
\newblock Quasi-median graphs and algebras.
\newblock {\em J. Graph Theory}, 18(7):681--703, 1994.

\bibitem[Bog96]{Bogop}
O.~Bogopolski.
\newblock Infinite commensurable hyperbolic groups are bi-{L}ipschitz
  equivalent.
\newblock {\em unpublished preprint}, 1996.

\bibitem[CG86]{MR837621}
J.~Cheeger and M.~Gromov.
\newblock {$L_2$}-cohomology and group cohomology.
\newblock {\em Topology}, 25(2):189--215, 1986.

\bibitem[CSV08]{MR2393636}
Y.~Cornulier, Y.~Stalder, and A.~Valette.
\newblock Proper actions of lamplighter groups associated with free groups.
\newblock {\em C. R. Math. Acad. Sci. Paris}, 346(3-4):173--176, 2008.

\bibitem[CSV12]{MR2888241}
Y.~Cornulier, Y.~Stalder, and A.~Valette.
\newblock Proper actions of wreath products and generalizations.
\newblock {\em Trans. Amer. Math. Soc.}, 364(6):3159--3184, 2012.

\bibitem[dCFK12]{MR2991419}
Y.~de~Cornulier, D.~Fisher, and N.~Kashyap.
\newblock Cross-wired lamplighter groups.
\newblock {\em New York J. Math.}, 18:667--677, 2012.

\bibitem[dlH00]{MR1786869}
P.~de~la Harpe.
\newblock {\em Topics in geometric group theory}.
\newblock Chicago Lectures in Mathematics. University of Chicago Press,
  Chicago, IL, 2000.

\bibitem[DO11]{MR2811580}
T.~Davis and A.~Olshanskii.
\newblock Subgroup distortion in wreath products of cyclic groups.
\newblock {\em J. Pure Appl. Algebra}, 215(12):2987--3004, 2011.

\bibitem[DPT15]{MR3318424}
T.~Dymarz, I.~Peng, and J.~Taback.
\newblock Bilipschitz versus quasi-isometric equivalence for higher rank
  lamplighter groups.
\newblock {\em New York J. Math.}, 21:129--150, 2015.

\bibitem[Dym05]{MR2139686}
T.~Dymarz.
\newblock Bijective quasi-isometries of amenable groups.
\newblock In {\em Geometric methods in group theory}, volume 372 of {\em
  Contemp. Math.}, pages 181--188. Amer. Math. Soc., Providence, RI, 2005.

\bibitem[Dym10]{MR2730576}
T.~Dymarz.
\newblock Bilipschitz equivalence is not equivalent to quasi-isometric
  equivalence for finitely generated groups.
\newblock {\em Duke Math. J.}, 154(3):509--526, 2010.

\bibitem[Dyu99]{MR1708557}
A.~Dyubina.
\newblock Characteristics of random walks on the wreath products of groups.
\newblock {\em Zap. Nauchn. Sem. S.-Peterburg. Otdel. Mat. Inst. Steklov.
  (POMI)}, 256(Teor. Predst. Din. Sist. Komb. i Algoritm. Metody. 3):31--37,
  264, 1999.

\bibitem[Dyu00]{MR1800990}
A.~Dyubina.
\newblock Instability of the virtual solvability and the property of being
  virtually torsion-free for quasi-isometric groups.
\newblock {\em Internat. Math. Res. Notices}, (21):1097--1101, 2000.

\bibitem[Eck92]{MR1171301}
B.~Eckmann.
\newblock Amenable groups and {E}uler characteristic.
\newblock {\em Comment. Math. Helv.}, 67(3):383--393, 1992.

\bibitem[EFW12]{EFWI}
A.~Eskin, D.~Fisher, and K.~Whyte.
\newblock Coarse differentiation of quasi-isometries {I}: {S}paces not
  quasi-isometric to {C}ayley graphs.
\newblock {\em Ann. of Math. (2)}, 176(1):221--260, 2012.

\bibitem[EFW13]{EFWII}
A.~Eskin, D.~Fisher, and K.~Whyte.
\newblock Coarse differentiation of quasi-isometries {II}: {R}igidity for {S}ol
  and lamplighter groups.
\newblock {\em Ann. of Math. (2)}, 177(3):869--910, 2013.

\bibitem[Ers03]{MR2011120}
A.~Erschler.
\newblock On isoperimetric profiles of finitely generated groups.
\newblock {\em Geom. Dedicata}, 100:157--171, 2003.

\bibitem[FM98]{MR1608595}
B.~Farb and L.~Mosher.
\newblock A rigidity theorem for the solvable {B}aumslag-{S}olitar groups.
\newblock {\em Invent. Math.}, 131(2):419--451, 1998.
\newblock With an appendix by Daryl Cooper.

\bibitem[FM99]{MR1709862}
B.~Farb and L.~Mosher.
\newblock Quasi-isometric rigidity for the solvable {B}aumslag-{S}olitar
  groups. {II}.
\newblock {\em Invent. Math.}, 137(3):613--649, 1999.

\bibitem[FM00]{MR1768110}
B.~Farb and L.~Mosher.
\newblock On the asymptotic geometry of abelian-by-cyclic groups.
\newblock {\em Acta Math.}, 184(2):145--202, 2000.

\bibitem[Gen17a]{QM}
A.~Genevois.
\newblock Cubical-like geometry of quasi-median graphs and applications in
  geometry group theory.
\newblock {\em PhD thesis, arXiv:1712.01618}, 2017.

\bibitem[Gen17b]{LampMedian}
A.~Genevois.
\newblock Lamplighter groups, median spaces, and a-{T}-menability.
\newblock {\em arXiv:1705.00834}, 2017.

\bibitem[Gen19]{MR4033512}
A.~Genevois.
\newblock Embeddings into {T}hompson's groups from quasi-median geometry.
\newblock {\em Groups Geom. Dyn.}, 13(4):1457--1510, 2019.

\bibitem[GLSZ00]{MR1797748}
R.~Grigorchuk, P.~Linnell, T.~Schick, and A.~\.{Z}uk.
\newblock On a question of {A}tiyah.
\newblock {\em C. R. Acad. Sci. Paris S\'{e}r. I Math.}, 331(9):663--668, 2000.

\bibitem[GM19]{AutQM}
A.~Genevois and A.~Martin.
\newblock Automorphisms of graph products of groups from a geometric
  perspective.
\newblock {\em Proc. London Math. Soc.}, 119(6):1745--1779, 2019.

\bibitem[Gro93]{GromovAsymptotic}
M.~Gromov.
\newblock Asymptotic invariants of infinite groups.
\newblock In {\em Geometric group theory, {V}ol.\ 2 ({S}ussex, 1991)}, volume
  182 of {\em London Math. Soc. Lecture Note Ser.}, pages 1--295. Cambridge
  Univ. Press, Cambridge, 1993.

\bibitem[GT21]{Kappa}
A.~Genevois and R.~Tessera.
\newblock Scaling quasi-isometries.
\newblock {\em preprint}, 2021.

\bibitem[LB20]{leBoudec}
A.~Le~Boudec.
\newblock Simple groups and irreducible lattices in wreath products.
\newblock {\em arxiv:2001.08689, to appear in Ergodic Theory and Dynamical
  Systems}, 2020.

\bibitem[Li10]{MR2644886}
S.~Li.
\newblock Compression bounds for wreath products.
\newblock {\em Proc. Amer. Math. Soc.}, 138(8):2701--2714, 2010.

\bibitem[McM98]{MR1616159}
C.~McMullen.
\newblock Lipschitz maps and nets in {E}uclidean space.
\newblock {\em Geom. Funct. Anal.}, 8(2):304--314, 1998.

\bibitem[MO10]{MR2557962}
N.~Monod and N.~Ozawa.
\newblock The {D}ixmier problem, lamplighters and {B}urnside groups.
\newblock {\em J. Funct. Anal.}, 258(1):255--259, 2010.

\bibitem[MPJSS20]{QIsubgroup}
E.~Mart\'inez-Pedroza and J.~Jorge S\'anchez~Salda{\~n}a.
\newblock Quasi-isometric rigidity of subgroups and filtered ends.
\newblock {\em arxiv:2012.10494}, 2020.

\bibitem[Nek98]{MR1693847}
V.~Nekrashevych.
\newblock Quasi-isometric hyperbolic groups are bi-{L}ipschitz equivalent.
\newblock {\em Dopov. Nats. Akad. Nauk Ukr. Mat. Prirodozn. Tekh. Nauki},
  (1):32--35, 1998.

\bibitem[NP11]{MR2783928}
A.~Naor and Y.~Peres.
\newblock {$L_p$} compression, traveling salesmen, and stable walks.
\newblock {\em Duke Math. J.}, 157(1):53--108, 2011.

\bibitem[Pap95]{MR1326733}
P.~Papasoglu.
\newblock Homogeneous trees are bi-{L}ipschitz equivalent.
\newblock {\em Geom. Dedicata}, 54(3):301--306, 1995.

\bibitem[Par92]{MR1062874}
W.~Parry.
\newblock Growth series of some wreath products.
\newblock {\em Trans. Amer. Math. Soc.}, 331(2):751--759, 1992.

\bibitem[Pen11a]{MR2860983}
I.~Peng.
\newblock Coarse differentiation and quasi-isometries of a class of solvable
  {L}ie groups {I}.
\newblock {\em Geom. Topol.}, 15(4):1883--1925, 2011.

\bibitem[Pen11b]{MR2860984}
I.~Peng.
\newblock Coarse differentiation and quasi-isometries of a class of solvable
  {L}ie groups {II}.
\newblock {\em Geom. Topol.}, 15(4):1927--1981, 2011.

\bibitem[PSC02]{MR1905862}
C.~Pittet and L.~Saloff-Coste.
\newblock On random walks on wreath products.
\newblock {\em Ann. Probab.}, 30(2):948--977, 2002.

\bibitem[PW02]{MR1898396}
P.~Papasoglu and K.~Whyte.
\newblock Quasi-isometries between groups with infinitely many ends.
\newblock {\em Comment. Math. Helv.}, 77(1):133--144, 2002.

\bibitem[SV07]{MR2318545}
Y.~Stalder and A.~Valette.
\newblock Wreath products with the integers, proper actions and {H}ilbert space
  compression.
\newblock {\em Geom. Dedicata}, 124:199--211, 2007.

\bibitem[Var83]{MR732356}
N.~Varopoulos.
\newblock Random walks on soluble groups.
\newblock {\em Bull. Sci. Math. (2)}, 107(4):337--344, 1983.

\bibitem[Why99]{MR1700742}
K.~Whyte.
\newblock Amenability, bi-{L}ipschitz equivalence, and the von {N}eumann
  conjecture.
\newblock {\em Duke Math. J.}, 99(1):93--112, 1999.

\end{thebibliography}

\Address

\end{document}